\documentclass[reqno,11pt]{amsart}
\usepackage{amsmath,amssymb,mathrsfs,amsthm,amsfonts,comment}
\usepackage[singlelinecheck=true]{caption}
\usepackage[inline]{enumitem} 
\usepackage[usenames,dvipsnames]{xcolor}
\usepackage{hyperref}
\usepackage{bbm}
\hypersetup{%
	colorlinks=true, linkcolor=blue,
	citecolor=ForestGreen
}
\usepackage[paper=letterpaper,margin=1in]{geometry}
\usepackage{tikz}
\usetikzlibrary{matrix,graphs,arrows,positioning, shapes, calc,decorations.markings,decorations.pathmorphing,shapes.symbols,hobby}

\newtheorem{theorem}{Theorem}[section]
\newtheorem{lemma}[theorem]{Lemma}
\newtheorem{proposition}[theorem]{Proposition}
\newtheorem{corollary}[theorem]{Corollary}
\theoremstyle{definition}
\newtheorem{definition}[theorem]{Definition}
\newtheorem{remark}[theorem]{Remark}
\numberwithin{equation}{section}
\allowdisplaybreaks
\usepackage{acronym}

\acrodef{KPZ}{Kardar--Parisi--Zhang}
\acrodef{SHE}{Stochastic Heat Equation}
\acrodef{LDP}{Large Deviation Principle}

\newcommand{\argmax}{\mathop{\mathrm{argmax}}} 
 
\renewcommand{\Pr}{\mathbb{P}}	
\newcommand{\Ex}{\mathbb{E}}	
\renewcommand{\d}{\mathrm{d}}	
\newcommand{\ind}{\mathbf{1}}	

\newcommand{\md}{\mathcal{M}}
\newcommand{\mx}{\mathcal{P}}
\newcommand{\W}{W}
\newcommand{\U}{U}
\newcommand{\nonintbb}{\mathsf{NonInt}\mbox{-}\mathsf{BrBridge}}
\newcommand{\nonintbm}{\mathsf{NonInt}\mbox{-}\mathsf{BM}}

\newcommand{\cdrp}{\mathsf{CDRP}}

\newcommand{\m}{\mathsf}
\newcommand{\wl}{W_{\operatorname{\ell}}}
\renewcommand{\wr}{W_{\operatorname{r}}}
\newcommand{\dbm}{\mathsf{DBM}}
\newcommand{\bb}{\mathcal{R}_{\operatorname{bb}}}
\newcommand{\Nice}{\mathsf{Nice}}

\newcommand{\hk}{p}					
\newcommand{\sd}{\textcolor{blue}}

\newcommand{\bme}{\mathfrak{B}_{\operatorname{me}}}

\newcommand{\Con}{\mathrm{C}} 

\newcommand{\N}{\mathbb{N}}
\newcommand{\R}{\mathbb{R}} 
\newcommand{\Z}{\mathbb{Z}} 

\newcommand{\e}{\varepsilon}

\newcommand{\g}{\Lambda}
\newcommand{\h}{\mathfrak{h}}

\newcommand{\calA}{\mathcal{A}}

\newcommand{\calD}{\mathcal{D}}

\newcommand{\calH}{\mathcal{H}}

\newcommand{\calM}{\mathcal{M}}

\newcommand{\calL}{\mathcal{L}}
\newcommand{\calR}{\mathcal{R}}
\newcommand{\calZ}{\mathcal{Z}}
\newcommand{\calF}{\mathcal{F}}
\newcommand{\B}{B}
\newcommand{\V}{V}

\newcommand{\til}{\widetilde}
\renewcommand{\bar}{\overline}

\usepackage{graphicx}

\title[Localization of the CDRP]{Localization of the continuum directed random polymer}

\author[S.\ Das]{Sayan Das}
\address{S.\ Das,
	Department of Mathematics, Columbia University,
	\newline\hphantom{\quad \ \ S. Das}
	2990 Broadway, New York, NY 10027 USA
}
\email{sayan.das@columbia.edu}

\author[W.\ Zhu]{Weitao Zhu}
\address{W.\ Zhu,
	Department of Mathematics, Columbia University,
	\newline\hphantom{\quad \ \ S. Das}
	2990 Broadway, New York, NY 10027 USA
}
\email{weitao.zhu@columbia.edu}

\begin{document}
	
	\begin{abstract} We consider the continuum directed random polymer (CDRP) model that arises as a scaling limit from $1+1$ dimensional directed polymers in the intermediate disorder regime. We show that for a point-to-point polymer of length $t$ and any $p\in (0,1)$, the quenched density of {the point on the path which is $pt$ distance away from the origin} when centered around its random mode $\md_{p,t}$ converges in law to an explicit random density function as $t\to\infty$ without any scaling. Similarly in the case of point-to-line polymers of length $t$, the quenched density of the endpoint of the path when centered around its random mode $\md_{*,t}$ converges in law to an explicit random density. The limiting random densities are proportional to $e^{-\calR_\sigma(x)}$ where $\calR_\sigma(x)$ is a two-sided {3D} Bessel process with appropriate diffusion coefficient $\sigma$.  In addition, the laws of the random modes $\md_{*,t}$, $\md_{p,t}$ themselves converge in distribution upon $t^{2/3}$ scaling to the maximizer of $\operatorname{Airy}_2$ process {minus a parabola} and points on the geodesics of the directed landscape respectively.
		
		Our localization results stated above provide an affirmative case of the folklore “favorite region” conjecture. {Our proof techniques also allow us to prove properties of the KPZ equation such as ergodicity and limiting Bessel behaviors around the maximum.}
	\end{abstract}

	\maketitle
	{
		\hypersetup{linkcolor=black}
		\setcounter{tocdepth}{1}
		\tableofcontents
	}	




	\section{Introduction}

	{The continuum directed random polymer ($\cdrp$) is a continuum version of the discrete directed polymer measures modeled by a path interacting with a space-time white noise that first appeared in \cite{akq}. It arises as a scaling limit of the 1+1 dimensional directed polymers in the ``intermediate disorder regime'' and can be defined through the Kardar-Parisi-Zhang (KPZ) equation with narrow wedge initial data (see Section \ref{sec:cdrp}). A folklore favorite region conjecture on directed polymers states that under strong disorder, the midpoint (or any other point) distribution of a point-to-point directed polymer is asymptotically localized in a region of stochastically bounded diameter ( \cite{cha}, \cite{batesch}, Section \ref{sec:intro}).}
	
	In light of this conjecture, we initiate such study of the $\cdrp$'s long-time localization behaviors. Our main result, stated in Section \ref{sec:cdrp}, asserts that any point at a fixed proportional location on the point-to-point $\cdrp$ relative to its length converges to an explicit density function when centered around its almost surely unique random mode. 
	A similar result for the endpoint of point-to-line $\cdrp$ is also obtained, confirming the favorite region conjecture for the $\cdrp.$ In this process, through the connections between the $\cdrp$ and the KPZ equation with narrow wedge initial data, we have shown properties such as ergodicity and Bessel behaviors around the maximum for the latter. These and other results are summarized in Section \ref{sec:results} and explained in fuller detail in Section \ref{sec:cdrp}.
	
 {As an effort to understand the broader localization phenomena}, our main theorems (Theorems \ref{t:main}, \ref{t:main2}) confirm the favorite region conjecture for the \textit{first} non-stationary integrable model and are the \textit{first} to obtain pointwise localization along the \textit{entire path} (Theorem \ref{t:main}). The first rigorous localization result for directed polymers in random environment appeared in \cite{carmonaHu1} and proved the existence of ``favorite sites" in the Gaussian environment, which was later extended to general environments in \cite{cometsy}. This notion of localization is known as the \textit{strong localization} and is weaker than the favorite region conjecture (See Section \ref{sec:intro} for discussions on different notions of localizations). The only other case where the favorite region conjecture is proved, is the mid and endpoint localizations of the point-to-point and point-to-line one-dimensional stationary log-gamma polymers in \cite{lgamma}. The proofs of the results and the specificity of the locales (mid/endpoints) in \cite{lgamma} relied on the stationary boundary condition of the model, which reduced the endpoint distribution to exponents of simple random walks \cite{sep}. For $\cdrp$, the absence of a similar stationarity necessitates a new approach towards the favorite region conjecture, which extends to every point on the polymer's path. Conversely, as we do not rely on integrability other than the Gibbs property, our proof for the $\cdrp$ has the potential to generalize to other integrable models. Other works that have considered localization along the whole path include the pathwise localization of the parabolic Anderson polymer model in \cite{comets2013overlaps} and that of the discrete polymer in Gaussian environments \cite{bates2020localization} \cite{bates2021full}. Lastly, accompanying our localization results, we establish the convergence of the scaled favorite points to the almost sure unique maximizer of the $\operatorname{Airy}_2$ process minus a parabola and the geodesics of the directed landscape respectively (Theorem \ref{t:favpt}).
	
	 {Finally, from the perspective of the KPZ universality class, our paper is an innovative application of several fundamental new techniques and results that have recently emerged in the community.} These include the Brownian Gibbs resampling property \cite{CH16}, the weak convergence from the KPZ line ensemble to the Airy line ensemble
	\cite{qs20}, the tail estimates of the KPZ equation with narrow wedge initial data \cite{utail,ltail,cgh} as well as probabilistic properties of the Airy line ensemble from \cite{dv18}. In particular, although the Gibbs property has been utilized before in works such as \cite{dv18,chh,cgh,chhm}, we overcome a unique challenge of quantifying the Gibbs property precisely on a \textit{symmetric} random interval around the joint local maximizer of two independent copies of the KPZ equation with narrow wedge initial data. This issue is resolved after we  prescribe the joint law of the KPZ equations around the desired interval. 
	A more detailed description of our main technical innovations is available in Section \ref{sec:pf}.
	
	Presently, we begin with the background of our model and related key concepts.
	
	\subsection{Introducing the $\cdrp$ through discrete directed lattice polymers}\label{sec:intro}
	 {Directed polymers in random environments were first introduced in statistical physics literature by  Huse and Henley \cite{huse} to study the phase
	boundary of the Ising model with random impurities. Later, it was mathematically reformulated as a random walk in a random environment by Imbrie and Spencer \cite{imb} and Bolthausen \cite{bol}. Since then immense progress has been made in understanding this model (see 
	 \cite{comets} for a general introduction and \cite{gia,batesch} for partial surveys). }
	
	In the $(d+1)$- dimensional discrete polymer case, the random environment is specified by a collection of zero-mean i.i.d.~random variables $\{\omega=\omega(i,j) \mid (i,j)\in \Z_{+}\times \Z^d\}$. Given the environment, the energy of the $n$-step nearest neighbour random walk $(S_i)_{i=0}^n$ starting and ending at the origin (one can take the endpoint to be any suitable $\mathbf{x}\in \R^d$ as well) is given by 
	$H_n^{\omega}(S):=\sum_{i=1}^n \omega(i,S_i).$
	The \textbf{point-to-point} polymer measure on the set of all such paths is then defined as 
	\begin{align}\label{eq:disGibbs}
		\Pr_{n,\beta}^{\omega}(S)=\frac1{Z_{n,\beta}^{\omega}} e^{\beta H_n^{\omega}(S)} \Pr(S),
	\end{align}
	where $\Pr(S)$ is the uniform measure on set of all $n$-step nearest neighbour paths starting and ending at origin, $\beta$ is the inverse temperature, and $Z_{n,\beta}^{\omega}$ is the partition function. Meanwhile, one can also consider the \textbf{point-to-line} polymer measures where the endpoint is `free' and the reference measure $\Pr$ is given by $n$-step simple symmetric random walks. In the polymer measure, there is a competition between the \textit{entropy} of paths and the \textit{disorder strength} of the environment. Under this competition, two distinct regimes appear depending on the inverse temperature $\beta$ \cite{comy}:
	
	\begin{itemize}
		\item \textit{Weak Disorder}: When $\beta$ is small or equivalently  in high temperature regime, intuitively the disorder strength diminishes. 
		The walk is dominated by the entropy and exhibits diffusive behaviors. This type of entropy domination is termed as \textit{weak disorder}.
		
		\item  \textit{Strong Disorder}: 
		If $\beta$ is large and positive or equivalently the temperature is low but remains positive, the polymer measure concentrates on singular paths with high energies and the diffusive behavior is no longer guaranteed. This type of disorder strength domination is known as the \textit{strong disorder}.
	\end{itemize}
 The precise definitions of weak and strong disorder regimes are available in \cite{comy}. Furthermore, there exists a critical inverse temperature $\beta_c(d)$ that depends on the dimension $d$ such that weak disorder holds for $0\le \beta<\beta_c$  and strong disorder for $\beta>\beta_c$. When $d=1$ or $d=2$, $\beta_c=0$, i.e. all positive $\beta$ fall into the strong disorder regime.

	The rest of the article focuses on $d=1$. While for $\beta=0$, the path fluctuations are of the order $\sqrt{n}$ via Brownian considerations, the situation is much more complex in the strong disorder regime. The following two phenomena are conjectured:
	
	\begin{itemize}
		\item \textit{Superdiffusivity}: Under strong disorder, the polymer measure is believed to be in the KPZ universality class and paths have typical fluctuations of the order $n^{2/3}$. This widely conjectured phenomenon is physics literature is known as superdiffusion (see  \cite{huse}, \cite{hhf}, \cite{kardar1986dynamic}, \cite{ks2}) and has been rigorously proven in specific situations (see \cite{lnp},\cite{mp}, \cite{joh},\cite{carmonaHu1}, \cite{om}). But much remains unknown especially for $d \ge 2.$ 
		\item \textit{Localization and the favorite region conjecture}: The polymer exhibits certain localization phenomena under strong disorder. 
		{The favorite region conjecture speculates that any point on the path of a point-to-point directed polymer is asymptotically localized in a region of stochastically bounded diameter} (see \cite{cha}, \cite{bates} for partial survey.)
	\end{itemize}
	
We remark that there exist many different notions of localizations. In addition to the favorite region one discussed above and the strong localization in \cite{carmonaHu1} mentioned earlier, the atomic localization \cite{vargas} and the geometric localization \cite{batesch} were studied in \cite{batesch} for simple random walks and later extended to general reference walks in \cite{bates}. Both of \cite{batesch}  and \cite{bates} provided sufficient criteria for the existence of  the `favorite region' of order one for the endpoint in arbitrary dimension. Yet in spite of the sufficiency, it is unknown how to check them for standard directed polymers.  We refer the readers to Bates' thesis \cite{batesphd} for a more detailed survey on this topic.

	Meanwhile, even though the critical inverse temperature $\beta_c(1)=0$ for $d= 1$, one might scale the inverse temperature with the length of the polymer critically to capture the transition between weak and strong disorder. In this spirit, the seminal work of \cite{akq2} considered \textit{an intermediate disordered regime} where $\beta=\beta_n = n^{-1/4}$ and $n$ is the length of the polymer. \cite{akq2} showed that the partition function $Z_{n,\beta_n}^{\omega}$ has a universal scaling limit given by the solution of the Stochastic Heat Equation (SHE) for $\omega$ with finite exponential moments. Furthermore, under the diffusive scaling, the polymer path itself converges to a universal object called the Continuous Directed Random Polymer (denoted as $\cdrp$ hereafter) which appeared first in \cite{akq} and depended on a continuum external environment given by the space-time white noise. 
	
	More precisely, given a white noise $\xi$ on $[0,t]\times \R$, $\cdrp$ is a path measure on the space of $C([0,t])$ (continuous functions on $[0,t]$) for each realization of $\xi$. Conditioned on the environment, the $\cdrp$ is a continuous Markov process with the same quadratic variation as the Brownian motion but is singular w.r.t.~the Brownian motion (\cite{akq2}). Due to this singularity w.r.t.~the Wiener measure, expressing the $\cdrp$ path measure in a Gibbsian form similar to \eqref{eq:disGibbs} is challenging. Instead, one can construct a consistent family of finite dimensional distributions using the partition functions which uniquely specify the path measure (see \cite{akq} or Section \ref{sec:cdrp}). 
	
	As the $\cdrp$ sits between weak and strong disorder regimes, it exhibits weak disorder type behaviors in the short-time regime $(t\downarrow 0)$ and strong ones in the long-time regime $(t \uparrow \infty)$. Indeed, the log partition function of $\cdrp$ is Gaussian in the short time limit (see \cite{acq}), which provides evidence for weak disorder. Upon varying the endpoint of the $\cdrp$ measure, the log partition function becomes a random function of the endpoint and converges to the parabolic $\operatorname{Airy}_2$ process under the $1:2:3$ KPZ scaling (see \cite{qs20,vir20}) with the characteristic $2/3$ exponent. This alludes to the superdiffusivity in the strong disorder regime. That said, the theory of universality class alone does not shed insight on the possible localization phenomena of the $\cdrp$ measures.

	\subsection{Summary of Results} \label{sec:results} The purpose of the present article is to study the localization phenomena for the long-time $\cdrp$ measure. The following summarizes our results, which we elaborate on individually in Section \ref{sec:cdrp}.
	Our first two results affirm the favorite region conjecture which so far has only been proven for the mid and endpoints of the log-gamma polymer model in \cite{lgamma}.
	\begin{itemize}
		\item For a point-to-point $\cdrp$ of length $t$,  the quenched density of its $pt$-point with fixed $p\in (0,1)$ when centered around its {almost sure unique }mode (which is the maximizer of the probability density function) $\md_{p,t}$,  
		converges weakly to a density proportional to $e^{-\calR_2(x)}$. Here, $\calR_2$ is a two-sided 3D-Bessel process with diffusion coefficient $2$ defined in \eqref{def:bessel}(Theorem \ref{t:main}). 
		
		\item For a point-to-line $\cdrp$ of length $t$, the quenched density of its endpoint when centered around its almost sure unique mode $\md_{*,t}$ converges weakly to a density proportional to $e^{-\calR_1(x)}$. $\calR_1$ is a two-sided 3D-Bessel process with diffusion coefficient $1$ (Theorem \ref{t:main2}). 
	\end{itemize}

	\begin{itemize}
		\item The random mode $\md_{*,t}$ of length-$t$ point-to-line $\cdrp$'s endpoint upon {$2^{-1/3}t^{2/3}$} scaling converges in law to the unique maximum of the $\operatorname{Airy}_2$ process minus a parabola; the random mode $\md_{p,t}$ of the $pt$ point of point-to-point $\cdrp$ of length $t$ upon $t^{2/3}$ scaling converges to $\Gamma(p\sqrt{2})$, the Directed landscape's geodesic from $(0,0)$ to $(0,p\sqrt{2})$ (Theorem \ref{t:favpt}). 
	\end{itemize}
	
	Next, the well-known KPZ equation {with narrow wedge initial data} forms the log-partition function of the $\cdrp$. Our main results below shed light on some of its local information:
	\begin{itemize}
		\item \textit{Ergodicity}: The spatial increment of the KPZ equation {with narrow wedge initial data} as $t \rightarrow \infty$ converges weakly to a standard two-sided Brownian motion (Theorem \ref{t:ergodic}).
		
		\item The sum of two independent copies of the KPZ equation {with narrow wedge initial data} when re-centered around its maximum converges to a two-sided 3D-Bessel process with diffusion coefficient $2$ (Theorem \ref{t:bessel}).
	\end{itemize}
	
	These results provide a comprehensive characterization of the localization picture for the $\cdrp$ model. We present the formal statements of the results in the next subsection.

	\subsection{The model and the main results} \label{sec:cdrp}  In order to define the $\cdrp$ model we use the stochastic heat equation (SHE) with multiplicative noise as our building blocks. Namely, we consider a four-parameter random field $\calZ(x,s;y,t)$ defined on 
	\begin{align*}
		\R_{\uparrow}^4:= \{(x,s;y,t)\in \R^4 : s<t\}.
	\end{align*}
	For each $(x,s)\in \R\times \R$, the field $(y,t)\mapsto \calZ(x,s;y,t)$ is the solution of the SHE starting from location $x$ at time $s$, i.e., the unique solution of
	\begin{align*}
		\partial_t\calZ=\tfrac12\partial_{xx}\calZ+\calZ\cdot \xi, \qquad (y,t) \in \R \times (s,\infty),
	\end{align*}
	with Dirac delta initial data $\lim_{t\downarrow s}\calZ(x,s;y,t)=\delta(x-y).$ Here $\xi=\xi(x,t)$ is the space-time white noise. The SHE itself enjoys a well-developed solution theory based on It\^o integral and chaos expansion \cite{bertini1995stochastic,walsh1986introduction} also \cite{corwin2012kardar,quastel2011introduction}. 
	Via the Feynmann-Kac formula (\cite{hhf, comets}) the four-parameter random field can be written in terms of chaos expansion as
	\begin{align}\label{eq:chaos}
		\calZ(x,s;y,t) = \sum_{k=0}^{\infty}  \int_{\Delta_{k,s,t}} \int_{\R^k} \prod_{\ell=1}^{k+1} \hk(y_{\ell}-y_{\ell-1},s_{\ell}-s_{\ell-1}) \xi(y_{\ell},s_{\ell})  d\vec{y} \,d \vec{s},
	\end{align}
	with $\Delta_{k,s,t}:=\{(s_\ell)_{\ell=1}^k : s<s_1<\cdots<s_k<t\}$, $s_0=s, y_0=x, s_{k+1}=t$, and $y_{k+1}=y$. Here $\hk(x,t):=(2\pi t)^{-1/2}\exp(-x^2/(2t))$ denotes the standard heat kernel. The field $\calZ$ satisfies several other properties including the Chapman-Kolmogorov equations \cite[Theorem 3.1]{akq}. For all $0\le s<r<t$, and $x,y\in \R$ we have
	\begin{align}\label{eq:chapman}
		\calZ(x,s;y,t)=\int_{\R} \calZ(x,s;z,r)\calZ(z,r;y,t) dz.
	\end{align}
	For all $(x,s;y,t)\in \R_{\uparrow}^4$, we also set
	\begin{align}\label{eq:chapman2}
		\calZ(x,s;*,t):=\int_{\R} \calZ(x,s;y,t)dy.
	\end{align}
	
	\begin{definition}[Point-to-point $\cdrp$]\label{def:cdrp} 
		Conditioned on the white noise $\xi$, let $\Pr^{\xi}$ be a measure $C([s,t])$ whose finite dimensional distribution is given by
		\begin{align}\label{eq:cdrp}
			\Pr^{\xi}(X({t_1})\in dx_1, \ldots, X({t_k})\in dx_k)=\frac1{\calZ(x,s;y,t)}\prod_{j=0}^k \calZ(x_j,t_j,;x_{j+1},t_{j+1})dx_1\ldots dx_k.
		\end{align}
		for $s=t_0\le t_1<\cdots<t_k\le t_{k+1}=t$, with $x_0=x$ and $x_{k+1}=y$. 
		
		The measure $\Pr^{\xi}$ also depends on $x$ and $y$ but we suppress it from our notations. We will also use the notation $\cdrp(x,s;y,t)$ and write $X \sim \cdrp(x,s;y,t)$ when  $X(\cdot)$ is a random continuous function on $[s,t]$ with $X(s)=x$ and $X(t)=y$ and its finite dimensional distributions given by \eqref{eq:cdrp} conditioned on $\xi$.
	\end{definition}
	
	\begin{definition}[Point-to-line $\cdrp$]\label{def:cdrp2} 
		Conditioned on the white noise $\xi$, we also let $\Pr_{*}^{\xi}$ be a measure $C([s,t])$ whose finite dimensional distributions are given by
		\begin{align}\label{eq:cdrp2}
			\Pr_*^{\xi}(X({t_1})\in dx_1, \ldots, X({t_k})\in dx_k)=\frac1{\calZ(x,s;*,t)}\prod_{j=0}^k \calZ(x_j,t_j,;x_{j+1},t_{j+1})dx_1\ldots dx_k.
		\end{align}
		for $s=t_0\le t_1<\cdots<t_k\le t_{k+1}=t$, with $x_0=x$ and $x_{k+1}=*$. 
	\end{definition}
	\begin{remark} Note that the Chapman-Kolmogorov equations \eqref{eq:chapman} and \eqref{eq:chapman2} ensure that the finite dimensional distributions in \eqref{eq:cdrp} and \eqref{eq:cdrp2} are consistent, and that $\Pr^{\xi}$ and $\Pr_*^{\xi}$ are probability measures. {The measure $\Pr_*^{\xi}$ also depends on $x$ but we again suppress it from our notations. We similarly use $\cdrp(x,y;*,t)$ to refer to $\Pr_*^{\xi}$.}
	\end{remark}
	
	\begin{theorem}[Pointwise localization for point-to-point $\cdrp$] \label{t:main}  Fix any $p\in (0,1)$. Let $X\sim \cdrp(0,0;0,t)$ and let $f_{p,t}(\cdot)$ denotes the density of $X(pt)$ which depends on the white noise $\xi$.  Then, for all {$t> 0$} the random density $f_{p,t}$ has almost surely a unique mode $\md_{p,t}$. Furthermore, as $t\to \infty$, we have the following convergence in law
		\begin{align}\label{e:main}
			f_{p,t}(x+\md_{p,t}) \stackrel{d}{\to} r_2(x):=\frac{e^{-\calR_2(x)}}{\int\limits_{\R} e^{-\calR_2(y)} d y},
		\end{align}
		in the uniform-on-compact topology. Here $\calR_2(\cdot)$ is a two-sided 3D-Bessel process with diffusion coefficient $2$ {defined in Definition \ref{def:bessel}}. 
	\end{theorem}
	\begin{theorem}[Endpoint localization for point-to-line $\cdrp$] \label{t:main2} Let $X\sim \cdrp(0,0;*,t)$ and let $f_{t}(\cdot)$ denotes the density of $X(t)$ which depends on the white noise $\xi$.  Then for {$t>0$}, the random density $f_{t}$ has almost surely a unique mode $\md_{*,t}$. Furthermore, as $t \to \infty$, we have the following convergence in law
		\begin{align}\label{e:main2}
			f_{*,t}(x+\md_{*,t}) \stackrel{d}{\to} r_1(x):=\frac{e^{-\calR_1(x)}}{\int\limits_{\R} e^{-\calR_1(y)} d y},
		\end{align}
		in the uniform-on-compact topology. Here $\calR_1(\cdot)$ is a two-sided 3D-Bessel process with diffusion coefficient $1$ {defined in Definition \ref{def:bessel}}. 
	\end{theorem}

	\begin{remark} In Proposition \ref{p:besselwd} we show that for a two-sided 3D-Bessel process $\calR_{\sigma}$ with diffusion coefficient $\sigma>0$, $\int_{\R} e^{-\calR_{\sigma}(y)} dy$ is finite almost surely. Thus $r_1(\cdot)$ and  $r_2(\cdot)$ defined in \eqref{e:main2} and \eqref{e:main} respectively are valid random densities. Theorems \ref{t:main} and \ref{t:main2} derive explicit limiting probability densities for the quenched distributions of the endpoints of the point-to-line polymers and the $pt$-point of point-to-point polymers when centered around their respective modes, providing a complete description of the localization phenomena in the $\cdrp$ model. More concretely, it shows that the corresponding points are concentrated in a microscopic region of order one around their ``favorite points" (see Corollary \ref{cor:tight}).
	\end{remark}
	
	We next study the random modes $\md_{*,t}$ and $\md_{p,t}$. The ``favorite point" $\md_{p,t}$ is of the order $t^{2/3}$ and converges in distribution upon scaling. The limit is given in terms of the directed landscape constructed in \cite{dov,mqr} which arises as an universal full scaling limit of several zero temperature models \cite{dv21}. Below we briefly introduce this limiting model in order to state our next result. 
	
	The directed landscape $\calL$ is a random continuous function $\R_{\uparrow}^4 \to \R$ that satisfies the metric composition law
	\begin{align}\label{def:metcomp}
		\calL(x,s;y,t)=\max_{z\in \R} \left[\calL(x,s;z,r)+\calL(z,r;y,t)\right],
	\end{align} 
	with the property that $\calL(\cdot,t_i;\cdot,t_i+s_i^3)$ are independent for any set of disjoint intervals $(t_i,t_i+s_i^3)$, and as a function in $x,y$, $\calL(x,t;y,t+s^3)\stackrel{d}{=}s\cdot\mathcal{S}(x/s^2,y/s^2)$, where $\mathcal{S}(\cdot,\cdot)$ is a {parabolic} Airy Sheet. We omit definitions of the {parabolic} Airy Sheet (see Definition 1.2 in \cite{dov}) except that {$\mathcal{S}(0,\cdot)\stackrel{d}{=}\mathcal{A}(\cdot)$} where {$\mathcal{A}$ is the parabolic $\operatorname{Airy}_2$ process and $\mathcal{A}(x)+x^2$ is the (stationary) $\operatorname{Airy}_2$ process constructed in \cite{ps02}}
	
	\begin{definition}[Geodesics of the directed landscape] \label{def:geo}
		For $(x,s;y,t)\in \R_{\uparrow}^4$, a geodesic from $(x,s)$ to $(y,t)$ of the directed landscape is a random continuous function $\Gamma: [s,t]\to \R$ such that $\Gamma(s)=x$ and $\Gamma(t)=y$ and for any  $s\le r_1< r_2 < r_3 \le t$ we have
		\begin{align*}
			\calL(\Gamma(r_1),r_1;\Gamma(r_3),r_3)=\calL(\Gamma(r_1),r_1;\Gamma(r_2),r_2)+\calL(\Gamma(r_2),r_2;\Gamma(r_3),r_3).
		\end{align*}
		Thus geodesics precisely contain the points where the equality holds in \eqref{def:metcomp}. Given any $(x,s;y,t)\in \R_{\uparrow}^4$, by Theorem 12.1 in \cite{dov}, it is known that almost surely there is a unique geodesic $\Gamma$ from $(x,s)$ to $(y,t)$.
	\end{definition}

	We are now ready to state our favorite point scaling result.
	
	\begin{theorem}[Favorite Point Limit]\label{t:favpt} Fix any $p\in (0,1)$. Consider  $\md_{p,t}$ and $\md_{*,t}$ defined almost surely in Theorems \ref{t:main} and \ref{t:main2} respectively.  As $t\to \infty$ we have
		\begin{align*}
			{2^{-1/3}}t^{-2/3}\md_{*,t} \stackrel{d}{\to} \mathcal{M}, \qquad t^{-2/3}\md_{p,t} \stackrel{d}{\to} \Gamma(p\sqrt{2})
		\end{align*}
		where $\mathcal{M}$ is the almost sure unique maximizer of the $\operatorname{Airy}_2$ process minus a parabola, and $\Gamma:[0,\sqrt{2}]\to \R$ is the almost sure unique geodesic of the directed landscape from $(0,0)$ to $(0,\sqrt{2})$.	
	\end{theorem}
	
	\begin{remark}Theorem \ref{t:favpt} shows that the random mode fluctuates in the order of $t^{2/3}$. This corroborates the fact that $\cdrp$ undergoes superdiffusion as $t\to \infty$. We remark that the $\md_{*,t}$ convergence was anticipated in \cite{moreno2013endpoint} modulo a conjecture about convergence of scaled KPZ equation to the parabolic $\operatorname{Airy}_2$ process. This conjecture was later proved in \cite{vir20,qs20}. 
	\end{remark}

	The proof of Theorem \ref{t:main} relies on establishing fine properties of the partition function $\calZ(x,t):=\calZ(0,0;x,t)$, or more precisely, properties of the log-partition function $\log\calZ(x,t)$. For delta initial data, $\calZ(x,t)>0$ for all $(x,t)\in \R\times (0,\infty)$ almost surely \cite{flo}. Thus the logarithm of the partition function $\calH(x,t):=\log \calZ(x,t)$ is well-defined. It formally solves the KPZ equation:
	\begin{align}\label{eq:kpz}
		\partial_t\calH=\tfrac12\partial_{xx}\calH+\tfrac12(\partial_x\calH)^2+\xi, \qquad \calH=\calH(x,t), \quad (x,t)\in \R \times [0,\infty).
	\end{align}
	The KPZ equation was introduced in \cite{kardar1986dynamic} to study the random growing interfaces and since then it has been extensively studied in both the mathematics and the physics communities. We refer to \cite{ferrari2010random,quastel2011introduction,corwin2012kardar,quastel2015one,chandra2017stochastic,corwin2019} for partial surveys. 	 
	
	As a stochastic PDE, \eqref{eq:kpz} is ill-posed due to the presence of the nonlinear term $\frac{1}{2}(\partial_x\calH)^2$. The above notion of solutions from the logarithm of the solution of SHE is referred to as the Cole-Hopf solution. The corresponding initial data is called the narrow wedge initial data for the KPZ equation. Other notions of solutions, such as  regularity structures \cite{hairer2013solving,hairer2014theory}, paracontrolled distributions \cite{gubinelli2015paracontrolled,gubinelli2017kpz}, and energy solutions \cite{gonccalves2014nonlinear,gubinelli2018energy}, have been shown to coincide with the Cole-Hopf solution  {within the class of initial datas the theory applies}. 
	
	\medskip
	
	To prove Theorem \ref{t:main}, one needs to understand how multiple copies of the KPZ equation behave around {the maximum of their sum}. We present below our first main result that {studies the limiting behavior of sum of two independent copies of KPZ equation re-centered around {the maximizer of the sum, which we often refer to as the joint maximizer in the subsequent text} as $t\to \infty$.} 
	
	{\begin{theorem}[Bessel behavior around {the maximizer}] \label{t:bessel}	  Fix $k=1$ or $k=2$. Consider $k$ independent copies of the KPZ equation $\{\calH_i(x,t)\}_{i=1}^k$ started from the narrow wedge initial data. For each $t>0$, almost surely, the process $x\mapsto \sum_{i=1}^k \calH_i(x,t)$ has a unique maximizer, say $\mx_{k,t}$.  Furthermore, as $t\to\infty$, we have the following convergence in law
			\begin{align}\label{eq:bessel}
				R_k(x,t):=\sum_{i=1}^k \left[\calH_i(\mx_{k,t},t)-\calH_i(x+\mx_{k,t},t)\right] \stackrel{d}{\to} \calR_k(x)
			\end{align}
			in the uniform-on-compact topology. Here $\calR_k(x)$ is a two-sided Bessel process with diffusion coefficient $k$.
	\end{theorem}}
	
	The next result captures the behaviors of the increments of $\calH(\cdot,t)$ and complements Theorem \ref{t:bessel}. It's a by-product of our analysis and doesn't appear in the proof of Theorem \ref{t:main}.
	
	{\begin{theorem}[Ergodicity of the KPZ equation] \label{t:ergodic} Consider the KPZ equation $\calH(x,t)$ started from the narrow wedge initial data. As $t\to\infty$, we have the following convergence in law
			\begin{align*}
				{	\calH(x,t)-\calH(0,t)} \stackrel{d}{\to} B(x)
			\end{align*}
			in the uniform-on-compact topology. Here $B(x)$ is a two-sided standard Brownian motion. 
	\end{theorem}}
	
	\begin{remark} For a Brownian motion on a compact interval, the law of the process when re-centered around its maximum is absolutely continuous w.r.t.~Bessel process. In light of Theorem \ref{t:ergodic}, one expects the Bessel process as a limit in Theorem \ref{t:bessel}. The diffusion coefficient is $k$ because there are $k$ independent copies of the KPZ equation.
	\end{remark}
	
	\begin{remark} We prove \eqref{eq:bessel} for $k=1$ and $k=2$ only, where $k=1$ case relates to Theorem \ref{t:main2} and the $k=2$ case relates to Theorem \ref{t:main}. Our proof strategy for Theorem \ref{t:bessel} can be adapted for general $k\ge 3$ and Remark \ref{r:gen_k} explains the missing pieces for the proof of \eqref{eq:bessel} for general $k$. While Theorem \ref{t:bessel} for general $k$ is an interesting result, due to brevity and the lack of applications to our localization problem, we restrict to only $k=1,2$. 
	\end{remark}
	
A useful property in establishing the ergodicity of a given Markov process is the strong Feller property. For instance, \cite{h18} 
studied the strong Feller property for singular SPDEs to establish ergodicity for a multicomponent KPZ equation. However, \cite{h18} techniques and results are limited to only periodic boundary conditions, i.e. on torus domain, and are thus inaccessible for the KPZ equation with narrow-wedge initial data. 
	
	In addition to the strong Feller property, we can also probe the KPZ equation's ergodicity through the lens of the KPZ universality class. Often viewed as the fundamental positive temperature model of the latter, the KPZ equation shares the same $1:2:3$ scaling exponents and universal long-time behaviors {expected or proven for other members of the class}. A widely-held belief about the KPZ universality class is that under the $1:2:3$ scaling and in the large scale limit, all models in the class converge to an universal scaling limit, the KPZ fixed point \cite{dov,mqr}. This very conjecture has been recently proved for the KPZ equation in \cite{qs20,vir20}. Here we recall a special case of the statement in \cite{qs20} useful to us later. Consider the  $1:2:3$ scaling of the KPZ equation (the scaled KPZ equation)
		$$\h_t(x):=t^{-1/3}\left(\calH(t^{2/3}x,t)+\tfrac{t}{24}\right).$$
		Then $2^{1/3}\h_t(2^{1/3}x)$ converges to the parabolic $\operatorname{Airy}_2$ process as $t\to \infty$. Note that the parabolic $\operatorname{Airy}_2$ process is the marginal of the parabolic Airy Sheet, which is a canonical object in the construction of the KPZ fixed point and the related directed landscape (see \cite{dov, qs20}).

	On the KPZ fixed point level, ergodicity and behaviors around the maximum are better understood. Under the zero temperature setting, numerous results and techniques address the ergodicity question for the KPZ fixed point. For instance, due to the $1:2:3$ scaling invariance, ergodicity of the fixed point is equivalent to the local Brownian behavior (\cite[Theorem 4.14 and 4.15]{mqr}) or can be deduced in \cite{pim17} using coupling techniques applicable only in zero temperature settings.
	
	Meanwhile, \cite{dsv} showed that local Brownianity and local Bessel behaviors around the maximizer hold for any process which is absolutely continuous w.r.t.~Brownian motions on every compact set. The scaled KPZ equation possesses such property \cite{CH16} and its ergodicity question can be transformed into local Brownian behaviors of the scaled KPZ equation. Note that we have
		$$\calH(x,t)-\calH(0,t)=t^{-1/3}\left(\h_t(t^{-2/3}x)-\h_t(0)\right).$$
		However, the law of $\h_t$ changes with respect to time and the diffusive scaling precisely depends on $t$. Therefore it is unclear how to extend the soft techniques in \cite[Lemma 4.3]{dsv} for the KPZ equation to address the limiting local Brownian behaviors in above setting.
	
	{Another recent line of inquiries regarding the behavior around the maxima is the investigation of the fractal nature of exceptional times for the KPZ fixed point with multiple maximizers \cite{chhm,d22}. In \cite{chhm}, the authors computed the Hausdorff dimension of the set of times for the KPZ fixed point with at least two maximizers and was extended to the case of exactly $k$ maximizers in \cite{d22}. The latter work relied on a striking property of the KPZ fixed point where it becomes stationary in $t$ after recentering at the maximum with Bessel initial conditions. This property considerably simplified their analysis. Other initial data were then accessed through a transfer principle from \cite{sv21}. Unfortunately, analogous properties for the KPZ equation are not available.}

	\subsection{Proof Ideas}\label{sec:pf}  
	In this section we sketch the key ideas behind the proofs of our main results. For brevity, we present a heuristic argument for the proofs of Theorem \ref{t:main} and the related Theorem \ref{t:bessel} with the $k=2$ case only. The proofs for the point-to-line case (Theorem \ref{t:main2}) and the related $k=1$ case of Theorem \ref{t:bessel} and ergodicity (Theorem \ref{t:ergodic}) follow from similar ideas. Meanwhile, the methods related to the uniqueness and convergence of random modes (Theorem \ref{t:favpt}) are of a different flavor. We present them directly in Section \ref{sec:rm} as the arguments are more straightforward.  
	
	Recall from Theorem \ref{t:main} that $f_{p,t}$ denotes the quenched density of $X(pt)$ for $X\sim \cdrp(0,0;0,t)$. To simplify our discussion below, we let $p=\frac12$ and replace $t$ by $2t$. \eqref{eq:cdrp} gives us
	\begin{align*}
		f_{\frac12,2t}(x)=\frac{\calZ(0,0;x,t)\calZ(x,t;0,2t)}{\calZ(0,0;0,2t)}.
	\end{align*}
	Recall the chaos expansion for $\calZ(x,s;y,t)$ from \eqref{eq:chaos}. Note that $\calZ(0,0;x,t)$ and $\calZ(x,t;0,2t)$ are independent for using different sections of the noise $\xi$. A change of variable and symmetry yields that $\calZ(x,t;0,2t)$ is same in distribution as $\calZ(0,0;x,t)$ as a process in $x$. Thus as a process in $x$, $\calZ(0,0;x,t)\calZ(x,t;0,2t)\stackrel{d}{=}e^{\calH_1(x,t)+\calH_2(x,t)}$ were $\calH_1(x,t)$ and $\calH_2(x,t)$ are independent copies of the KPZ equation with narrow wedge initial data. This puts Theorem \ref{t:main} in the framework of Theorem \ref{t:bessel}. Viewing the density around its unique random mode $\md_{\frac12,2t}$ (that is the maximizer), we may thus write $f_{\frac12,2t}(x+\md_{\frac12,2t})$ as $$\frac{e^{-R_2(x,t)}}{\int\limits_{\R} e^{-R_2(y,t)}dy},$$
	where $R_2(x,t)$ is defined in \eqref{eq:bessel}. For simplicity, let us use the notation $\mx=\md_{\frac12,2t}$. 
	
	The rest of the argument hinges on the following two results:
	\begin{enumerate}[label=(\roman*)]
		\setlength{\itemsep}{0.5em}
		\item \label{idea1} \textit{Bessel convergence}: $R_2(x,t)$ converges weakly to $3$D-Bessel process with diffusion coefficient $2$ in the uniform-on-compact topology (Theorem \ref{t:bessel}). 
		\item \label{idea2} \textit{Controlling the tails}: $\int_{[-K,K]^c} e^{-R_2(y,t)}dy$ can be made arbitrarily small for all large $t$ by taking large $K$ (Proposition \ref{p:ctail}).
	\end{enumerate}
	
	Theorem \ref{t:main} then follows from the above two items by standard analysis. We now explain the ideas behind items \ref{idea1} and \ref{idea2} and
	our principal tool is the Gibbsian line ensemble, which is an object of integrable origin often used in probabilistic settings. More precisely, we use the \textit{KPZ line ensemble} (recalled in Proposition \ref{line-ensemble}), i.e. a set of random continuous functions whose lowest indexed curve is same in distribution as the narrow wedge solution of the KPZ equation. The law of the lowest indexed curve enjoys a Gibbs property called the $\mathbf{H}$-Brownian Gibbs property. This property states that the law of the lowest indexed curve conditioned on an interval depends only on the curve indexed one below and the starting and ending points. Furthermore, this conditional law is absolutely continuous w.r.t.~a Brownian bridge of the same starting and ending points with an explicit expression of the Radon-Nikodym derivative. 
	
	 We now recast \ref{idea1} in the language of Gibbsian line ensemble. Note that $R_2(x,t)$ is a sum of two independent KPZ equations viewed from the maximum of the sume (\eqref{eq:bessel}). Accessing its distribution requires a precise description of the conditional joint law of the top curves of two independent copies of the KPZ line ensemble on random intervals around the joint maximizer. Thus \ref{idea1} reduces to the following results, which we elaborate on individually:
		\begin{enumerate}[label=(\alph*)]
		\setlength{\itemsep}{0.5em}
		\item \label{idea3} Two Brownian bridges when viewed around the maximum of their sum can be given by two pairs of non-intersecting Brownian bridges to either side of the maximum (Proposition \ref{propA}). 
		\item \label{idea4} For a suitable $K(t)\uparrow \infty$, the Radon-Nikodym derivatives associated with the KPZ line ensembles (see \eqref{eq:bgibbs} for the precise expression of Radon-Nikodym derivative)  on the random interval $[\mx-K(t),\mx+K(t)]$ containing the maximizer goes to $1$.
	\end{enumerate}

	Combining the above two ideas, we can conclude the joint law of  \begin{align}
		    \label{d1d2}
		    (D_1(x,t),D_2(x,t)):=(\calH_1(\mx,t)-\calH_1(\mx+x,t),\calH_2(\mx+x,t)-\calH_2(\mx,t))
		\end{align} on $x\in [-K(t),K(t)]$ is close to two-sided pair of non-intersecting Brownian bridges with the same starting point and appropriate endpoints. Upon taking $t\to \infty$, one obtains a two-sided Dyson Brownian motion $(\calD_1,\calD_2)$ defined in Definition \ref{def:dbm} as a distributional limit. Proposition \ref{p:dyson} is the precise rendering of this fact. Finally a 3D-Bessel process emerges as the difference of two parts of the Dyson Brownian motion: $\calD_1(\cdot)-\calD_2(\cdot)$ (see Lemma \ref{l:dtob}).  
	
	Before expanding upon items \ref{idea3} and \ref{idea4}, let us explain the reasons behind our approach. Since our desired random interval includes the maximizer of two independenet copies or the joint maximizer, it is not a stopping domain and is inaccessible by classical properties such as the strong Gibbs property for KPZ line ensemble. Note that a similar context of the KPZ fixed point appeared in \cite{chhm}, where the authors used Gibbs property on random intervals defined to the right of the maximizer in their proof. However, \cite{chhm} relied on a path decomposition of Markov processes at certain spatial times from \cite{millar}, which states that conditioned on the maximizer, the process to the \textit{right} of the maximizer is Markovian. However in our case, the intervals around the maximum is \textit{two-sided}. Thus the abstract setup of \cite{millar} is not suited for our case. 
	Thus, the precise description of the law given in item \ref{idea3} is indispensable to our argument.
	
	Next, one needs an exact comparison of the Brownian law and KPZ law to transition between the two. Traditional tools such as stochastic monotonicity for the KPZ line ensembles help obtain one-sided bounds for monotone events. Especially for tail estimates of the KPZ equation, it reduces the problem to the setting of Brownian bridges, which can be treated classically. However, this approach only produces a one-sided bound, which is insufficient for the precise convergence we need. Hence we treat the Radon-Nikodym derivative directly to exactly compare the two laws.
	
	\begin{figure}[h!]
		\centering
		\includegraphics[width=16cm]{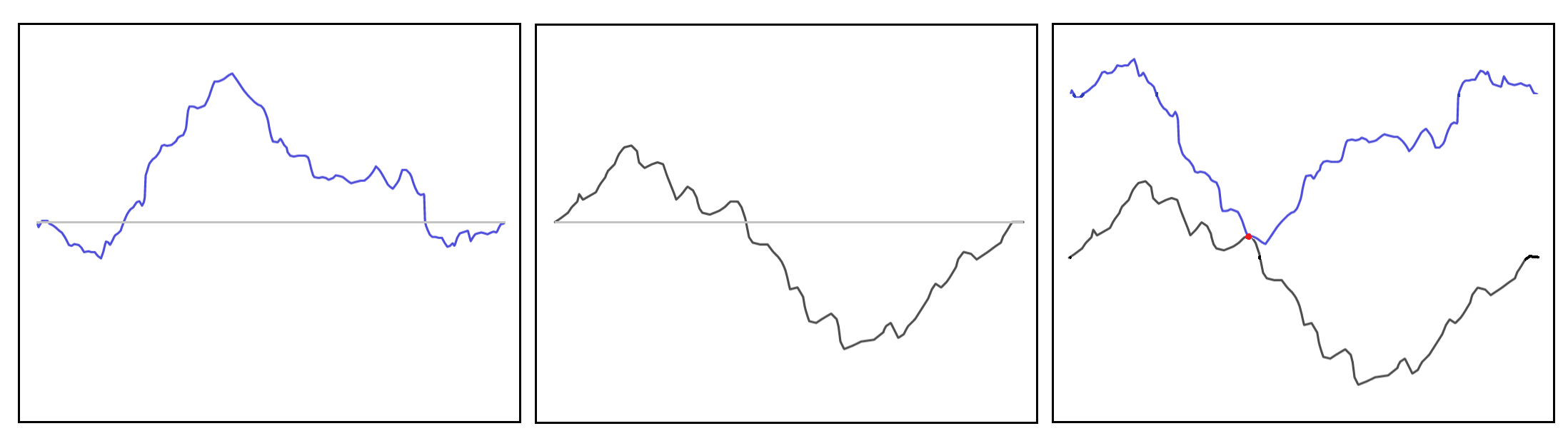}
		\caption{First idea for the proof: The first two figures depict two independent Brownian bridges `blue' and `black' on $[0,1]$ starting and ending at zero. We flip the blue one and shift it appropriately so that when it is superimposed with the black one, the blue curve always stays above the black one and touches the black curve at exactly one point. The superimposed figure is shown in third figure. The red point denotes the `touching' point or equivalently the joint maximizer. Conditioned on the max data, the trajectories on the left and right of the red points are given by two pairs of non-intersecting Brownian bridges with appropriate end points.} 
		\label{fig:idea1}
	\end{figure}

	To describe the result in item \ref{idea3}, consider two independent Brownian bridges $\bar{B}_1$ and $\bar{B}_2$ on $[0,1]$ both starting and ending at zero. See Figure \ref{fig:idea1}. Let $M=:\operatorname{argmax} (\bar{B}_1(x)+\bar{B}_2(x))$. We study the conditional law of $(\bar{B}_1,\bar{B}_2)$  given the max data: $(M,\bar{B}_1(M),\bar{B}_2(M))$. The key fact from Proposition \ref{propA} is that conditioned on the max data $$(\bar{B}_1(M)-\bar{B}_1(M-x),\bar{B}_2(M-x)-\bar{B}_2(M))_{x\in [0,M]}, \quad (\bar{B}_1(M)-\bar{B}_1(x),\bar{B}_2(x)-\bar{B}_2(M))_{x\in [M,1]}$$
	are independent and each is a non-intersecting Brownian bridge with appropriate end points (see Definition \ref{def:nibb}). The proof proceeds to show such a decomposition at the level of discrete random walks before taking diffusive limits to get the same for Brownian motions and finally for Brownian bridges. The details are presented in Section \ref{sec:localmax}.

	\begin{figure}[h!]
		\centering
		\includegraphics[width=9cm]{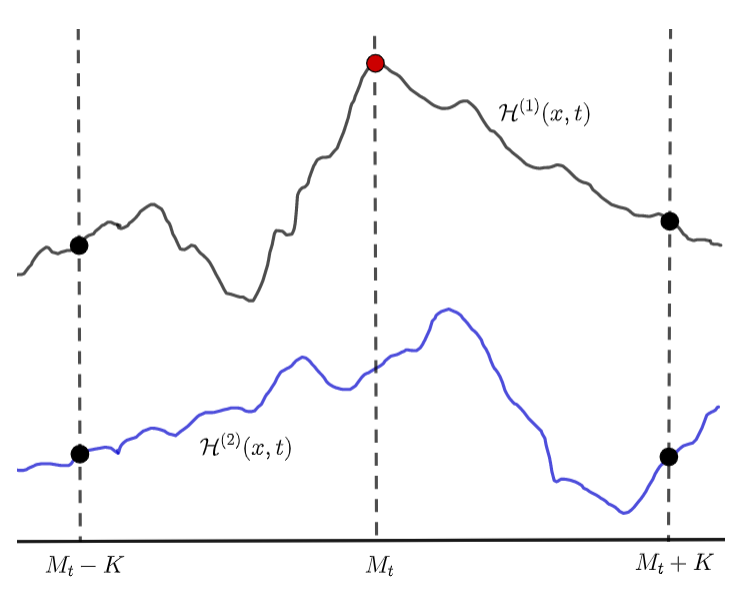}
		\caption{Second idea for the proof: For all ``good" boundary data and max data, with high probability, there is an uniform separation of order $t^{1/3}$ between the first two curves on the random interval $[M_t-K,M_t+K]$.} 
		\label{fig:idea2}
	\end{figure}

To illustrate the idea behind item \ref{idea4}, let us consider an easier yet pertinent scenario. Let $\calH^{(1)}(x,t)$ and $\calH^{(2)}(x,t))$ be the first two curves of the KPZ line ensemble. Let $M_t=\operatorname{argmax} \calH^{(1)}(x,t)$. We consider the interval $I_t:=[M_t-K,M_t+K]$. See Figure \ref{fig:idea2}. We show that
	\begin{enumerate}
		\item The maximum is not too high: $\calH^{(1)}(M_t,t)-\calH^{(1)}(M_t\pm K,t)=O(1)$,
		\item The gap at the end points is sufficiently large: $\calH^{(1)}(M_t\pm K,t)-\calH^{(2)}(M_t\pm K,t)=O(t^{1/3})$.
		\item The  fluctuations of the second curve on $I_t$ are $O(1)$.
	\end{enumerate}   Under the above favorable class of boundary data: $\calH^{(1)}(M_t\pm K,t), \calH^{(2)}(\cdot,t)$ and the max data: $(M_t,\calH^{(1)}(M_t,t))$, we show that the conditional fluctuations of the first curve are $O(1)$. This forces a uniform separation between the first two curves throughout the random interval $I_t$. Consequently the Radon-Nikodym derivative in \eqref{eq:bgibbs} converges to $1$ as $t \rightarrow \infty$. 
	
	We rely on tail estimates for the KPZ equation as well as some properties of the Airy line ensemble which are the distributional limits of the scaled KPZ line ensemble defined in \eqref{eq:htx} to conclude such a statement rigorously. Section \ref{sec:tools} contains a review of the necessary tools. Note that the rigorous argument for the Radon-Nikodym derivative in the proof of Theorem \ref{t:main} (Proposition \ref{p:dyson}) is more involved. Indeed, one needs to consider another copy of line ensemble and argue that similar uniform separation holds for both when viewed around the joint maximum $\mx$. We also take $K=K(t)\uparrow \infty$ and the separation length is consequently different.

	We have argued so far that $(D_1(x,t),D_2(x,t))$ defined in \eqref{d1d2} {jointly} converges to a two-sided Dyson Brownian motion. This convergence holds in the uniform-on-compact topology. However, this does not address the question about behavior of the tail integral in \ref{idea2} $$\int_{[-K,K]^c} e^{D_2(y,t)-D_1(y,t)}\d y.$$ 
	
	\begin{figure}[h!]
		\centering
		\includegraphics[width=12cm]{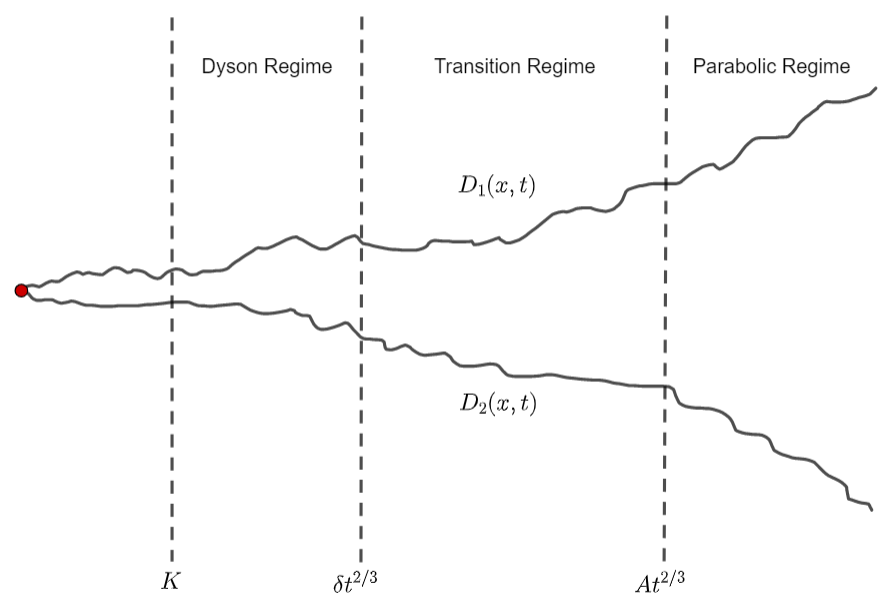}
		\caption{Third idea for the proof: The three regimes} 
		\label{fig:idea3}
	\end{figure}

	To control the tail, we divide  the tail integral into  three parts based on the range of integration (See Figure \ref{fig:idea3}):
	
	\begin{itemize}
		\item \textit{Dyson regime:} The law of $(\calD_1(x,t),\calD_2(x,t))$ on the interval $[0,\delta t^{2/3}]$ is comparable to that of the Dyson Brownian motions for small $\delta$ and for large $t$. For Dyson Brownian motions, w.h.p. $\calD_1(x)-\calD_2(x) \ge \e |x|^{1/4}$ for all large enough $|x|$. This translates to $(D_1(x,t),D_2(x,t))$ and provides a decay estimate over this interval.
		
		\item \textit{Parabolic Regime:} The maximizer $\mx$ lies in a window of order $t^{2/3}$ region w.h.p.. On the other hand, the KPZ equation upon centering has a parabolic decay: $\calH(x,t)+\frac{t}{24} \approx -\frac{x^2}{2t}+O(t^{1/3})$. Thus taking $A$ large enough ensures w.h.p. $D_1(x,t) \approx \frac{x^2}{4t}$ and $D_2(x,t) \approx -\frac{x^2}{4t}$ on the interval $[At^{2/3},\infty)$. These estimates give a rapid decay of our integral in this regime.
		
		\item \textit{Transition Regime:} Between the two regimes, we use soft arguments of non-intersecting brownian bridges to ensure that $D_1(x,t)-D_2(x,t) \ge \rho t^{1/3}$ w.h.p. uniformly on $[\delta t^{2/3},A t^{2/3}]$. 
	\end{itemize}
	
	Proposition \ref{pgamma} and Proposition \ref{p:ctail} are the precise manifestations of the above idea. Proposition \ref{pgamma} provides decay estimates in the Dyson and transition regimes for Brownian objects. Proposition \ref{p:ctail} translates the estimates in Proposition \ref{pgamma} to $D_1,D_2$ for the ``shallow tail regime" (see Figure \ref{fig:tail}). The parabolic regime or the ``deep tail" in Section \ref{sec:pfmain} is addressed in Proposition \ref{p:ctail}.

	\subsection*{Outline} The remainder of the paper is organized as follows. Section \ref{sec:tools} reviews some of the existing results related to the KPZ line ensemble and its zero temperature counterpart, the Airy line ensemble. We then  prove the existence and uniqueness of random modes in Theorem \ref{t:favpt} in Section \ref{sec:rm}. Section \ref{sec:localmax} is dedicated to the behaviors of the Brownian bridges around their joint maximum. Two important objects are defined in this Section: the Bessel bridges and the non-intersecting Brownian bridges. Several properties of these two objects are subsequently proved in Section \ref{sec:bbnibb}. The proofs of Theorems \ref{t:bessel} and \ref{t:ergodic} comprise section \ref{sec:ergbb}. Finally in Section \ref{sec:pfmain}, we complete the proofs of Theorems \ref{t:main} and \ref{t:main2}. Appendix \ref{sec:ap1} contains a convergence result about non-intersecting random walks used in Section \ref{sec:localmax}. 
	
	\subsection*{Acknowledgements}  We thank Shirshendu Ganguly and Promit Ghosal for numerous discussions in fall 2021. In particular, parts of Theorem \ref{t:favpt} were anticipated by Shirshendu Ganguly and Promit Ghosal. We thank Ivan Corwin for inputs on an earlier draft of the paper and several useful discussions. We also thank Milind Hegde and Shalin Parekh for useful discussions. The project was initiated during the authors' participation in the `Universality and Integrability in Random Matrix Theory and Interacting Particle
	Systems' research program hosted by the Mathematical Sciences Research Institute (MSRI) in Berkeley, California in fall 2021. The authors thank the program organizers for their hospitality and acknowledge the support from NSF DMS-1928930. 

	\section{Basic framework and tools}\label{sec:tools} 
	
	\subsection*{Remark on Notations} Throughout this paper we use $\Con = \Con(\alpha, \beta,\gamma, \ldots) > 0$ to denote a generic deterministic positive finite constant that may change from line to line, but dependent on the designated variables $\alpha,\beta, \gamma, \ldots$. We will often write $\Con_{\alpha}$ in case we want to stress the dependence of the constant to the variable $\alpha$. We will use serif fonts such as $\m{A}, \m{B}, \ldots$ to denote events as well as $\cdrp, \dbm \ldots$ to denote laws. The distinction will be clear from the context. The complement of an event $\m{A}$ will be denoted as $\neg \m{A}$. \\

	{In this section, we present the necessary background on the directed landscape and Gibbsian line ensembles including the Airy line ensemble and the KPZ line ensemble as well as known results on these objects that are crucial in our proofs.}
	
	\subsection{The directed landscape and the Airy line ensemble} We recall the definition of the directed landscape and several related objects from \cite{dov,dv18}. The directed landscape is the central object in the KPZ universality class constructed as a scaling limit of the Brownian Last Passage percolation (BLPP). We recall the setup of the BLPP below to define the directed landscape. 
	
	\begin{definition}[Directed landscape]
		Consider an infinite collection $B:=(B_{k}(\cdot))_{k\in \Z}$ of independent two-sided Brownian motions with diffusion coefficient $2$. For $x\le y$ and {$n\le m$}, the last passage value from $(x,m)$ to $(y,n)$ is defined by
		\begin{align*}
			B[(x,m)\to (y,n)]=\sup_{\pi} \sum_{k=n}^{m} [B_k(\pi_k)-B_k(\pi_{k-1})],
		\end{align*}
		where the supremum is over all $\pi\in \Pi_{m,n}(x,y):=\{\pi_{m} \le \cdots \le \pi_n \le \pi_{n-1} \mid \pi_{m}=x,\pi_{n-1}=y\}$. Now for any $(x,s;y,t)\in \R_{\uparrow}^4$, we denote $(x,s)_n:=(s+2xn^{-1/3},-\lfloor sn\rfloor)$ and $(y,t)_n:=(t+2yn^{-1/3},-\lfloor tn\rfloor)$ and define
		\begin{align*}
			\calL_n(x,s;y,t):=n^{1/6} B_n[(x,s)_n\to (y,t)_n]-2(t-s)n^{2/3}-2(y-x)n^{1/3}.
		\end{align*}
		The directed landscape $\calL$ is the distributional limit of $\calL_n$ as $n \rightarrow \infty$ with respect to the uniform convergence on compact subsets of $\R_{\uparrow}^4$. By \cite{dov}, the limit exists and is unique.
	\end{definition}
	
	The marginal $\mathcal{A}_1(x):=\mathcal{L}(0,0;x,1)$ is known as the parabolic $\operatorname{Airy}_2$ process. In \cite{ps02} the $\operatorname{Airy}_2$ process $\mathcal{A}_1(x)+x^2$ was constructed as the scaling limit of the polynuclear growth model. At the same time, $\mathcal{A}_1(x)$ can also be viewed as the top curve of the Airy line ensemble, which we define formally below in the approach of \cite{CH14}.
	
	\begin{definition}[Brownian Gibbs Property]\label{def:bgp} Recall the general notion of line ensembles from Section 2 in \cite{CH14}. Fix $k_1\le k_2$ with $k_1,k_2\in \N$ and an interval $(a,b)\in \R$ and two vectors $\vec{x},\vec{y} \in \R^{k_2-k_1+1}$. Given two measurable functions $f,g:(a,b)\to \R\cup \{\pm\infty\}$, let $\mathbb{P}^{k_1, k_2, (a,b), \vec{x}, \vec{y},f,g}_{\mathrm{nonint}}$ be the law of $k_2-k_1+1$ many independent Brownian bridges (with diffusion coefficient $2$) $\{B_i:[a,b]\to \R\}_{i=k_1}^{k_2}$ with $B_i(a)=x_i$ and $B_i(b)=y_i$ conditioned on the event that
		$$f(x)>B_{k_1}(x)>B_{k_1+1}(x)>\cdots>B_{k_2}(x)>g(x), \quad \mbox{for all }x\in [a,b].$$ 
		Then the $\mathbb{N}\times \mathbb{R}$ indexed line ensemble $\mathcal{L}=(\mathcal{L}_1,\mathcal{L}_2,\ldots)$ is said to enjoy the \emph{Brownian Gibbs property} if, for all $K = \{k_1,\ldots, k_2\}\subset \mathbb{N}$ and $(a,b)\subset \R$, the following distributional equality holds:
		\begin{align*}
			\mathrm{Law}\Big(\mathcal{L}_{K\times (a,b)} \text{ conditioned on }\mathcal{L}_{\mathbb{N}\times \R \backslash K \times (a,b)}\Big) = \mathbb{P}^{k_1,k_2, (a,b),\vec{x}, \vec{y}, f, g}_{\mathrm{nonint}},
		\end{align*}
		where $\vec{x}= (\mathcal{L}_{k_1}(a), \ldots , \mathcal{L}_{k_2}(a))$, $\vec{y} =(\mathcal{L}_{k_1}(b),\ldots , \mathcal{L}_{k_2}(b))$, {$\mathcal{L}_{k_1-1}=f$ (or    $\infty$ if $k_1=1$)} and  $\mathcal{L}_{k_2+1}=g$. 
	\end{definition}
	
	\begin{definition}[Airy line ensemble]\label{def:ale} The Airy line ensemble $\mathcal{A}=(\mathcal{A}_1,\mathcal{A}_2, \ldots)$ is the unique $\N\times \R$-indexed line ensemble satisfying Brownian Gibbs property whose top curve $\mathcal{A}_1(\cdot)$ is the parabolic $\operatorname{Airy}_2$ process. The existence and uniqueness of $\calA$ follow from \cite{CH14} and \cite{dm20} respectively. 
	\end{definition}
	
	The Airy line ensemble is in fact a strictly ordered line ensemble in the sense that almost surely, \begin{align}\label{eq:order}
		\mathcal{A}_k(x) >\mathcal{A}_{k+1}(x) \mbox{ for all }k\in \N, x\in \R.
	\end{align} 
	\eqref{eq:order} follows from the Brownian Gibbs property and the fact that for each $x\in \R$, $(\mathcal{A}_k(x)+x^2)_{k\ge 1}$ is equal in distribution to the Airy point process. The latter is strictly ordered. In \cite{dv18}, the authors studied several probabilistic properties of the Airy line ensembles such as tail estimates and modulus of continuity. Below we state an extension of one of such results used later in our proof.
	
	\begin{proposition}\label{p:amodcon} Fix $k\ge 1$. There exists a universal constant $\Con_k>0$ such that for all $m>0$ and $R\ge 1$ we have
		\begin{align}\label{e:amodcon}
			\Pr\left(\sup_{\substack{x\neq y\in [-R,R] \\ |x-y|\le 1}} \frac{|\mathcal{A}_k(x)+x^2-\mathcal{A}_k(y)-y^2|}{\sqrt{|x-y|}\log^{\frac12}\frac{2}{|x-y|}}\ge m\right) \le \Con_k \cdot R \exp\left(-\tfrac1\Con_k m^2\right).
		\end{align}
	\end{proposition}
	
	\begin{proof} Fix $k\ge 1$. By \cite[Lemma 6.1]{dv18} there exists a constant $\Con_k$ such that for all $x,y\in \R$ with $|x-y|\le 1$, we have
		\begin{align*}
			\Pr\left( {|\mathcal{A}_k(x)+x^2-\mathcal{A}_k(y)-y^2|}\ge m\sqrt{x-y}\right) \le \Con_k \exp\left(-\tfrac1{\Con_k} m^2\right).
		\end{align*}
		Thus applying Lemma 3.3 in \cite{dv18} (with $d=1$, $T=[-R,R]$, $r_1=1$, $\alpha_1=\frac12$, $\beta_1=2$) and adjusting the value of $\Con_k$ yields \eqref{e:amodcon}.
	\end{proof}

	\subsection{KPZ line ensemble}
	Let $\mathcal{L}=(\mathcal{L}_1,\mathcal{L}_2,\ldots)$ be an $\mathbb{N}\times \mathbb{R}$-indexed line ensemble. Fix $k_1\le k_2$ with $k_1,k_2\in \N$ and an interval $(a,b)\in \R$ and two vectors $\vec{x},\vec{y} \in \R^{k_2-k_1+1}$.
	Given a continuous function $\mathbf{H}: \R \to [0,\infty)$ (Hamiltonian) and two measurable functions $f,g:(a,b)\to \R\cup \{\pm\infty\}$, the law $\mathbb{P}^{k_1, k_2, (a,b), \vec{x}, \vec{y}, f,g}_{\mathbf{H}}$ on $\mathcal{L}_{k_1},\ldots,\mathcal{L}_{k_2} : (a,b) \to \R$ has the following Radon-Nikodym derivative with respect to $\mathbb{P}^{k_1, k_2, (a,b), \vec{x}, \vec{y}}_{\mathrm{free}},$ the law of $k_2-k_1+1$ many independent Brownian bridges (with diffusion coefficient $1$) taking values $\vec{x}$ at time $a$ and $\vec{y}$ at time $b$: 
	\begin{align}\label{eq:bgibbs}
		\frac{d\mathbb{P}^{k_1, k_2,(a,b), \vec{x}, \vec{y}, f, g}_{\mathbf{H}}}{d\mathbb{P}^{k_1, k_2, (a,b), \vec{x}, \vec{y}}_{\mathrm{free}}}(\mathcal{L}_{k_1}, \ldots , \mathcal{L}_{k_2}) &= \frac{\exp\bigg\{- \sum_{i=k_1}^{k_2+1} \int \mathbf{H}\big(\mathcal{L}_{i}(x)- \mathcal{L}_{i-1}(x)\big) dx \bigg\} }{Z^{k_1, k_2, (a,b), \vec{x}, \vec{y}, f,g}_{\mathbf{H}}},
	\end{align}
	where  $\mathcal{L}_{k_1-1}=f$, or    $\infty$ if $k_1=1$; and  $\mathcal{L}_{k_2+1}=g$. Here, $Z^{k_1, k_2, (a,b), \vec{x}, \vec{y}, f, g}_{\mathbf{H}}$ is the normalizing constant which produces a probability measure. {We say $\mathcal{L}$ enjoys the \emph{$\mathbf{H}$-Brownian Gibbs property} if, for all $K = \{k_1,\ldots, k_2\}\subset \mathbb{N}$ and $(a,b)\subset \R$, the following distributional equality holds:
		\begin{align*}
			\mathrm{Law}\Big(\mathcal{L}_{K\times (a,b)} \text{ conditioned on }\mathcal{L}_{\mathbb{N}\times \R \backslash K \times (a,b)}\Big) = \mathbb{P}^{k_1,k_2, (a,b),\vec{x}, \vec{y}, f, g}_{\mathbf{H}} \, ,
		\end{align*}
		where $\vec{x}= (\mathcal{L}_{k_1}(a), \ldots , \mathcal{L}_{k_2}(a))$, $\vec{y} =(\mathcal{L}_{k_1}(b),\ldots , \mathcal{L}_{k_2}(b))$, and where again $\mathcal{L}_{k_1-1}=f$, or    $\infty$ if $k_1=1$; and  $\mathcal{L}_{k_2+1}=g$. }\\
	
	In the following text, we consider a specific class of $\mathbf{H}$ such that 
	\begin{align}
		\mathbf{H}_t(x)= {t^{2/3}}e^{t^{1/3}x}.   \label{eq:Ht}
	\end{align}
	The next proposition then recalls the unscaled and scaled KPZ line ensemble constructed in~\cite{CH16} with $\mathbf{H}_t$-Brownian Gibbs property.
	
	\begin{proposition}[Theorem 2.15 in \cite{CH16}] \label{line-ensemble}
		Let $t\ge 1$. There exists an $\N\times \R$-indexed line ensemble $\mathcal{H}_t =\{\mathcal{H}^{(n)}_{t}(x)\}_{n\in \N, x\in \R}$ such that:
		\begin{enumerate}[label=(\alph*), leftmargin=15pt]
			\item the lowest indexed curve $\mathcal{H}^{(1)}_{t}(x)$ is equal in distribution (as a process in $x$) to the Cole-Hopf solution $\mathcal{H}(x,t)$ of the KPZ equation started from the narrow wedge initial data and the line ensemble $\mathcal{H}_{t}$ satisfies the $\mathbf{H}_1$-Brownian Gibbs property;
			\item the scaled KPZ line ensemble $\{\mathfrak{h}^{(n)}_t(x)\}_{n\in \N,x\in \R}$, defined by 
			\begin{align}\label{eq:scaleKPZ}
				\mathfrak{h}^{(n)}_t(x) :=   t^{-1/3} \Big( \mathcal{H}^{(n)}_{t}\big(t^{2/3} x\big)+ t/24 \Big)
			\end{align}
			satisfies the $\mathbf{H}_{t}$-Brownian Gibbs property. Furthermore, for any interval $(a,b)\subset \R$ and $\e>0$, there exists $\delta>0$ such that, for all $t\ge 1$,
			\begin{align*}
				\Pr\left(Z_{\mathbf{H}_t}^{1,1,(a,b),\h_t^{(1)}(a),\h_t^{(1)}(b),\infty,\h_t^{(2)}}<\delta\right)\le \e,
			\end{align*}
			where $Z_{\mathbf{H}_t}^{1,1,(a,b),\h_t^{(1)}(a),\h_t^{(1)}(b),\infty,\h_t^{(2)}}$ is the normalizing constant defined in \eqref{eq:bgibbs}.
		\end{enumerate}
	\end{proposition}

	\begin{remark} In part (3) of Theorem 2.15 \cite{CH16} it is erroneously mentioned that the scaled KPZ line ensemble satisfies $\mathbf{H}_t$-Brownian Gibbs property with $\mathbf{H}_t(x)=e^{t^{1/3}x}$ (instead of $\mathbf{H}_t(x)=t^{2/3}e^{t^{1/3}x}$ from \eqref{eq:Ht}). This error was reported by Milind Hegde and has been acknowledged by the authors of \cite{CH16}, who are currently preparing an errata for the same.
	\end{remark}
	
	More recently, it has also been shown in \cite{dm21} that the KPZ line ensemble as defined in Proposition \ref{line-ensemble} is unique as well. We will make extensive use of this scaled KPZ line ensemble $\h_t^{(n)}(x)$ in our proofs in later sections. For $n=1$, we also adopt the shorthand notation: 
	\begin{align}\label{eq:htx}
		\h_t(x):=\h_t^{(1)}(x)= t^{-1/3}\left(\calH(t^{2/3}x,t)+\tfrac{t}{24}\right).
	\end{align}
	Note that for $t$ large, the Radon-Nikodym derivative in \eqref{eq:bgibbs} attaches heavy penalty if the curves are not ordered. Thus, intuitively at $t\to \infty$, one expects to get completely ordered curves, where the $\mathbf{H}_t$-Brownian Gibbs property will be replaced by the usual Brownian Gibbs property (see Definition \ref{def:bgp}) for non-intersecting Brownian bridges. Thus it's natural to expect the scaled KPZ line ensemble to converge to the Airy line ensemble. Along with the recent progress on the tightness of KPZ line ensemble \cite{wu21} and characterization of Airy line ensemble \cite{dm20}, this remarkable result has been recently proved in \cite{qs20}. 
	
	\begin{proposition}[Theorem 2.2 (4) in \cite{qs20}] \label{p:leconv} Consider the KPZ line ensemble and the Airy line ensemble defined in Proposition \ref{line-ensemble} and Definition \ref{def:ale} respectively. For any $k\ge 1$, we have
		\begin{align*}
			(2^{1/3}\h_t^{(i)}(2^{1/3}x))_{i=1}^k \stackrel{d}{\to} (\calA_i(x))_{i=1}^k,
		\end{align*}
		in the uniform-on-compact topology.
	\end{proposition}
	
	The $2^{1/3}$ factor in Proposition \ref{p:leconv} corrects the different diffusion coefficient used when we define the Brownian Gibbs property and $\mathbf{H}_t$ Brownian Gibbs property. We end this section by recalling several known results and tail estimates for the scaled KPZ equation with narrow wedge initial data.
	
	\begin{proposition}\label{p:kpzeq} Recall $\h_t(x)$ from \eqref{eq:htx}. The following results hold:
		\begin{enumerate}[label=(\alph*), leftmargin=15pt]
			\item \label{p:stat} For each $t>0$, $\h_t(x)+{x^2}/{2}$ is stationary in $x$.
			\item \label{p:tail} Fix $t_0>0$. There exists a constant $\Con=\Con(t_0)>0$ such that for all $t\ge t_0$ and $m>0$ we have
			\begin{align*}
				\Pr(|\h_t(0)|\ge m) \le \Con\exp\left(-\tfrac1\Con m^{3/2}\right).
			\end{align*}
			\item \label{p:supproc} Fix $t_0>0$ and $\beta>0$. There exists  a constant $\Con=\Con({\beta},t_0)>0$ such that for all $t\ge t_0$ and $m>0$ we have
			\begin{align*}
				\Pr\left(\sup_{x\in \R} \big(\h_t(x)+\tfrac{x^2}{2}(1-\beta)\big)\ge m\right) \le \Con\exp\left(-\tfrac1{\Con} m^{3/2}\right).
			\end{align*}
		\end{enumerate}
	\end{proposition}
	The results in Proposition \ref{p:kpzeq} is a culmination of results from several papers. Part \ref{p:stat} follows from \cite[Corollary 1.3 and Proposition 1.4]{acq}. The one-point tail estimates for KPZ equation are obtained in \cite{utail,ltail}. One can derive part \ref{p:tail} from those results or can combine the statements of Proposition 2.11 and 2.12 in \cite{cgh} to get the same. Part \ref{p:supproc} is Proposition 4.2 from \cite{cgh}. 
	\section{Uniqueness and convergence of random modes} \label{sec:rm}
	In this section we prove the uniqueness of random modes that appears in Theorems \ref{t:main} and \ref{t:main2} and prove Theorem \ref{t:favpt} which claims the convergences of random modes to appropriate limits. 
	The following lemma settles the uniqueness question. 
	\begin{lemma}\label{lem1} Fix $p\in (0,1)$ and $t>0$. Recall $f_{p,t}$ and $f_{*,t}$ from Theorem \ref{t:main} and \ref{t:main2}. Then $f_{*,t}$ has almost surely a unique mode $\md_{*,t}$ and $f_{p,t}$ has almost surely a unique mode $\md_{p,t}$. Furthermore for any $t_0>0$, there exist a constant $\Con(p,t_0)>0$ such that for all $t>t_0$ we have 
		\begin{align}\label{tailC}
			\Pr(t^{-2/3}|\md_{p,t}|>m)\le \Con \exp\left(-\tfrac{1}{\Con}m^{3}\right), \quad \mbox{and} \quad \Pr(t^{-2/3}|\md_{*,t}|>m)\le \Con \exp\left(-\tfrac{1}{\Con}m^{3}\right).
		\end{align}
	\end{lemma}
	\begin{proof} We first prove the point-to-point case. Fix $p\in (0,1)$ and set $q=1-p$. Take $t>0$. Throughout the proof $\Con>0$ will depend on $p$, we won't mention this further. 
		
		Note that \eqref{eq:cdrp} implies that the density $f_{p,t}(x)$ is proportional to $\calZ(0,0;x,pt)\calZ(x,pt;0,t)$ and that $\calZ(0,0;x,pt)$ and $\calZ(x,pt;0,t)$ are independent. By time reversal property of SHE we have $\calZ(x,pt;0,t) \stackrel{d}{=} \calH(x,qt)$ as functions in $x$. Using the $1:2:3$ scaling from \eqref{eq:htx} we may write 
		\begin{align}\label{def:fptx}
			f_{p,t}(x) \stackrel{d}{=} \frac1{\til{Z}_{p,t}}\exp\left(t^{1/3}p^{1/3}\h_{pt, \uparrow}(p^{-2/3}t^{-2/3}x) + t^{1/3}q^{1/3}\h_{qt, \downarrow}(q^{-2/3}t^{-2/3}x)\right)
		\end{align}
		where $\h_{t,\uparrow}(x)$ and $\h_{t,\downarrow}(x)$ are independent copies of the scaled KPZ line ensemble $\h_t(x)$ defined in \eqref{eq:htx} and $\til{Z}_{p,t}$ is the normalizing constant. Thus it suffices to study the maximizer of  
		\begin{align}
			\label{eq:hupdwn}
			\mathcal{S}_{p,t}(x):=p^{1/3}\h_{pt, \uparrow}(p^{-2/3}x) + q^{1/3}\h_{qt, \downarrow}(q^{-2/3}x).
		\end{align}
		Note that maximizer of $f_{p,t}$ can be retrieved from that of $\mathcal{S}_{p,t}$ by a $t^{-2/3}$ scaling. \\
		
		We first claim that for all $m>0$ we have
		\begin{align}\label{lm11}
			\Pr\left(\m{A}_1\right) & \le \Con\exp\left(-\tfrac1{\Con}m^3\right), \quad \mbox{where } \m{A}_1:=\left\{\h_{pt, \uparrow}(p^{-2/3}x)>\h_{pt, \uparrow}(0) \mbox{ for some }|x|> m\right\} \\ 
			\label{lm12}
			\Pr\left(\m{A}_2\right) & \le \Con\exp\left(-\tfrac1{\Con}m^3\right), \quad \mbox{where } \m{A}_2:=\left\{\h_{qt, \downarrow}(q^{-2/3}x)>\h_{qt, \downarrow}(0) \mbox{ for some }|x|> m\right\}.
		\end{align}
		Let us prove \eqref{lm11}.  Define 
		\begin{align*}
			\m{D}_1 & : = \left\{\sup_{x \in \R} \left(\h_{pt, \uparrow}(p^{-2/3}x) + \frac{x^2}{4p^{4/3}}\right) \le \frac{m^2}{8p^{4/3}} \right\},\quad \m{D}_2 := \left\{|\h_{pt,\uparrow}(0)| \le \frac{m^2}{16p^{4/3}}\right\}.
		\end{align*}
		Note that on $\m{D}_2$, $\h_{pt,\uparrow}(0)\in [-\frac{m^2}{16p^{4/3}},\frac{m^2}{16p^{4/3}}]$, whereas on $\m{D}_1$, for all $|x|> m$ we have
		$$\h_{pt, \uparrow}(p^{-2/3}x)< \tfrac{m^2}{8p^{4/3}}-\tfrac{m^2}{4p^{4/3}} =-\tfrac{m^2}{8p^{4/3}}.$$
		Thus $\m{A}_1 \subset \neg\m{D}_1\cup \neg \m{D}_2$ where $\m{A}_1$ is defined in \eqref{lm11}. On the other hand, by Proposition \ref{p:kpzeq}\ref{p:supproc} with $\beta=\frac12$ and Proposition \ref{p:kpzeq} \ref{p:tail} we have 
		\begin{align*}
			\Pr(\m{D}_1) > 1- \Con\exp\left(-\tfrac{1}{\Con}m^{3}\right),
			\quad		\Pr(\m{D}_2) > 1- \Con\exp\left(-\tfrac{1}{\Con}m^{3}\right).
		\end{align*}
		Hence by union bound we get $\Pr(\m{A}_1) \le \Pr(\neg \m{D}_1)+\Pr(\neg\m{D}_2) \le \Con\exp(-\frac1\Con m^3)$. This proves \eqref{lm11}. Proof of \eqref{lm12} is analogous. 
		
		Now via the Brownian Gibbs property $\h_t$ is absolute continuous w.r.t.~Brownian motion on every compact interval. Hence for each $t> 0$, $\mathcal{S}_{p,t}(x)$ defined in \eqref{eq:hupdwn} has a unique maximum on any compact interval almost surely. But due to the bounds in \eqref{lm11} and \eqref{lm12}, we see that
		\begin{align}
			\label{eq:Stail}
			\Pr\left(\mathcal{S}_{p,t}(x)>\mathcal{S}_{p,t}(0) \mbox{ for some }|x|> m\right) \le \Con\exp\left(-\tfrac1\Con m^3\right).
		\end{align}
		Thus $\mathcal{S}_{p,t}(\cdot)$ has a unique maximizer almost surely. By the definitions of $f_{p,t}(x)$ and $\mathcal{S}_{p,t}(x)$ from \eqref{def:fptx} and \eqref{eq:hupdwn}, this implies $f_{p,t}(x)$ also has a unique maximizer $\md_{p,t}$ and we have that
		\begin{align}\label{eq:maxiden}
			\md_{p,t}\stackrel{d}{=}t^{2/3}\argmax_{x\in \R} \mathcal{S}_{p,t}(x).
		\end{align} 
		In view of \eqref{eq:Stail}, the above relation \eqref{eq:maxiden} leads to the first inequality in \eqref{tailC}.

		For the point-to-line case, note that via \eqref{eq:cdrp2} and \eqref{eq:htx}, $f_{*,t}(x)$ is proportional to $\exp(t^{1/3}\h_t(t^{-2/3}x))$. The proofs of uniqueness of the maximizer and the second bound in \eqref{tailC} then follow by analogous arguments. This completes the proof.
	\end{proof}
	In the course of proving the above lemma, we have also proved an important result that connects the random modes to the maximizers of the KPZ equations. We isolate this result as a separate lemma.
	\begin{lemma}\label{l:conec}
		Consider three independent copies $\calH,\calH_{\uparrow},\calH_{\downarrow}$ of the KPZ equation started from the narrow wedge initial data. The random mode $\md_{p,t}$ of $f_{p,t}$ (defined in statement of Theorem \ref{t:main}) is same in distribution as the maximizer of 
		$$\calH_{\uparrow}(x,pt)+\calH_{\downarrow}(x,qt).$$
		Similarly one has that the random mode $\md_{*,t}$ of $f_{*,t}$ (defined in statement of Theorem \ref{t:main2})  is same in distribution as the maximizer of $\calH(x,t)$.
	\end{lemma}

	\begin{proof}[Proof of Theorem \ref{t:favpt}] Due to the identity in \eqref{eq:maxiden} we see that $t^{-2/3}\md_{p,t}$ is same in distribution as
		\begin{align*}
			\argmax_{x\in \R} S_{p,t}(x)
		\end{align*}
		where $\mathcal{S}_{p,t}(x)$ is defined in \eqref{eq:hupdwn}. By Proposition \ref{p:leconv} we see that as $t\to \infty$
		$$\mathcal{S}_{p,t}(x)\stackrel{d}{\to} 2^{-1/3}\left(p^{1/3}\calA_{1, \uparrow}(2^{-1/3}p^{-2/3}x) + q^{1/3}\calA_{1, \downarrow}(2^{-1/3}q^{-2/3}x)\right)$$
		in the uniform-on-compact topology where $\calA_{1, \uparrow}, \calA_{1,\downarrow}$ are independent parabolic $\operatorname{Airy}_2$ processes. Note that the expression in the r.h.s.~of the above equation is the same as 
		\begin{align}
			\label{ax}
			\mathcal{A}(x):=2^{-1/2}\left(\calA_{\uparrow}^{(p\sqrt{2})}(x)+\calA_{\downarrow}^{(q\sqrt{2})}(x)\right)
		\end{align} where $\calA_{\uparrow}^{(p\sqrt{2})}(x),\calA_{\downarrow}^{(q\sqrt{2})}(x)$ are independent Airy sheets of index $p\sqrt{2}$ and $q\sqrt{2}$ respectively. By Lemma 9.5 in \cite{dov} we know that $\calA(x)$ has almost surely a unique maximizer on every compact set. Thus,
		\begin{align}
			\label{mxfinite}
			\argmax_{x\in [-K,K]} \mathcal{S}_{p,t}(x) \stackrel{d}{\to} \argmax_{x\in [-K,K]} \calA(x).
		\end{align}
		Finally the decay bounds for the maximizer of $\mathcal{S}_{p,t}(x)$ from Lemma \ref{lem1} and the decay bounds for the maximizer of $\calA$ from  \cite{quastel2015tails} allow us to extend the weak convergence to the case of maximizers on the full line. However, by the definition of the geodesic of the directed landscape from Definition \ref{def:geo}, we see that $\Gamma(p\sqrt{2})\stackrel{d}{=}\argmax_{x\in \R} \calA(x).$ This concludes the proof for the point-to-point case. For the point-to-line case, following Lemma \ref{l:conec} and recalling again the scaled KPZ line ensemble from \eqref{eq:htx}, we have 
		$$2^{-1/3}t^{-2/3}\md_{*,t}=\argmax_{x\in \R} \calH(2^{1/3}t^{2/3}x,t)=\argmax_{x\in \R} \left(t^{1/3}\h_t(2^{1/3}x)-\tfrac{t}{24}\right)= \argmax_{x\in \R} 2^{1/3}\h_t(2^{1/3}x).$$
		From Proposition \ref{p:leconv} we know $2^{1/3}\h_t(2^{1/3}x)$ converges in distribution to $\calA_1(x)$ in the uniform-on-compact topology. Given the decay estimates for $\md_{*,t}$ from \eqref{tailC} and decay bounds for the maximizer of $\calA_1$ from \cite{dov}, we thus get that $\argmax_{x\in \R} 2^{1/3}\h_t(2^{1/3}x)$ converges in distribution to $\mathcal{M}$, the unique maximizer of the parabolic $\operatorname{Airy}_2$ process. This completes the proof.
	\end{proof}
	\section{Decomposition of Brownian bridges around joint maximum}\label{sec:localmax}
	
	The goal of this section is to prove certain decomposition properties of Brownian bridges around the joint maximum outlined in Proposition \ref{p:bbdecomp} and Proposition \ref{propA}. To achieve this goal, we first discuss several Brownian objects and their related properties in Section \ref{sec:brwn} which will be foundational for the rest of the paper. Then we prove Proposition \ref{p:bbdecomp} and \ref{propA} in the ensuing subsection. We refer to Figure \ref{fig:laws} for the structure and various Brownian laws convergences in this and the next sections. The notation $p_t(y):=(2\pi t)^{-1/2}e^{-y^2/(2t)}$ for the standard heat kernel will appear throughout the rest of the paper. 
	
		\begin{center}
	\small
  \begin{figure}[h]
  \begin{tikzpicture}[auto,
  	block_main/.style = {rectangle, draw=black, thick, fill=white, text width=17em, text centered, minimum height=4em, font=\small},
	block_density/.style = {rectangle, draw=black, fill=white, thick, text width=11em, text centered, minimum height=4em},
	block_rewrite/.style = {rectangle, draw=black, fill=white, thick, text width=17em, text centered, minimum height=4em},
	block_kernels/.style = {rectangle, draw=black, fill=white, thick, text width=19em, text centered, minimum height=4em},
        line/.style ={draw, thick, -latex', shorten >=0pt}]
		\node [block_rewrite] (nibm) at (3,-5.5) {Non-intersecting Brownian motions ($\nonintbm$) defined in Defintion \ref{def:nibm}.};
		\node [block_kernels] (nibb) at (3,-8.5) {Non-intersecting Brownian bridges ($\nonintbb$) defined in Defintion \ref{def:nibb}.};
		\node [block_density] (dbm) at (-5.7,-8.5) {Dyson Brownian motion ($\dbm$) defined in Definition \ref{def:dbm}};
		\node [block_density] (bes) at (-5.7,-11.5) {3D Bessel process (Definition \ref{def:bessel})};
		\node [block_density] (bb) at (3,-11.5) {Bessel bridges (Definition \ref{def:bbridge})};
		\node [block_density] (nirw) at (-5.7,-5.5) {Non-intersecting random walks};
    \begin{scope}[every path/.style=line]
		\path (nirw) -- node[above] {Diffusive limit} node[below] {Lemma \ref{propF}} (nibm);
		\path (nibm) --  node[right] {(Conditioning)} (nibb);
		\path (nibb) --node[above] {Diffusive limit} node[below] {Proposition \ref{lemmaC}} (dbm);
		\path (dbm) -- node[right] {(Taking Differences)} node[left] {Lemma \ref{l:dtob}} (bes);
		\path (bb) -- node[above] {Diffusive limit} node[below] {Corollary \ref{lemmaC1}} (bes);
		\path (nibb) -- node[right] {(Taking Differences)} node[left] {Lemma \ref{l:nibtobb}} (bb);
    \end{scope}
  \end{tikzpicture}
  \noindent\caption{Relationship between different laws used in Sections \ref{sec:localmax} and \ref{sec:bbnibb}.}
  \label{fig:laws}
  \end{figure}
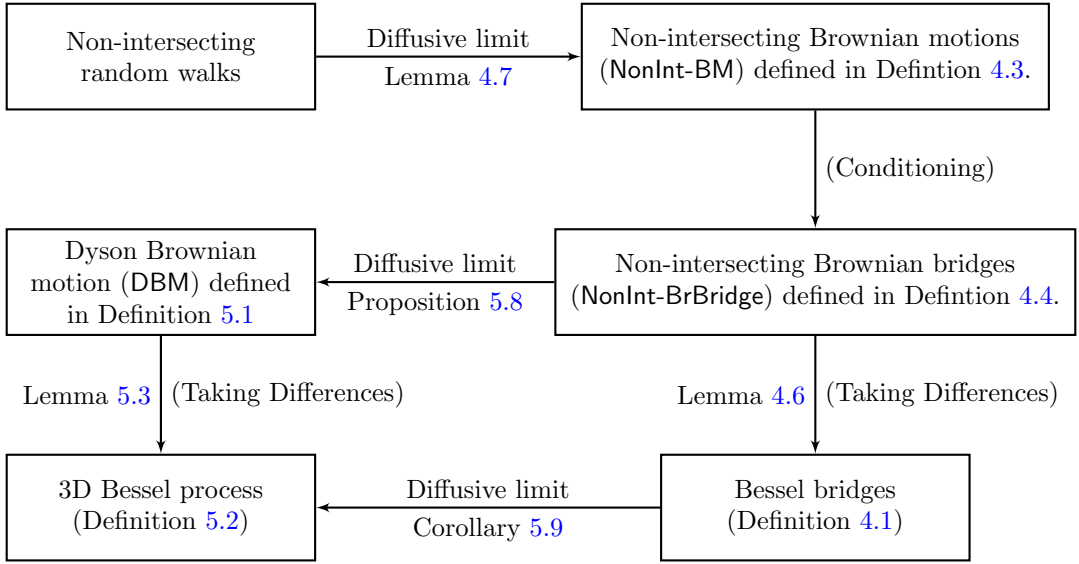
\end{center}
	
	\subsection{Brownian objects} \label{sec:brwn} In this section we recall several objects related to Brownian motion, including the Brownian meanders, Bessel bridges, non-intersecting Brownian motions and non-intersecting Brownian bridges.
	\begin{definition}[Brownian meanders and Bessel bridges] \label{def:bbridge} Given a standard Brownian motion $B(\cdot)$ on $[0,1]$, a standard Brownian meander $\bme: [0,1]\to \R$ is a process defined by
		$$\bme(x)=(1-\theta)^{-\frac12}|B(\theta+(1-\theta)x)|, \quad x\in [0,1],$$ 
		where $\theta=\sup\{x\in [0,1]\mid B(x)=0\}$.  In general, we say a process $\bme:[a,b]\to \R$ is a Brownian meander on $[a,b]$ if 
		$$\bme'(x):=(b-a)^{-\frac12}\bme(a+x(b-a)), \quad x\in [0,1]$$
		is a standard Brownian meander. A Bessel bridge $\bb$ on $[a,b]$ ending at $y>0$ is a Brownian meander $\bme$ on $[a,b]$ subject to the condition (in the sense of Doob)  $\bme(b)=y$.
	\end{definition}	
	A Bessel bridge can also be realized as conditioning a 3D Bessel process to end at some point and hence the name. As we will not make use of this fact, we do not prove this in the paper.
	
	\begin{lemma}[Transition densities for Bessel Bridge] \label{l:tdbes} Let $V$ be a Bessel bridge on $[0,1]$ ending at $a$. Then for $0<t<1$,
		\begin{align*}
			\Pr(V(t)\in d x)=\frac{x}{at}\frac{p_t(x)}{p_1(a)}[p_{1-t}(x-a)-p_{1-t}(x+a)] dx, \quad x\in [0,\infty).
		\end{align*}
		For $0<s<t<1$ and $x>0$,
		\begin{align*}
			\Pr(V(t)\in d y\mid V(s)=x)=\frac{[p_{t-s}(x-y)-p_{t-s}(x+y)][p_{1-t}(y-a)-p_{1-t}(y+a)]}{[p_{1-s}(x-a)-p_{1-s}(x+a)]}dy,\quad y\in [0,\infty).
		\end{align*}
	\end{lemma}
	\begin{proof} We recall the joint density formula for Brownian meander $W$ on $[0,1]$ from \cite{igle}. For $0=t_0<t_1<t_2<\cdots<t_k\le 1$:
		\begin{align*}
			\Pr(W(t_1)\in dx_1, \ldots, W(t_k)\in dx_k)= \frac{x_1}{t}p_{t_1}(x_1)\Psi(\tfrac{x_k}{\sqrt{1-t_k}})\prod_{j=1}^{k-1} g(x_j,x_{j+1};t_{j+1}-t_j) \prod_{j=1}^k dx_j
		\end{align*}
		where
		\begin{align*}
			& g(x_j,x_{j+1};t_{j+1}-t_j):=[p_{t_{j+1}-t_j}(x_j-x_{j+1})-p_{t_{j+1}-t_j}(x_j+x_{j+1})], \\ & \Psi(x):=(2/\pi)^{\frac12}\int_0^x e^{-u^2/2}du.
		\end{align*}
		The joint density is supported on $[0,\infty)^k$. We now use Doob-$h$ transform to get the joint density  for Bessel bridge on $[0,1]$ ending at $a$. For $0=t_0<t_1<t_2<\cdots<t_k<1$:
		\begin{align*}
			\Pr(V(t_1)\in dx_1, \ldots, V(t_k)\in dx_k)= \frac{x_1}{at_1}\frac{p_{t_1}(x_1)}{p_1(a)}\prod_{j=1}^{k} g(x_j,x_{j+1};t_{j+1}-t_j) \prod_{j=1}^k dx_j
		\end{align*}
		where $x_{k+1}=a$ and $t_{k+1}=1$.  Formulas in Lemma \ref{l:tdbes} is obtained easily from the above joint density formula.
	\end{proof}

	\begin{definition}[Non-intersecting Brownian motions] \label{def:nibm} We say a random continuous function $W(t)=(W_1(t),W_2(t)) : [0,1]\to \R^2$ is a pair of non-intersecting Brownian motion ($\nonintbm$ in short) if its distribution is given by the following formulas:
		\begin{enumerate}[label=(\alph*), leftmargin=15pt]
			\item We have for any $y_1, y_2 \in \R$
			\begin{align}\label{nibm1}
				\Pr(\W_1(1)\in \d y_1, \W_2(1)\in \d y_2) = \frac{\ind\{y_1>y_2\}(y_1 - y_2)p_1(y_1)p_1(y_2)}{\int_{r_1 > r_2}(r_1 - r_2)p_1(r_1)p_1(r_2) \d r_1 \d r_2}\d y_1 \d y_2.
			\end{align}
			
			\item For $0 < t < 1$, we have 
			\begin{equation}
				\label{nibm2}
				\begin{aligned}
					& \Pr(\W_1(t)\in \d y_1, \W_2(t)\in \d y_2) \\ & =\frac{\ind\{y_1>y_2\}(y_1 - y_2)p_t(y_1)p_t(y_2)\int_{r_1 > r_2}\det(p_{1-t}(y_i - r_j))_{i, j = 1}^2\d r_1 \d r_2}{t\int_{r_1 > r_2}(r_1 - r_2)p_1(r_1)p_1(r_2) \d r_1 \d r_2} \d y_1 \d y_2. 
				\end{aligned}
			\end{equation}
			\item For $0 <s< t \le1$ and $x_1 > x_2$, we have 
			\begin{equation}
				\label{nibm3}
				\begin{aligned}
					& \Pr(\W_1(t)\in \d y_1, \W_2(t)\in \d y_2|\W_1(s) = x_1, \W_2(s) = x_2)\\ & = \ind\{y_1>y_2\}\frac{\det(p_{t - s}(y_i - x_j))_{i, j = 1}^2\int_{r_1 > r_2}\det(p_{1-t}(y_i - r_j))_{i, j = 1}^2 \d r_1 \d r_2 }{\int_{r_1 > r_2}\det(p_{1-s}(x_i - r_j))_{i, j = 1}^2 \d r_1 \d r_2 }\d y_1 \d y_2.
				\end{aligned}
			\end{equation}
		\end{enumerate}
		We call $W^{[0,M]}$ a $\nonintbm$ on $[0,M]$ if $(M^{-1/2}W_1^{[0,M]}(Mx),M^{-1/2}W_2^{[0,M]}(Mx))$	is a $\nonintbm$ on $[0,1]$.
	\end{definition}
	\begin{definition}[Non intersection Brownian bridges] \label{def:nibb}
		A $2$-level non-intersecting Brownian bridge ($\nonintbb$ in short) $V=(V_1,V_2)$ on $[0,1]$ ending at $(z_1,z_2)$ with $(z_1\neq z_2)$ is a $\nonintbm$ on $[0,1]$ defined in Definition \ref{def:nibm} subject to the condition (in the sense of Doob) $V(1)=z_1, V(1)=z_2$. It is straight forward to check we have the following formulas for the distribution of $\V$:
		\begin{enumerate}[label=(\alph*), leftmargin=15pt]
			\item For $0 < t < 1$, we have 
			\begin{align*}
				\Pr(\V_1(t)\in \d y_1, \V_2(t)\in \d y_2) = \frac{\ind\{y_1>y_2\}(y_1 - y_2)p_t(y_1)p_t(y_2)}{t(z_1-z_2)p_1(z_1)p_1(z_2)}\det(p_{1-t}(y_i - z_j))_{i,j =1}^2 \d y_1 \d y_2.
			\end{align*}
			\item For $0 <s< t \le1$ and $x_1 > x_2$, we have 
			\begin{equation*}
				\begin{aligned}
					& \Pr(\V_1(t)\in \d y_1, \V_2(t)\in \d y_2|\V_1(s) = x_1, \V_2(s) = x_2) \\ & \hspace{2cm}= \frac{\det(p_{t-s}(y_i - x_j))_{i,j = 1}^2\det(p_{1-t}(y_i - z_j))_{i, j =1}^2}{\det(p_{1-s}(x_i - z_j))_{i,j =1}^2}\d y_1 \d y_2.
				\end{aligned}
			\end{equation*}
		\end{enumerate}
		Just like $\nonintbm$, we call $V^{[0,M]}$ a $\nonintbb$ on $[0,M]$ if $(\frac1{\sqrt{M}}V_1^{[0,M]}(Mx),\frac1{\sqrt{M}}V_2^{[0,M]}(Mx))$	is a $\nonintbb$ on $[0,1]$.
	\end{definition}
	
	\begin{remark} It is possible to specify the distributions for a $n$-level non-intersecting Brownian bridge. However, the notations tend to get cumbersome due to the possibility of some paths sharing the same end points. We refer to Definition 8.1 in \cite{reu} for a flavor of such formulas. We remark that in this paper we will focus exclusively on the $n=2$ case with distinct endpoints. 
	\end{remark}
	
	The following Lemma connects $\nonintbb$ with Bessel bridges.
	
	\begin{lemma}[$\nonintbb$ to Bessel bridges]\label{l:nibtobb} Let $V=(V_1,V_2)$ be a $\nonintbb$ on $[0,1]$ ending at $(z_1,z_2)$ with $z_1>z_2$. Then, as functions in $x$, we have $V_1(x)-V_2(x)\stackrel{d}{=} \sqrt2\bb(x)$ where $\bb:[0,1]\to \R$ is a Bessel bridge (see Definition \ref{def:bbridge}) ending at $(z_1-z_2)/\sqrt{2}$.
	\end{lemma}
	
	The proof of Lemma \ref{l:nibtobb} is based on the following technical lemma that discusses how $\nonintbm$ comes up as a limit of non-intersecting random walks.
	
	\begin{lemma}\label{propF}
		{Let $X_j^{i}$ be i.i.d. $\operatorname{N}(0,1)$ random variables. Let $S_0^{(i)} = 0$ and $S_k^{(i)} = \sum_{j= 1}^{k}X_j^{i}.$ Consider $Y_n(t) = (Y_{n,1}(t),Y_{n,2}(t)) := (\frac{S_{nt}^{(1)}}{\sqrt{n}}, \frac{S_{nt}^{(2)}}{\sqrt{n}})$ an $\R^2$ valued process on $[0,1]$ where the in-between points are defined by linear interpolation. Then conditioned on the non-intersecting event $\Lambda_n:= \cap_{j=1}^n\{ S_j^{(1)} > S_j^{(2)}\},$
			$Y_n \stackrel{d}{\to} \W$, where $\W(t) = (\W_1(t), \W_2(t))$ is distributed as $\nonintbm$ defined in Definition \ref{def:nibm}.}
	\end{lemma}
	
	We defer the proof of this lemma to the Appendix as it roughly follows standard calculations based on the Karlin-McGregor formula \cite{karlin1959coincidence}.

	\begin{proof}[Proof of Lemma \ref{l:nibtobb}] Let $X_j^{i}$ to be i.i.d.~$\operatorname{N}(0,1)$ random variables. Let $S_0^{(i)} = 0$ and $S_k^{(i)} = \sum_{j= 1}^{k}X_j^{i}.$ Set $Y_n(t) = (\frac{S_{nt}^{(1)}}{\sqrt{n}}, \frac{S_{nt}^{(2)}}{\sqrt{n}})$ an $\R^2$ valued process on $[0,1]$ where the in-between points are defined by linear interpolation. By Lemma \ref{propF}, conditioned on the non-intersecting event $\Lambda_n:= \cap_{j=1}^n\{ S_j^{(1)} > S_j^{(2)}\},$ $Y_n$ converges to $\W=(\W_1,\W_2)$, a $\nonintbm$ on $[0,1]$ defined in Definition \ref{def:nibm}. On the other hand, classical results from \cite{igle} tell us, $(S_{nt}^{(1)}-S_{nt}^{(2)})/\sqrt{n}$ conditioned on $\Lambda_n$ converges weakly to $\sqrt2\bme(t)$, where $\bme(\cdot)$ is a Brownian meander defined in Definition \ref{def:bbridge}. The $\sqrt2$ factor comes because $S_k^{(1)}-S_{k}^{(2)}$ is random walk with variance $2$. Thus $$W_1(\cdot)-W_2(\cdot)\stackrel{d}{=} \sqrt2\bme(\cdot).$$  From \cite{igle}, $\bme$ is known to be Markov process. Hence the law of $W_1-W_2$ depends on $(W_1(1),W_2(1))$ only through $W_1(1)-W_2(1)$. In particular conditioning on $(W_1(1)=z_1,W_2(1)=z_2)$, for any $z_1>z_2$, makes $\W$ to be a $\nonintbb$ on $[0,1]$ ending at $(z_1,z_2)$ and the conditional law of $\frac1{\sqrt2}(W_1-W_2)$ is then a Bessel bridge ending at $\frac1{\sqrt2}(z_1-z_2)$. This completes the proof.
	\end{proof}
	
	\subsection{Decomposition Results} In this section we prove two path decomposition results around the maximum for a single Brownian bridge and for a sum of two Brownian bridges. The first one is for a single Brownian bridge.
	
	\begin{proposition}[Bessel bridge decomposition] \label{p:bbdecomp}
		Let $\bar{B}:[a,b]\to \R$ be a Brownian bridge with $\bar{B}(a)=x$ and $\bar{B}(b)=y$. Let $M$ be the almost sure unique maximizer of $\bar{B}$. Consider
		$\wl: [a,M]\to \R$ defined as $\wl(x)=\bar{B}(M)-\bar{B}(M+a-x)$, and $\wr: [M,b]\to \R$ defined as $\wr(x)=\bar{B}(M)-\bar{B}(x)$. Then, conditioned on $(M,\bar{B}(M))$,
		\begin{enumerate}[label=(\alph*), leftmargin=15pt]
			\item $\wl(\cdot)$ and $\wr(\cdot)$ are independent.
			\item $\wl(\cdot)$ is a Bessel bridge on $[a,M]$ starting at zero and ending at $\bar{B}(M)-x$.
			\item $\wr(\cdot)$ is a Bessel bridge on $[M,b]$ starting at zero and ending at $\bar{B}(M)-y$.
		\end{enumerate}	
		{Recall that the Bessel bridges are defined in Definition \ref{def:bbridge}.}
	\end{proposition}
	
	\begin{remark} There is a technical issue in considering the regular conditional distribution of $\wl$, $\wr$ separately as the objects are defined on intervals of random length. Instead we should always view $\wl$, $\wr$ appended together as one random function defined on the deterministic interval $[a,b]$, as done in our proofs in Sections \ref{sec:ergbb} and \ref{sec:pfmain}. However here in Proposition \ref{p:bbdecomp} (as well as in Proposition \ref{propA}), we state their distributions separately for simplicity. 
	\end{remark}
	
	\begin{proof} We will prove the result for $a=0$, $b=1$ and $x=0$; the general case then follows from Brownian scaling and translation property of bridges.  We recall the classical result of Brownian motion decomposition around maximum from \cite{denisov}. Consider a map $\Upsilon : C([0,1]) \to C([0,1])\times C([0,1])$ given by
		\begin{align*}
			(\Upsilon f)_-(t) & :=
			M^{-\frac12}[f(M)-f(M(1-t))], \quad t\in [0,1], \\ (\Upsilon f)_+(t) & :=(1-M)^{-\frac12}[f(M)-f(M+(1-M)t)],  \quad t\in [0,1],
		\end{align*}
		where $M=M(f):=\inf\{t\in [0,1] \mid f(s)\le f(t), 0\le s\le 1\}$ is the left-most maximizer of $f$. We set $(\Upsilon f)_-\equiv (\Upsilon f)_+\equiv 0$ if $M=0$ or $M=1$. We also define
		\begin{align*}
			(\Upsilon^M f)_{{-}}(t) & := M^{1/2}(\Upsilon f)_-(\tfrac{t}M), \quad t\in [0,M], \\
			(\Upsilon^M f)_{{+}}(t) & :=(1-M)^{\frac12}(\Upsilon f)_+(\tfrac{t-M}{1-M}),  \quad t\in [M,1].
		\end{align*}
		Given a standard Brownian motion $B$ on $[0,1]$, by Theorem 1 in \cite{denisov},  $\Upsilon(B)$ is independent of $M=M(B)$ and $\Upsilon(B)_-$ and $\Upsilon(B)_+$ are independent Brownian meanders on $[0,1]$. By the Brownian scaling and the fact that $\Upsilon (B)$ is independent of $M(B)$, conditioned on $M(B)$,  we see that $(\Upsilon^M B)_-$ and $(\Upsilon^M B)_+$ are independent Brownian meanders on $[0,M]$ and $[M,1]$ respectively. Observe that $(\Upsilon^M {f})_-(M)={f}(M)$ and $(\Upsilon^M {f})_+(1)={f}(M)-{f}(1)$ for any $f\in C([0,1])$.  Thus conditioning on $(B(M)=v, B(1)=y)$ is equivalent to conditioning on $((\Upsilon^M {B})_-(M)=v, (\Upsilon^M {B})_+(1)=v-y)$. By definition the conditional law of Brownian meanders upon conditioning their end points are Bessel bridges. Thus conditioning on $(M=m,B(M)=v,B(1)=y)$, we see that $(\Upsilon^M B)_-$ and $(\Upsilon^M B)_+$ are independent Bessel bridges on $[0,M]$ and $[M,1]$ ending at $v$ and $v-y$ respectively. But the law of a Brownian motion conditioned on $(M=m,B(M)=v,B(1)=y)$ is same as the law of a Brownian bridge $\bar{B}$ on $[0,1]$ starting at $0$ and ending at $y$, conditioned on $(M(\bar{B})=m,\bar{B}(M)=v)$.  Identifying $(\Upsilon^M \bar{B})_-$ and $(\Upsilon^M \bar{B})_+$ with $\wl$ and $\wr$ gives us the desired result.
	\end{proof}

	The next proposition show that for two Brownian bridges the decomposition around the joint maximum is given by non-intersecting Brownian bridges.

	\begin{proposition}[Non-intersecting Brownian bridges decomposition]\label{propA}
		Let $\bar{B}_1, \bar{B}_2: [a, b]\rightarrow \R$ be independent Brownian bridges such that $\bar{B}_i(a) = x_i, \bar{B}_i(b) = y_i$. Let $M$ be the almost sure unique maximizer of $(\bar{B}_1(x)+\bar{B}_2(x))$ on $[0,1]$. Define $\bar{V}_{\ell}(x) : [0,M-a]\to \R^2$ and $\bar{V}_{r}: [0,b-M]\to \R^2$ as follows:
		\begin{align*}
			\bar{V}_{\ell}(x) & := (\bar{B}_1(M)- \bar{B}_1(M-x), - \bar{B}_2(M)+\bar{B}_2(M-x))\\
			\bar{V}_{r}(x) & := (\bar{B}_1(M)- \bar{B}_1(M+x), - \bar{B}_2(M)+\bar{B}_2(M+x))
		\end{align*}
		Then, conditioned on $(M,\bar{B}_1(M),\bar{B}_2(M))$,
		\begin{enumerate}[label=(\alph*), leftmargin=15pt]
			\item $\bar{V}_{\ell}(\cdot)$ and $\bar{V}_{r}(\cdot)$ are independent.
			\item $\bar{V}_{\ell}(\cdot)$ is a $\nonintbb$ on $[0,M-a]$ ending at $(\bar{B}_1(M)- x_1, x_2-\bar{B}_2(M))$.
			\item $\bar{V}_{r}(\cdot)$ is a $\nonintbb$ on $[0,b-M]$ ending at $(\bar{B}_1(M)- y_1, y_2- \bar{B}_2(M))$.
		\end{enumerate}	
		Recall that $\nonintbb$s are defined in Definition \ref{def:nibb}.
	\end{proposition}
	
	As in the proof of Proposition \ref{p:bbdecomp}, to prove Proposition \ref{propA} we rely on a decomposition result for Brownian motions instead. To state the result we introduce the $\Omega$ map which encodes the trajectories of around the joint maximum of the sum of two functions.
	
	\begin{definition}\label{def:omega}
		For any $f = (f_1, f_2) \in C([0,1]\rightarrow \R^2),$ we define $\Omega f \in C([-1,1]\rightarrow \R^2)$ as follows: 
		\begin{equation*}
			(\Omega f)_1(t) = \begin{cases}
				[f_1(M) - f_1(M(1+t))]M^{-1/2} & -1 \le t \le 0\\
				[f_1(M) - f_1(M+ (1-M)t)](1-M)^{-1/2} & 0 \le t \le 1
			\end{cases}
		\end{equation*} 
		\begin{equation*}
			(\Omega f)_2(t) = \begin{cases}
				-[f_2(M) - f_2(M(1+t))]M^{-1/2} & -1 \le t \le 0\\
				-[f_2(M) - f_2(M+ (1-M)t)](1-M)^{-1/2} & 0 \le t \le 1
			\end{cases}
		\end{equation*} 
		where $M = \inf\{t \in [0,1]: f_1(s)+ f_2(s) \le f_1(t)+f_2(t), \forall s \in [0,1]\}$ is the left most maximizer. We set $(\Omega f)\equiv (0,0)$ if $M = 0$ or $1$ on $[0,1].$  With this we define two functions in $C([0,1]\to\R^2)$ as follows
		\begin{align*}
			(\Omega f)_{+}(x) &:= ((\Omega f)_1(x),(\Omega f)_2(x)), \quad x\in [0,1] \\
			(\Omega f)_{-}(x) &:= ((\Omega f)_1(-x),(\Omega  f)_2(-x)), \quad x\in [0,1].
		\end{align*}
	\end{definition}
	
	We are now ready to state the corresponding result for Brownian motions.
	
	\begin{lemma}\label{lemmaA1}
		Suppose $\B = (\B_1, \B_2)$ are independent Brownian motions on $[0,1]$ with $\B_i(0) = x_i. $ Let $$M = \argmax_{t \in [0,1]}(\B_1(t) + \B_2(t)).$$ Then $(\Omega \B)_{+}, (\Omega \B)_{-}$ are independent and distributed as non-intersecting Brownian motions on $[0,1]$ (see Definition \ref{def:nibm}). Furthermore, $(\Omega \B)_+, (\Omega \B)_{-}$ are independent of $M$.
	\end{lemma}
	
	We first complete the proof of Proposition \ref{propA} assuming the above Lemma.

	\begin{proof}[Proof of Proposition \ref{propA}]
		Without loss of generality, we set $a = 0$ and $b=1$. Let $\B_1, \B_2: [0,1]\rightarrow \R$ be two independent Brownian bridges with $\B_i(0) = x_i$ and denote $M= \argmax_{x\in [0, 1]}\B_1(x) +\B_2(x).$ Consider
		\begin{align*}
			\V_{\ell}(x) & := (\B_1(M)- \B_1(M-x), - \B_2(M)+\B_2(M-x))_{x \in [0, M]} \\
			\V_{r}(x) & := (\B_1(M)- \B_1(M+x), - \B_2(M)+\B_2(M+x))_{x \in [0, 1-M]}
		\end{align*}
		By Lemma \ref{lemmaA1}, conditioned on $M$ {and after Brownian re-scaling}, we have
		where   $\V_r, \V_{\ell}$ are conditionally independent and $\V_r\sim W^{[0,1-M]}$ and $\V_{\ell} \sim W^{[0,M]}$ where $W^{[0,\rho]}$ denote a $\nonintbm$ on $[0,\rho]$ defined in Definition \ref{def:nibb}. To convert the above construction to Brownian bridges, we observe that the map
		$$(\B_1(M), \B_2(M), \B_1(1), \B_2(1))\leftrightarrow (\V_{r}(1-M), \V_{\ell}(M))$$ is bijective. Indeed, we have that 
		\begin{equation*}
			\begin{pmatrix}
				\B_1(M) = b_1,\ \B_2(M)=b_2 \\\B_1(1)= y_1,\ \B_2(1)=y_2
			\end{pmatrix}
			\Leftrightarrow
			\begin{pmatrix}
				\V_{r}(1-M) = (b_1-y_1, -b_2+y_2)\\\V_{\ell}(M)= (b_1-x_1, -b_2+x_2)
			\end{pmatrix}.
		\end{equation*}
		Thus conditioned on $(M=m, \B_{i}(M) = b_i, \B_i(1)= y_i)$, $\V_{r}(\cdot)$ is now a $\nonintbb$ Brownian bridge on $[0, 1-m]$ ending at $(b_1- y_1, -b_2+y_2)$ and $\V_{\ell}(\cdot)$ is a $\nonintbb$ on $[0,m]$ ending at $(b_1-x_1, -b_2+x_2)$ where both are conditionally independent of each other. But the law of a Brownian motions conditioned on $(M=m, \B_{i}(M) = b_i,\B_i(1)= y_i)$ is same as the law of a Brownian bridges $\bar{B}$ on $[0,1]$ starting at $(x_1,x_2)$ and ending at $(y_1,y_2)$, conditioned on $(M=m, \bar{B}_{i}(M) = b_i)$. Thus this leads to the desired decomposition for Brownian bridges.
	\end{proof}
	
	Let us now prove Lemma \ref{lemmaA1}. The proof of Lemma \ref{lemmaA1} follows similar ideas from \cite{denisov} and \cite{igle}. To prove such a decomposition holds, we first show it at the level of random walks. Then we take diffusive limit to get the same decomposition for Brownian motions. 
	
	\begin{proof}[Proof of Lemma \ref{lemmaA1}]
		Let $X_j^{(i)}\stackrel{i.i.d.}{\sim} \operatorname{N}(0,1), \ i = 1, 2, \  j \ge 1$ and set $S_k^{(i)} = \sum_{j= 1}^n X_j^{(i)}.$  Define $$M_n := \argmax_{k=1}^n\{S_k^{(1)}+S_k^{(2)}\},$$ and let ${A}_j^{(i)}$ be subsets of $\R$. Define
		\begin{align}\label{LA1}
			\mathcal{I}:=\Pr\bigg(\bigcap_{\substack{j=0\\i = 1, 2}}^{k-1}\{S_k^{(i)} - S_j^{(i)}\in (-1)^{i+1}A_{k-j}^{(i)}\} \cap \bigcap_{\substack{j=k+1\\i = 1, 2}}^{n}\{S_k^{(i)} - S_j^{(i)}\in (-1)^{i+1}A_{j}^{(i)}\} \cap \{M_n = k\} \bigg).
		\end{align}
		Noting that the  event $\{M_n = k\} $ is the same as $$\bigcap_{\substack{j=0 \\i = 1, 2}}^{k-1}\{S_k^{(1)}+ S_k^{(2)}> S_j^{(1)}+S_j^{(2)}\}\bigcap_{\substack{j=k+1\\i = 1, 2}}^{n}\{S_k^{(1)}+ S_k^{(2)}> S_j^{(1)}+S_j^{(2)}\},$$
		we have 
		\begin{align*}
			\text{ r.h.s of  }\eqref{LA1} & = \Pr\bigg(\bigcap_{\substack{j=0\\i = 1, 2}}^{k-1}\{S_k^{(i)} - S_j^{(i)}\in (-1)^{i+1}A_{k-j}^{(i)}, S_k^{(1)}+ S_k^{(2)}> S_j^{(1)}+S_j^{(2)}\}\notag \\ & \hspace{2cm}\cap \bigcap_{\substack{j=k+1\\i = 1, 2}}^{n}\{S_k^{(i)} - S_j^{(i)}\in (-1)^{i+1}A_{j}^{(i)},  S_k^{(1)}+ S_k^{(2)}> S_j^{(1)}+S_j^{(2)}\}\bigg).
		\end{align*}   
		We also observe that the pairs $(S_k^{(1)}- S_j^{(1)}, S_k^{(2)}- S_j^{(2)})_{j =0}^{k-1}$ and $(S_k^{(1)}- S_j^{(1)}, S_k^{(2)}- S_j^{(2)})_{j =k+1}^{n}$ are independent of each other and as $X_j^{i}$ is symmetric
		\begin{align*}
			& (S_k^{(1)}- S_j^{(1)}, S_k^{(2)}- S_j^{(2)})_{j =0}^{k-1} \stackrel{(d)}{=}(S_{k-j}^{(1)}, - S_{k-j}^{(2)})_{j=0}^{k-1} \\ & (S_k^{(1)}- S_j^{(1)}, S_k^{(2)}- S_j^{(2)})_{j =k+1}^{n} \stackrel{(d)}{=}(S_{j-k}^{(1)}, - S_{j-k}^{(2)})_{j=k+1}^{n}.
		\end{align*}
		Thus, 
		\begin{align}\label{LA4}
			\mathcal{I} =\Pr\bigg(\bigcap_{\substack{j=0\\i = 1, 2}}^{k-1}\{S_j^{(i)}\in A_{j}^{(i)}, S_j^{(1)}>S_j^{(2)}\}\bigg) \cdot \Pr\bigg( \bigcap_{\substack{j=1\\i = 1, 2}}^{n-k}\{S_j^{(i)}\in A_{j}^{(i)}, S_j^{(1)}>S_j^{(2)}\}\bigg).
		\end{align}
		Based on \eqref{LA4}, we obtain that 
		\begin{equation}
			\label{LA5}
			\begin{aligned}
				\frac{\mathcal{I}}{\Pr(M_n = k)} & =\Pr\bigg(\bigcap_{\substack{j=0\\i = 1, 2}}^{k-1}\{S_j^{(i)}\in A_{j}^{(i)}\}|\bigcap_{j=1}^k \{S_j^{(1)}>S_j^{(2)}\}) \\ & \hspace{2cm}\cdot \Pr\bigg( \bigcap_{\substack{j=1\\i = 1, 2}}^{n-k}\{S_j^{(i)}\in A_{j}^{(i)}\}|\bigcap_{j=1}^{n-k} \{S_j^{(1)}>S_j^{(2)}\}\bigg)
			\end{aligned}
		\end{equation}
		where we utilize the fact $\Pr(M_n =k) = \Pr(\cap_{j=1}^k S_j^{(1)}>S_j^{(2)})\Pr(\cap_{j=1}^{n-k} S_j^{(1)}>S_j^{(2)}).$ The above splitting essentially shows that conditioned on the maximizer, the left and right portion of the maximizer are independent non-intersecting random walks.\\
		
		We now consider $Z_n(t)= (\frac{S_{nt}^{(1)}}{\sqrt{n}}, \frac{S_{nt}^{(2)}}{\sqrt{n}})$ on $[0,1]$ where it is linearly interpolated in between. By Donsker's invariance principle, $Z_n \Rightarrow \B = (\B_1, \B_2)$ independent Brownian motions on $[0,1].$ Recall $\Omega$ from Definition \ref{def:omega}. Clearly $\Pr(\B \in \text{Discontinuity of }\Omega ) = 0,$ so 
		\begin{align*}
			(\Omega Z_n)_{+}\Rightarrow (\Omega \B)_{+} \text{ and }(\Omega Z_n)_{-}\Rightarrow (\Omega \B)_{-}.
		\end{align*} On the other hand, following \eqref{LA5} we see that conditioned on $M_n = k,$ $(\Omega Z_n)_{+} \stackrel{(d)}{=} Y_{n-k}$ and $(\Omega Z_n)_{-}\stackrel{(d)}{=}Y_k$ are independent where $Y_n(\cdot)$ is the linearly interpolated non-intersecting random walk defined in Proposition \ref{propF}.  As $k, n \rightarrow \infty,$ $Y_k(\cdot), Y_{n-k}(\cdot) \stackrel{d}{\to} \W$ where $W$ is the non-intersecting Brownian motion on $[0,1]$ defined in Definition \ref{def:nibm}. At the same time, $\frac{M_n}{n}\Rightarrow M$, which has density $\propto \frac{1}{\sqrt{t(1-t)}}$ on $[0,1].$ Thus, $(\Omega \B)_{+}, (\Omega \B)_{-}, M$ are independent and $(\Omega \B)_{+}\stackrel{(d)}{=}(\Omega \B)_{-}\stackrel{(d)}{=}\W.$
	\end{proof}
	
	\begin{remark}\label{r:gen_k}
		We expect similar decomposition results to hold for $3$ or more Brownian motions or bridges around the maximizer of their sums. More precisely, if $M$ is the maximizer of $B_1(x)+B_2(x)+B_3(x)$, where $B_i$ are independent Brownian motion on $[0,1]$, we expect the law of 
		$$(B_1(M)-B_1(M+x),B_2(M)-B_2(M+x),B_3(M)-B_3(M+x))$$
		to be again Brownian motions but their sum conditioned to be positive (its singular conditioning; so requires some care to define properly). Indeed, such a statement can be proven rigorously at the level of random walks. Then a possible approach is to take diffusive limit of random walks under conditioning and prove existence of weak limits. Due to lack of results for such conditioning event, proving such a statement require quite some technical work. Since it is extraneous for our purpose, we do not pursue this direction here.
	\end{remark}

	\section{Bessel bridges and non-intersecting Brownian bridges} \label{sec:bbnibb}
	
	In this section, we study diffusive limits and separation properties of Bessel bridges and non-intersecting Brownian bridges. The central object that appears in this section is the Dyson Brownian motion \cite{dyson1962brownian}  which are intuitively several Brownian bridges conditioned on non-intersection. In Section \ref{sec:dyson}, we recall Dyson Brownian motion and study different properties of it. In Section \ref{sec:sep} we prove a technical estimate that indicates the two parts of non-intersecting Brownian bridges have uniform separation and derive the diffusive limits of non-intersecting Brownian bridges. The precise renderings of these results are given in Proposition \ref{pgamma} and Proposition \ref{lemmaC}.
	
	\subsection{Diffusive limits of Bessel bridges and $\nonintbb$} \label{sec:dyson} We first recall the definition of Dyson Brownian motion. Although they are Brownian motions conditioned on non-intersection, since the conditioning event is singular, such an interpretation needs to be justified properly. There are several ways to rigorously define the Dyson Brownian motion, either through the eigenvalues of Hermitian matrices with Brownian motions as entries or as a solution of system of stochastic PDEs. In this paper, we recall the definition via specifying the entrance law and transition densities (see \cite{o2002representation} and \cite[Section 3]{warren2007dyson} for example). 
	
	\begin{definition}[Dyson Brownian motion] \label{def:dbm} A $2$-level Dyson Brownian motion $\calD(\cdot)=(\calD_1(\cdot),\calD_2(\cdot))$ is an $\R^2$ valued process on $[0,\infty)$ with $\calD_1(0)=\calD_2(0)=0$ and with the entrance law
		\begin{align}\label{eq:dysonent}
			\Pr\left(\calD_1(t)\in \d y_1,\calD_2(t) \in \d y_2\right)=\ind\{y_1>y_2\}\frac{(y_1-y_2)^2}{t}p_t(y_1)p_t(y_2)\d y_1 \d y_2, \quad t>0.
		\end{align}
		For $0<s<t<\infty$ and $x_1>x_2$, its transition densities are given by
		\begin{equation}
			\label{eq:dysontd}
			\begin{aligned}
				& \Pr\left(\calD_1(t)\in \d y_1,\calD_2(t) \in \d y_2\mid \calD_1(s)=x_1,\calD_2(s)=x_2\right)\\ & \hspace{2cm}=\ind\{y_1>y_2\}\frac{y_1-y_2}{x_1-x_2}\det(p_{t-s}(x_i-y_j))_{i,j=1}^2\d y_1 \d y_2.
			\end{aligned}
		\end{equation}
		The above formulas can be extended to $n$-level Dyson Brownian motions with (see \cite[Section 3]{warren2007dyson}) but for the rest of the paper we only require the $n=2$ case. So, we will refer to the $2$-level object defined above loosely as Dyson Brownian motion or $\dbm$ in short.
	\end{definition}
	We next define the Bessel processes via specifying the entrance law and transition densities which are also well known in literature (see \cite[Chapter VI.3]{revuz}). 
	\begin{definition}[Bessel Process]\label{def:bessel}  A 3D Bessel process $\mathcal{R}_1$ with diffusion coefficient $1$ is an $\R$-valued process on $[0,\infty)$ with $\mathcal{R}_1(0)=0$ and with the entrance law 
		\begin{align*}
			\Pr(\mathcal{R}_1(t)\in d y)=\tfrac{2y^2}{t}p_t(y)\d y, \quad x\in [0,\infty), \quad t>0.
		\end{align*}
		For $0<s<t<\infty$ and $x>0$, its transition densities are given by
		\begin{align*}
			\Pr(\mathcal{R}_1(t)\in d y\mid \mathcal{R}_1(s)=x)=\frac{y}{x}[p_{t-s}(x-y)-p_{t-s}(x+y)]dy,\quad y\in [0,\infty).
		\end{align*}
		More generally, $\mathcal{R}_{\sigma}(\cdot)$ is a 3D Bessel process with diffusion coefficient $\sigma>0$ if $\sigma^{-1/2}\mathcal{R}_{\sigma}(\cdot)$ is a 3D Bessel process with diffusion coefficient $1$. 
	\end{definition}
	
	In this paper we will only deal with $3$-dimensional Bessel processes. Thus we will just loosely refer to the above processes as Bessel processes. 
	
	$\dbm$ is directly linked with Bessel processes. Indeed the difference of the two paths of $\dbm$ is known (see \cite{dube} for example) to be a 3D Bessel process with diffusion coefficient $2$. This fact can be proven easily via SPDE or the Hermitian matrices interpretation of $\dbm$. Since we use this result repeatedly in later sections we record it as a lemma below. 
	
	\begin{lemma}[Dyson to Bessel]\label{l:dtob} Let $\calD=(\calD_1,\calD_2)$ be a $\dbm$. Then, as a function in $x$, we have $\calD_1(x)+\calD_2(x)\stackrel{d}{=}\sqrt2 B(x)$ and $\calD_1(x)-\calD_2(x)\stackrel{d}{=}\mathcal{R}_2(x)$ where $B(x)$ is a Brownian motion and $\mathcal{R}_2:[0,\infty)\to \R$ is  a Bessel process (see Definition \ref{def:bessel}) with diffusion coefficient $2$.
	\end{lemma}
	
	We end this subsection by providing two lemmas that compare the densities of $\nonintbb$ and $\dbm$.
	
	\begin{lemma}\label{lalpha}
		Suppose the pair of random variables $(\U_1, \U_2)$ has joint probability density function: 
		\begin{align}\label{eq:umar}
			\Pr(\U_1 \in \d y_1, \U_2 \in \d y_2) = \frac{(y_1-y_2)^2}{t}p_t(y_1)p_t(y_2), \ y_1 > y_2.
		\end{align}
		Conditioning on $(\U_1, \U_2), $ we consider a $\nonintbb$  $(\V_1, \V_2)$ on $[0,t]$ ending at $(\U_1, \U_2)$, see Definition \ref{def:nibb}. Then unconditionally, $(\V_1, \V_2)$ is equal in distribution as $\dbm$ $(\calD_1, \calD_2)$ on $[0,t].$ (see Definition \ref{def:dbm}).
	\end{lemma}
	
	\begin{lemma}\label{lbeta}
		Fix $\delta, M > 0.$ Consider a $\nonintbb$ $(\V_1, \V_2)$ on $[0,1]$ ending at $(a_1, a_2)$ (see Definition \ref{def:nibb}), where $a_1 > a_2$. Then, there exists a constant $\Con_{M, \delta} > 0 $ such that for all $t \in (0, \delta)$, $y_1>y_2$ and $-M \le a_2 < a_1 \le M$,
		\begin{align}\label{la1}
			\frac{\Pr(\V_1(t)\in \d y_1, \V_2(t)\in \d y_2)}{\Pr(\calD_1(t)\in \d y_1, \calD_2(t)\in \d y_2)} \le \Con_{M, \delta},
		\end{align}
		where $(\calD_1,\calD_2)$ is a $\dbm$ defined in Definition \ref{def:dbm}.
	\end{lemma}
	
	\begin{proof}[Proof of Lemma \ref{lalpha}] To show that  $(\V_1, \V_2)$ is equal in distribution to $(\calD_1, \calD_2)$ on $[0,t],$ it suffices to show that $(\V_1, \V_2)$ has the same finite dimensional distribution as $(\calD_1,\calD_2)$ on $[0,t]$. Fix any $k \in \N$, and $0< s_1 < \ldots < s_k < t$ and $y_1>y_2$. Using Brownian scaling and the formulas from Definition \ref{def:nibb} we have
		\begin{align*}
			& \Pr\bigg(\bigcap_{i =1}^k \{\V_1(s_i) \in \d x_{i,1}, \V_2(s_i) \in \d x_{i,2}\}|U_1=y_1,U_2=y_2\bigg) \\ & \hspace{2cm} =\frac{(x_{1,1}-x_{1,2})}{s_1}p_{s_1}(x_{1,1})p_{s_1}(x_{1,2})\prod_{m=1}^{k-1}\det(p_{s_{m+1}-s_m}(x_{m+1,i}-x_{m,j}))_{i, j=1}^2\\ & \hspace{5cm}\cdot \frac{\det(p_{t-{s_k}}(x_{k,i}-y_{j}))_{i, j=1}^2}{{\frac1t(y_1-y_2)p_t(y_1)p_t(y_2)}} \prod_{i=1}^k \d x_{i,1}\d x_{i,2},
		\end{align*}
		where the above density is supported on $\{x_{i,1}>x_{i,2} \mid i=1,2,\ldots,k\}$. For convenience, in the rest of the calculations, we drop $\prod_{i=1}^k \d x_{i,1}\d x_{i,2}$ from the above formula. In view of the marginal density of $(U_1,U_2)$ given by \eqref{eq:umar}, we thus have that
		\begin{align*}
			&\Pr\bigg(\bigcap_{i =1}^k \{\V_1(s_i) \in \d x_{i,1}, \V_2(s_i) \in \d x_{i,2}\}\bigg)\\ & = \int_{y_1>y_2}\Pr\bigg(\bigcap_{i =1}^k \{\V_1(s_i) \in \d x_{i,1}, \V_2(s_i) \in \d x_{i,2}\}|U_1=y_1,U_2=y_2\bigg)\frac{(y_1-y_2)^2}{t}p_t(y_1)p_t(y_2)\d y_1 \d y_2 \\ & = \frac{(x_{1,1}-x_{1,2})}{s_1}p_{s_1}(x_{1,1})p_{s_1}(x_{1,2})\prod_{m=1}^{k-1}\det(p_{s_{m+1}-s_m}(x_{m+1,i}-x_{m,j}))_{i, j=1}^2\\ & \hspace{5cm}\cdot \int_{y_1>y_2}(y_1-y_2){\det(p_{t-{s_k}}(x_{k,i}-y_{j}))_{i, j=1}^2}\d y_1 \d y_2.
		\end{align*}
		But given the transition densities for $\dbm$ from \eqref{eq:dysontd}. we know that $$ \int_{y_1>y_2}(y_1-y_2){\det(p_{t-{s_k}}(x_{k,i}-y_{j}))_{i, j=1}^2}\d y_1 \d y_2 = x_{k,1}-x_{k,2}.$$ 
		Plugging this back we get 
		\begin{align*}
			&\Pr\bigg(\bigcap_{i =1}^k \{\V_1(s_i) \in \d x_{i,1}, \V_2(s_i) \in \d x_{i,2}\}\bigg)\\ & = \frac{(x_{1,1}-x_{1,2})^2}{s_1}p_{s_1}(x_{1,1})p_{s_1}(x_{1,2})\prod_{m=1}^{k-1}\frac{x_{m+1,1}-x_{m+1,2}}{x_{m,1}-x_{m,2}}\det(p_{s_{m+1}-s_m}(x_{m+1,i}-x_{m,j}))_{i, j=1}^2.
		\end{align*}
		Using the entrance law and transition densities formulas for $\dbm$ from Definition \ref{def:dbm}, we see that the above formula matches with the finite dimensional density formulas for $\dbm$. This completes the proof.
	\end{proof}

	\begin{proof}[Proof of Lemma \ref{lbeta}] Fix any arbitrary $y_1 > y_2$ and $t \in (0, \delta)$ Recall the density formulas for $\nonintbb$ and $\dbm$ from Definitions \ref{def:nibb} and \ref{def:dbm}. We have
		\begin{align}\label{la2}
			\text{ l.h.s of }\eqref{la1}&= \frac{\det(p_{1-t}(y_i-a_j))_{i,j=1}^2}{(y_1 - y_2)(a_1-a_2)p_1(a_1)p_1(a_2)}\\ & \label{la10} =\frac{p_{1-t}(y_1-a_2)p_{1-t}(y_2-a_1)}{(y_1 - y_2)(a_1-a_2)p_1(a_1)p_1(a_2)}\left[e^{\frac{(y_1- y_2)(a_1- a_2)}{1-t}}-1\right].
		\end{align}
		If $(y_1-y_2)(a_1-a_2) \ge 1-t$, then 
		$$\mbox{r.h.s.~of \eqref{la2}} \le \frac{\det(p_{1-t}(y_i-a_j))_{i,j=1}^2}{(1-t)p_1(a_1)p_1(a_2)} \le \tfrac1{(1-t)^2}e^{\frac{a_1^2+a_2^2}{2}} \le \tfrac{1}{(1-\delta)^2}e^{M^2}.$$
		If $(y_1-y_2)(a_1-a_2) \le 1-t$, we utilize the elementary inequality that $\gamma(e^{\frac1\gamma}-1)\le e-1$, for all $\gamma\ge 1$. Indeed, taking $\gamma=\frac{1-t}{(y_1-y_2)(a_1-a_2)}\ge 1$ in this case we have
		$$\mbox{r.h.s.~of \eqref{la10}} \le \frac{p_{1-t}(y_1-a_2)p_{1-t}(y_2-a_1)}{(1-t)p_1(a_1)p_1(a_2)}(e-1) \le \tfrac2{(1-t)^2}e^{\frac{a_1^2+a_2^2}{2}} \le \tfrac{2}{(1-\delta)^2}e^{M^2}.$$
		Combining both cases yields the desired result.
	\end{proof}

	\subsection{Uniform separation and diffusive limits} \label{sec:sep} The main goal of this subsection is to prove Proposition \ref{pgamma} and Proposition \ref{lemmaC}. Proposition \ref{pgamma} highlights a uniform separation between the two parts of the $\nonintbb$ defined in Definition \ref{def:nibb}, while Proposition \ref{lemmaC} shows $\dbm$s are obtained as diffusive limits of $\nonintbb$s.
	
	\begin{proposition}\label{pgamma} Fix $M>0$.
		Let $(\V_1^{(n)}, \V_2^{(n)})$ be a sequence of $\nonintbb$s (see Definition \ref{def:nibb}) on $[0,1]$ beginning at 0 and ending at $(a_1^{(n)},a_2^{(n)}).$ Suppose that $a_1^{(n)} - a_2^{(n)} > \frac1M$ and $|a_i^{(n)}|\le M$ for all $n$ and $i=1,2$.  Then for all $\rho > 0$, we have 
		\begin{align}\label{pgammaineq}
			\limsup_{\theta \rightarrow \infty}\limsup_{n \rightarrow \infty}\Pr\left(\int_{\theta}^{n}\exp\left(-\sqrt{n}\big[\V_1^{(n)}(\tfrac{y}{n}) - \V_2^{(n)}(\tfrac{y}{n})\big]\right)\d y \ge \rho\right) = 0.
		\end{align}	
	\end{proposition}
	
	Recall that by Lemma \ref{l:nibtobb}, the difference of the two parts of $\nonintbb$ is given by a Bessel bridge (upto a constant). Hence we can recast the above result in terms of separations between Bessel bridges from the $x$-axis as well. 
	
	\begin{corollary}\label{pgamma1} Fix $M>0$.
		Let $\bb^{(n)}$ be a sequence of Bessel bridges (see Definition \ref{def:bbridge}) on $[0,1]$ beginning at 0 and ending at $a^{(n)}.$ Suppose that $M>a_1^{(n)} > \frac1M$ for all $n$.  Then for all $\rho > 0$, we have 
		\begin{align*}
			\limsup_{\theta \rightarrow \infty}\limsup_{n \rightarrow \infty}\Pr\left(\int_{\theta}^{n}\exp\left(-\sqrt{n}\bb^{(n)}(\tfrac{y}{n})\right)\d y \ge \rho\right) = 0.
		\end{align*}	
	\end{corollary}

	\begin{proof}[Proof of Proposition \ref{pgamma}]
		
		We fix $\delta \in (0,\frac14)$. To prove the inequality in \eqref{pgammaineq}, we divide the integral from $\theta $ to $n$ into two parts: $(\theta, n \delta)$ and $[n\delta, n)$ for some $\delta\in (0,1)$ and $n$ large and prove each one separately. For the interval $(n\delta, n)$ interval, we use the fact that the non-intersecting Brownian bridges $\V_1^{(n)}(y), \V_2^{(n)}(y)$ are separated by a uniform distance when away from 0. For the smaller interval $(\theta, n\delta)$ close to 0, we define a $\m{Gap}_{n,\theta, \delta}$ event that occurs with high probability and utilize Lemmas \ref{lalpha} and \ref{lbeta} to transform the computations of $\nonintbb$ into those of the $\dbm$ to simplify the proof. \\
		
		We now fill out the details of the above road-map. First, as $(\V_1^{(n)}, \V_2^{(n)})$ are non-intersecting Brownian bridges on $[0,1]$ starting from 0 and ending at two points which are within $[-M,M]$ and are separated by at least $\frac1M$, for every $\lambda, \delta > 0$, there exists $\alpha(M,\delta,\lambda)>0$ small enough such that 
		\begin{align}\label{pg1}
			\Pr\left(\V_1^{(n)}(y) - \V_2^{(n)}(y) \ge \alpha, \forall y \in [\delta, 1]\right) \ge 1 - \lambda.
		\end{align}
		\eqref{pg1} implies that with probability at least $1-\lambda$,
		\begin{align}\label{pg2}
			\int_{n\delta}^n \exp\left(-\sqrt{n}\big[\V_1^{(n)}(\tfrac{y}{n}) - \V_2^{(n)}(\tfrac{y}{n})\big]\right)\d y \le (n-n\delta)e^{-\sqrt{n}\alpha} 
		\end{align}
		which converges to 0 as $n \rightarrow \infty.$
		Next we define the event $$\m{Gap}_{n,\theta, \delta} := \left\{\sqrt{n}\big[\V_1^{(n)}(\tfrac{y}{n}) - \V_2^{(n)}(\tfrac{y}{n})\big]\ge y^{\frac{1}{4}}, \forall y \in [\theta, n\delta]\right\}.$$ 
		We claim that $\neg \m{Gap}_{n,\theta, \delta}$ event is negligible in the sense that
		\begin{align}\label{pg5}
			\limsup_{\theta \rightarrow \infty}\limsup_{n \rightarrow \infty}\Pr(\neg\m{Gap}_{n,\theta, \delta})=0.
		\end{align}
		Note that on $\m{Gap}_{n,\theta, \delta}$ event, we have
		\begin{align}\label{pg8}
			\int_{\theta}^{n\delta}\exp\left(-\sqrt{n}\big[\V_1^{(n)}(\tfrac{y}{n}) - \V_2^{(n)}(\tfrac{y}{n})\big]\right)\d y \le \int_{\theta}^{n\delta}\exp(-y^{1/4})\d y
		\end{align}
		which goes to zero as $n\to\infty$, followed by $\theta\to \infty$. In view of the probability estimates from \eqref{pg1} and \eqref{pg2}, combining \eqref{pg5} and \eqref{pg8} yields
		\begin{align}\label{pg9}
			\limsup_{\theta \rightarrow \infty}\limsup_{n \rightarrow \infty}\Pr\left(\int_{\theta}^{n}\exp\left(-\sqrt{n}\big[\V_1^{(n)}(\tfrac{y}{n}) - \V_2^{(n)}(\tfrac{y}{n})\big]\right)\d y \ge \rho\right)\le \lambda.
		\end{align}
		Since $\lambda$ is arbitrary, \eqref{pg9} completes the proof. Hence it suffices to show \eqref{pg5}. Towards this end, by the properties of the conditional expectation, if we condition on the values of  $\V_1^{(n)}(2\delta), \V_2^{(n)}(2\delta),$ we have that 
		\begin{align}\label{pg3}
			\Pr(\neg\m{Gap}_{n,\theta, \delta}) &= \Ex\left[\Pr\left(\neg\m{Gap}_{n,\theta, \delta}|\V_1^{(n)}(2\delta), \V_2^{(n)}(2\delta)\right)\right]\notag \\&= \int_{y_1>y_2} \Pr_{y_1, y_2}(\neg\m{Gap}_{n,\theta, \delta})\Pr(\V_1^{(n)}(2\delta)\in \d y_1,\V_2^{(n)}(2\delta)\in \d y_2 ) \end{align}
		where $\Pr_{y_1,y_2}$ is the conditional law of $\nonintbb$ conditioned on $(\V_1^{(n)}(2\delta)= y_1,\V_2^{(n)}(2\delta)=y_2)$. Note that $\m{Gap}_{n,\theta, \delta}$ event depends only on the $[0,\delta]$ path of the $\nonintbb$. Thus by Markovian property of the $\nonintbb$, $\Pr_{y_1, y_2}(\m{Gap}_{n,\theta, \delta})$ can be computed by assuming the $\nonintbb$ is on $[0,2\delta]$ and ends at $(y_1,y_2)$. 
		
		On the other hand, $\Pr(\V_1^n(2\delta)\in \d y_1,\V_2^n(2\delta)\in \d y_2 )$ is the probability density function of the marginal density of $\nonintbb$ on $[0,1]$. Via Lemma \ref{lbeta}, this is comparable to the density of $(\calD_1(2\delta),\calD_2(2\delta))$, where $\calD$ follows $\dbm$ law defined in Definition \ref{def:dbm}. Thus by \eqref{la1} the r.h.s of \eqref{pg3} is bounded from above by
		\begin{align}\label{pg4}
			\text{ r.h.s of  }\eqref{pg3} &\le \Con_{M, 2\delta}\int\Pr_{y_1, y_2}(\neg\m{Gap}_{n,\theta, \delta})\Pr(\calD_1(2\delta)\in \d y_1,\calD_2(2\delta)\in \d y_2 )\d y_1 \d y_2 \notag \\& = \Con_{M, 2\delta}\cdot \Pr_{\text{Dyson}}(\neg\m{Gap}_{n,\theta, \delta}).
		\end{align}
		Here the notation $\Pr_{\text{Dyson}}$ means the law of the paths $(V_1,V_2)$ is assumed to follow $\dbm$ law. With this notation, the last equality of \eqref{pg4} follows from Lemma \ref{lalpha}. From the density formulas of $\dbm$ from Definition \ref{def:dbm}, it is clear that $\dbm$ is invariant under diffusive scaling, i.e.
		\begin{align}
			\label{eq:invdbm}
			\sqrt{n}(\calD_1(\tfrac{\cdot}{n}), \calD_2(\tfrac{\cdot}{n}))\stackrel{d}{=}(\calD_1(\cdot), \calD_2(\cdot))
		\end{align} 
		and by Lemma \ref{l:dtob}, $\calD_1(\cdot)- \calD_2(\cdot)=\mathcal{R}_2(\cdot)$, a 3D Bessel process with diffusion coefficient 2. Thus, we obtain that for any $n \in \N,$
		\begin{align}\label{pg6}
			\Pr_{\text{Dyson}}(\neg\m{Gap}_{n,\theta, \delta}) \le \Pr(\mathcal{R}_2(y) \le y^{1/4}, \  \mbox{ for some } y \in [\theta,\infty)).
		\end{align}
		Meanwhile, Motoo's theorem \cite{motoo1959proof} tells us that 
		\begin{align}\label{pg7}
			\limsup_{\theta \rightarrow \infty}\Pr(\mathcal{R}_2(y) \le y^{1/4},  \  \mbox{ for some } y \in [\theta,\infty)) = 0.
		\end{align}
		Hence \eqref{pg4}, \eqref{pg6} and \eqref{pg7} imply \eqref{pg5}. This completes the proof.
	\end{proof}

	We now state our results related to the diffusive limits of $\nonintbb$ (defined in Definition \ref{def:nibb}) and Bessel bridges (defined in Definition \ref{def:bbridge}) with varying endpoints.

	\begin{proposition}
		\label{lemmaC}
		{Fix $M>0$. Let  $\V^{(n)} = (\V_1^{(n)}, \V_2^{(n)}):[0, a_n] \rightarrow \R$ be a sequence of $\nonintbb$s (defined in Definition \ref{def:nibb}) with  $\V_i^{(n)}(0) = 0$ and $\V_i^{(n)}(a_n) = z_i^{(n)}$. Suppose that for all $n\ge 1$ and $i=1,2$, $M>a_n > \frac1{M}$ and $|z_i^{(n)}| < \frac{1}{M}$. Then as $n\to \infty$ we have:
			\begin{align*}
				\sqrt{n}\big(\V_1^{(n)}(\tfrac{t}{n}), \V_2^{(n)}(\tfrac{t}{n})\big)\stackrel{d}{\to} (\calD_1(t), \calD_2(t))
			\end{align*}
			in the uniform-on-compact topology. Here $\calD$ is a $\dbm$ defined in Definition \ref{def:dbm}. }
	\end{proposition}	
	
	In view of Lemma \ref{l:nibtobb} and Lemma \ref{l:dtob}, Proposition \ref{lemmaC} also leads to the following corollary.

	\begin{corollary}\label{lemmaC1}
		Fix $M>0$. Let  $\bb^{(n)}:[0, a_n] \rightarrow \R$ be a sequence of Bessel bridges (defined in Definition \ref{def:bbridge}) with  $\bb^{(n)}(0) = 0$ and $\bb^{(n)}(a_n) = z^{(n)}$. Suppose for all $n\ge 1$, $M>a_n > \frac1{M}$ and $|z^{(n)}| < \frac{1}{M}$. Then as $n\to \infty$ we have:
		\begin{align*}
			\sqrt{n}\bb^{(n)}(\tfrac{t}{n})\stackrel{d}{\to} \mathcal{R}_1(t)
		\end{align*}
		in the uniform-on-compact topology. Here $\mathcal{R}_1$ is a Bessel process with diffusion coefficient $1$, defined in Definition \ref{def:bessel}. 
	\end{corollary}			
	
	\begin{proof}[Proof of Proposition \ref{lemmaC}] For convenience, we drop the superscript $(n)$ from $V_1,V_2$ and $z_i$'s. We proceed by showing convergence of one-point densities and transition densities of $\sqrt{n}(\V_1(\frac{t}{n}), \V_2(\frac{t}{n}))$ to that of $\dbm$ and then verifying the tightness of the sequence. Fix any $t>0$. For each fixed $y_1>y_2$, it is not hard to check that we have as $n\to \infty$
		\begin{align}\label{tech}
			\frac{a_n\sqrt{n}\det(p_{a_n-\frac{t}n}(\tfrac{y_i}{\sqrt{n}} - z_j))_{i,j =1}^2}{(z_1-z_2)p_{a_n}(z_1)p_{a_n}(z_2)} \to y_1-y_2.
		\end{align}
		uniformly over $a_n\in [\frac1M,M]$ and $z_1,z_2\in [-M,M]$. 
		
		Utilizing the one-point densities and transition densities formulas for $\nonintbb$ of length $1$ in Definition \ref{def:nibb}, we may perform a Brownian rescaling to get analogous formulas for $V_1,V_2$ which are $\nonintbb$ of length $a_n$. Then by a change of variable, the density of $(\sqrt{n}\V_1(\tfrac{t}{n}),\sqrt{n}\V_2(\tfrac{t}{n}))$ is given by
		\begin{align*}
			\frac{a_n(y_1 - y_2)p_t(y_1)p_t(y_2)}{t(z_1-z_2)p_{a_n}(z_1)p_{a_n}(z_2)}\sqrt{n}\det(p_{a_n-\frac{t}n}(\tfrac{y_i}{\sqrt{n}} - z_j))_{i,j =1}^2.
		\end{align*}
		Using \eqref{tech} we see that for each fixed $y_1>y_2$, the above expression goes to $\frac{(y_1-y_2)^2}{t}p_t(y_1)p_t(y_2)$ 
		which matches with \eqref{eq:dysonent}. 
		
		Similarly for the transition probability, letting $0< s< t< a_n$, $y_1 > y_2$ and $x_1 > x_2$, we have
		\begin{align}\label{db2}
			&\Pr\left(\sqrt{n}\V_1(\tfrac{t}{n})\in \d y_1, \sqrt{n}\V_2(\tfrac{t}{n})\in \d y_2\mid \sqrt{n}\V_1(\tfrac{s}{n})\in \d x_1, \sqrt{n}\V_2(\tfrac{s}{n})\in \d x_2\right) \notag \\=&\det(p_{t-s}(y_i-x_j))_{i,j =1}^2\frac{\det(p_{a_n-\frac{t}{n}}(\frac{y_i}{\sqrt{n}}-z_j))_{i,j =1}^2}{\det(p_{a_n-\frac{s}{n}}(\frac{x_i}{\sqrt{n}}-z_j))_{i,j =1}^2}\d y_1\d y_2.
		\end{align}
		Applying \eqref{tech} we see that as $n\to \infty$
		\begin{align*}
			\text{ r.h.s of  }\eqref{db2} \to \det(p_{t-s}(x_i-y_j))_{i,j=1}^2\cdot \frac{y_1-y_2}{x_1-x_2}.
		\end{align*}
		which matches with \eqref{eq:dysontd}. This verifies the finite dimensional convergence by Scheffe's theorem. For tightness we will show that there exists a constant $\Con_{K,M}>0$ such that for all $0< s< t< K$,
		\begin{align}\label{db5}
			\sum_{i = 1}^2\Ex\left[\left(\sqrt{n}\V_i(\tfrac{t}{n})-\sqrt{n}\V_i(\tfrac{s}{n})\right)^4\right] \le \Con_{K,M}(t-s)^2.
		\end{align}
		We compute the above expectation by comparing with $\dbm$ as was done in the proof of Proposition \ref{pgamma}. Using definition of the conditional expectation we have
		\begin{align*}
			& \Ex\left[\left(\sqrt{n}\V_i(\tfrac{t}{n})-\sqrt{n}\V_i(\tfrac{s}{n})\right)^4\right] \\ & =\int_{y_1>y_2}\Ex\left[\left(\sqrt{n}\V_i(\tfrac{t}{n})-\sqrt{n}\V_i(\tfrac{s}{n})\right)^4\mid \V_1(\tfrac{K}{n})=y_1,\V_1(\tfrac{K}{n})=y_2\right] \Pr(\V_1(\tfrac{K}{n})\in \d y_1,\V_2(\tfrac{K}{n})\in \d y_2 ) \\ & \le \Con_{K,M}\int_{y_1>y_2}\Ex\left[\left(\sqrt{n}\V_i(\tfrac{t}{n})-\sqrt{n}\V_i(\tfrac{s}{n})\right)^4\mid \V_1(\tfrac{K}{n})=y_1,\V_1(\tfrac{K}{n})=y_2\right] \Pr(\calD_1(\tfrac{K}{n})\in \d y_1,\calD_2(\tfrac{K}{n})\in \d y_2 )
		\end{align*}
		where the last inequality follows from {Lemma \ref{lbeta}} by taking $n$ large enough. Here $\calD=(\calD_1,\calD_2)$ follows $\dbm$ law. Due to {Lemma \ref{lalpha}} and \eqref{eq:invdbm} the last integral above is precisely $\Ex[(\calD_i(t)-\calD_i(s))^4]$. Hence it suffices to show
		\begin{align}
			\label{eq:dtight}
			\Ex[(\calD_i(t)-\calD_i(s))^4]\le \Con(t-s)^2.
		\end{align}
		By Lemma \ref{l:dtob}, we see $\sqrt{2}B(x):=\calD_1(x)+\calD_2(x)$ and $\sqrt{2}\mathcal{R}(x):=\calD_1(x)-\calD_2(x)$ are a standard Brownian motion and a 3D Bessel process with diffusion coefficient $1$ respectively. We have
		$$\Ex[(\calD_i(t)-\calD_i(s))^4] \le \Con \left[\Ex[(\mathcal{R}(t)-\mathcal{R}(s))^4]+\Ex[(B(t)-B(s))^4]\right].$$
		We have $\Ex[(B(t)-B(s))^4]=3(t-s)^2$, whereas for $\mathcal{R}(\cdot)$, we use Pitman's theorem \cite[Theorem VI.3.5]{revuz}, to get that $\mathcal{R}(t)\stackrel{d}{=}2M(t)-B(t)$, where $B$ is a Brownian motion and $M(t)=\sup_{u\le t} B(u)$.  Thus,
		\begin{align*}\Ex[(\mathcal{R}(t)-\mathcal{R}(s))^4] & \le \Con \left[\Ex[(M(t)-M(s))^4]+\Ex[(B(t)-B(s))^4]\right] \\ & \le \Con \left[\Ex\big[\big(\sup_{s\le u\le t} B(u)-B(s)\big)^4\big]+\Ex[(B(t)-B(s))^4]\right].
		\end{align*}
		Clearly both the expressions above are at most $\Con(t-s)^2$. This implies \eqref{eq:dtight} completing the proof.
	\end{proof}

	\section{Ergodicity and Bessel behavior of the KPZ equation} \label{sec:ergbb}
	
	The goal of this section is to prove Theorems \ref{t:bessel} and \ref{t:ergodic}. {As the proof of the latter is shorter and illustrates some of the ideas behind the proof of the former, we  first tackle Theorem \ref{t:ergodic} in Section \ref{sec:ergo}. After that in Section \ref{sec:dysonapp}, we state a general version of the $k=2$ case of Theorem \ref{t:bessel}, namely Proposition \ref{p:dyson}.} This proposition will then be utilized in the proof of Theorem \ref{t:main}. Finally in Section \ref{sec:last}, we show how to obtain Theorem \ref{t:bessel} from Proposition \ref{p:dyson}.
	
	\subsection{Proof of Theorem \ref{t:ergodic}} \label{sec:ergo} For clarity we divide the proof into several steps. 
	
	\medskip
	
	\noindent\textbf{Step 1.} In this step, we introduce necessary notations used in the proof and explain the heuristic idea behind the proof. 
	
	Fix any $a>0$. Consider any Borel set $A$ of $C([-a,a])$ which is also a continuity set of a two-sided Brownian motion $B(x)$ restricted to $[-a,a]$. By Portmanteau theorem, it suffices to show 
	\begin{align}\label{eq:ergtoshow}
		\Pr ((\calH(\cdot,t)-\calH(0,t)\in A) \to \Pr(B(\cdot)\in A).
	\end{align}
	For simplicity let us write $\Pr_t(A):=\Pr ((\calH(\cdot,t)-\calH(0,t)\in A)$. Using \eqref{eq:htx} we have $\calH(x,t)-\calH(0,t)=t^{1/3}(\h_t(t^{-2/3}x)-\h_t(0))$. Recall that $\h_t(\cdot)=\h_t^{(1)}(\cdot)$ can be viewed as the top curve of the KPZ line ensemble $\{\h_t^{(n)}(\cdot)\}_{n\in \N}$ which satisfies the $\mathbf{H}_t$-Brownian Gibbs property with $\mathbf{H}_t$ given by \eqref{eq:Ht}.

	Note that at the level of the scaled KPZ line ensembles we are interested in understanding the law of $\h_t^{(1)}(\cdot)$ restricted to a very small interval: $x\in [-t^{-2/3}a,t^{-2/3}a]$. At such a small scale, we expect the Radon-Nikodym derivative appearing in \eqref{eq:bgibbs} to be very close to $1$. Hence the law of top curve should be close to a Brownian bridge with appropriate end points. To get rid of the endpoints we employ the following strategy, {which is also illustrated in Figure \ref{fig:erg1} and its caption.}
	
	\begin{itemize}
		\item We start with a slightly larger but still vanishing interval $I_t:=(-t^{-\alpha},t^{-\alpha})$ with $\alpha=\frac16$ say. We show that conditioned on the end points $\h_t^{(1)}(-t^{-\alpha}), \h_t^{(1)}(t^{-\alpha})$ of the first curve and the second curve $\h_t^{(2)}$, the law of $\h_t^{(1)}$ is close to that of a Brownian bridge on $I_t$ starting and ending at $\h_t^{(1)}(-t^{-\alpha})$ and $\h_t^{(1)}(t^{-\alpha})$ respectively.
		\item  Once we probe further into an even narrower window of $[-t^{2/3}a,t^{2/3}a]$, the Brownian bridge no longer feels the effect of the endpoints and one gets a Brownian motion in the limit.  
	\end{itemize}
	\begin{figure}[h!]
		\centering
		\includegraphics[width=16cm]{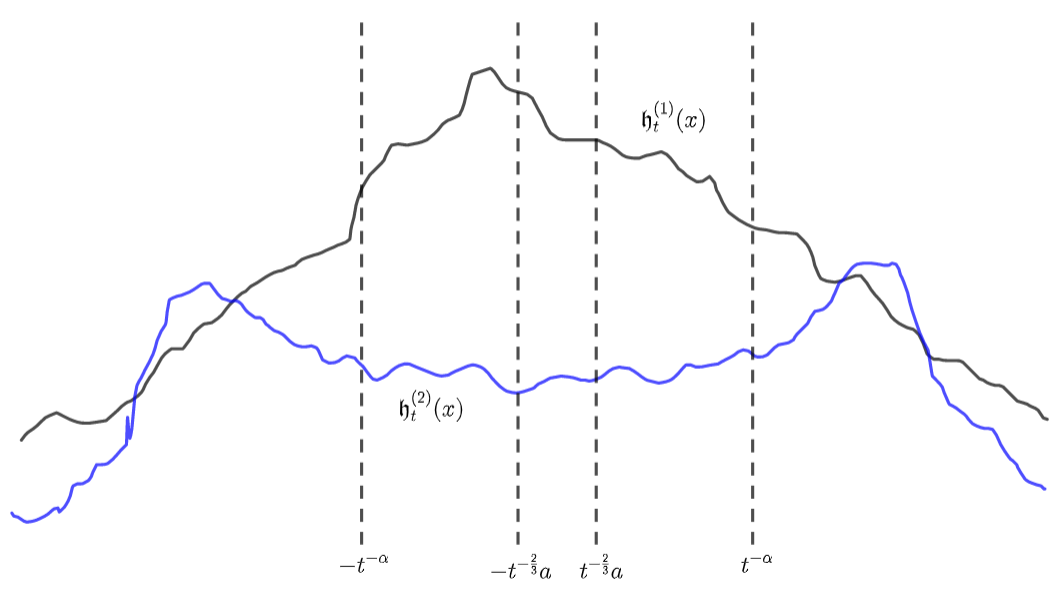}
		\caption{Illustration of the proof of Theorem \ref{t:ergodic}. In a window of $[t^{-\alpha},t^{\alpha}]$, the curves $\h_t^{(1)}(x), \h_t^{(2)}(x)$ attains an uniform gap with high probability. This allows us to show law of $\h_t^{(1)}$ on that small patch is close to a Brownian bridge. Upon zooming in a the tiny interval $[-t^{2/3}a,t^{2/3}a]$ we get a two-sided Brownian bridge as explained in \textbf{Step 1} of the proof.}
		\label{fig:erg1}
	\end{figure}

	\medskip
	
	\noindent\textbf{Step 2.} In this step and next step, we give a technical roadmap of the heuristics presented in \textbf{Step 1}. Set $\alpha=\frac16$ and consider the small interval $I_t=(t^{-\alpha},t^{-\alpha})$. Let $\mathcal{F}$ be the $\sigma$-field generated by \begin{align}
		\label{def:calf}
		\mathcal{F}:=\sigma\left(\{\h_t^{(1)}(x)\}_{x\in I_t^c}, \{\h_t^{(n)}(\cdot)\}_{n\ge 2}\right).
	\end{align} 
	Fix any arbitrary $\delta>0$ and consider the following three events:
	\begin{align}
		& \m{Gap}_t(\delta)  :=\left\{\h_t^{(2)}(-t^{-\alpha}) \le \min\{\h_t^{(1)}(t^{-\alpha}),\h_t^{(1)}(-t^{-\alpha})\}-\delta\right\}, \label{def:gap}\\
		& \m{Rise}_t(\delta)  :=\left\{\sup_{x\in I_t} \h_t^{(2)}(x)\le \tfrac14\delta+\h_t^{(2)}(-t^{-\alpha})\right\}, \label{def:rise} \\ & \m{Tight}_t({\delta})  :=\left\{-\delta^{-1} \le \h_t^{(1)}(t^{-\alpha}),\h_t^{(1)}(-t^{-\alpha})\le \delta^{-1}\right\}. \label{def:tight}
	\end{align}
	Note that all the above events are measurable w.r.t.~$\mathcal{F}$. A visual interpretation of the above events are given later in Figure \ref{fig:sink0}. Since the underlying curves are continuous almost surely, while specifying events over $I_t$, such as the $\m{Rise}_t(\delta)$ event defined in \eqref{def:rise}, one may replace $I_t$ by its closure $\bar{I}_t=[-t^{-\alpha},t^{-\alpha}]$ as well; the events will remain equal almost surely. We will often make use of this fact, and  will not make a clear distinction between $I_t$ and $\bar{I}_t$.  
	
	We set \begin{align}\label{def:fav}
		\m{Fav}_t({\delta}):=\m{Gap}_t(\delta)\cap \m{Rise}_t(\delta) \cap  \m{Tight}_t({\delta}).
	\end{align}
	The $\m{Fav}_t(\delta)$ event is a favorable event in the sense that given any $\e>0$, there exists $\delta_0\in (0,1)$ such that for all $\delta\in (0,\delta_0)$ 
	\begin{align}\label{eq:favbd}
		\liminf_{t\to\infty} \m{Fav}_{t}(\delta) \ge 1-\e. 
	\end{align}
	We will prove \eqref{eq:favbd} in \textbf{Step 4}. For the moment, we assume this and continue with our calculations. We now proceed to find tight upper and lower bounds for $\Pr_t(A) =\Pr ((\calH(\cdot,t)-\calH(0,t)\in A)$. Recall the $\sigma$-field $\calF$ from \eqref{def:calf}. Note that using the tower property of the conditional expectation we have 
	\begin{align}\label{eq:erglow}
		\Pr_t(A)=\Ex\left[\Pr_t(A\mid \mathcal{F})\right] \ge \Ex\left[\ind\{\m{Fav}_t(\delta)\}\Pr_t(A\mid \mathcal{F})\right].
	\end{align}
	\begin{align}\label{eq:ergup}
		\Pr_t(A)=\Ex\left[\Pr_t(A\mid \mathcal{F})\right] \le \Ex\left[\ind\{\m{Fav}_t(\delta)\}\Pr_t(A\mid \mathcal{F})\right]+\Pr\left(\neg\m{Fav}_t(\delta)\right).
	\end{align}
	Applying the $\mathbf{H}_t$-Brownian Gibbs property for the interval $I_t$ we have
	\begin{align}\label{eq:erg1}
		\Pr_t(A\mid \mathcal{F})   =\Pr_{\mathbf{H}_t}^{1,1,I_t,\h_t(-t^{-\alpha}),\h_t(t^{-\alpha}),\h_t^{(2)}}\left(A\right)=\frac{\Ex_{\operatorname{free},t}\left[W\ind_{A}\right]}{\Ex_{\operatorname{free},t}\left[W\right]},
	\end{align}
	where 
	\begin{align}\label{eq:ergW}
		W:=\exp\left(-t^{2/3}\int_{t^{-\alpha}}^{t^{-\alpha}} e^{t^{1/3}(\h_t^{(2)}(x)-\h_t^{(1)}(x))}\d x\right)
	\end{align}
	and $\Pr_{\operatorname{free},t}:=\Pr_{\operatorname{free}}^{1,1,I_t,\h_t(-t^{-\alpha}),\h_t(t^{-\alpha})}$ and  $\Ex_{\operatorname{free},t}:=\Ex_{\operatorname{free}}^{1,1,I_t,\h_t(-t^{-\alpha}),\h_t(t^{-\alpha})}$ are the probability and the expectation operator respectively for a Brownian bridge $B_1(\cdot)$ on $I_t$ starting at $\h_t(-t^{-\alpha})$ and ending at $\h_t(t^{-\alpha})$. Note that the second equality in \eqref{eq:erg1} follows from \eqref{eq:bgibbs}. We now seek to find upper and lower bounds for r.h.s.~of \eqref{eq:erg1}. For $W$, we have the trivial upper bound: $W\le 1$. For the lower bound, we claim that there exists $t_0(\delta)>0$, such that for all $t\ge t_0$, we have
	\begin{align}\label{eq:Wlbd}
		\ind\{\m{Fav}_t(\delta)\}\Pr_{\operatorname{free},t}(W\ge 1-\delta) \ge \ind\{\m{Fav}_t(\delta)\}(1-\delta).
	\end{align}
	Note that \eqref{eq:Wlbd} suggests that the $W$ is close to $1$ with high probability. This is the technical expression of the first conceptual step that we highlighted in \textbf{Step 1}. In the similar spirit for the second conceptual step, we claim that there exists $t_0(\delta)>0$, such that for all $t\ge t_0$, we have
	\begin{align}
		\ind\{\m{Fav}_t(\delta)\}\left|\Pr_{\operatorname{free},t}(A)-\gamma(A)\right| \le 	\ind\{\m{Fav}_t(\delta)\} \cdot \delta, \label{eq:bmconv}
	\end{align}
	where $\gamma(A):=\Pr(B(\cdot)\in A)\in [0,1]$. We remark that the l.h.s.~of \eqref{eq:Wlbd} and \eqref{eq:bmconv} are random variables measurable w.r.t.~$\mathcal{F}$. The inequalities above hold pointwise. We will prove \eqref{eq:Wlbd} and \eqref{eq:bmconv} in \textbf{Step 5} and \textbf{Step 6} respectively. We next complete the proof of the Theorem \ref{t:ergodic} assuming the above claims.
	
	\medskip
	
	\noindent\textbf{Step 3.} In this step we assume \eqref{eq:favbd}, \eqref{eq:Wlbd}, and \eqref{eq:bmconv} and complete the proof of \eqref{eq:ergtoshow}. Fix any $\e\in (0,1)$. Get a $\delta_0\in (0,1)$, so that \eqref{eq:favbd} is true for all $\delta\in (0,\delta_0)$. Fix any such $\delta\in (0,\delta_0)$. Get $t_0(\delta)$ large enough so that both \eqref{eq:Wlbd} and \eqref{eq:bmconv} hold for all $t\ge t_0$. Fix any such $t\ge t_0$. 
	
	\medskip
	
	As $W\le 1$, we note that on the event $\m{Fav}_t(\delta)$, 
	\begin{align*}
		\frac{\Ex_{\operatorname{free},t}\left[W\ind_{A}\right]}{\Ex_{\operatorname{free},t}\left[W\right]} & \ge \Ex_{\operatorname{free},t}\left[W\ind_{A}\right] \\ & \ge (1-\delta)\Pr_{\operatorname{free},t}\left({A}\cap\{W \ge 1-\delta\}\right) \\ & \ge  (1-\delta)\Pr_{\operatorname{free},t}({A})- (1-\delta)\Pr_{\operatorname{free},t}(W <1-\delta) \\ & \ge (1-\delta)\Pr_{\operatorname{free},t}({A})- (1-\delta)\delta,
	\end{align*}
	where in the last line we used \eqref{eq:Wlbd}. Plugging this bound back in \eqref{eq:erglow} we get
	\begin{align*}
		\Pr_t(A) & \ge (1-\delta)\Ex\left[\ind\{\m{Fav}_t(\delta)\}\Pr_{\operatorname{free},t}(A)\right]-(1-\delta)\delta \\ & \ge (1-\delta)\Ex\left[\ind\{\m{Fav}_t(\delta)\}\gamma(A)-\delta\right]-(1-\delta)\delta \\ & =\gamma(A)(1-\delta)\Pr(\m{Fav}_t(\delta))-2\delta(1-\delta).
	\end{align*}
	where the inequality in the penultimate line follows from \eqref{eq:bmconv}. Taking $\liminf$ both sides as $t\to \infty$, in view of \eqref{eq:favbd} we see that
	\begin{align*}
		\liminf_{t\to\infty}\Pr_t(A) \ge (1-\delta)(1-\e)\gamma(A)-2\delta(1-\delta).
	\end{align*}
	Taking $\liminf_{\delta\downarrow 0}$ and using the fact that $\e$ is arbitrary, we get that $	\liminf_{t\to\infty}\Pr_t(A)\ge \gamma(A)$. 
	Similarly for the {upper} bound, on the event $\m{Fav}_t(\delta)$ we have
	\begin{align*}
		\frac{\Ex_{\operatorname{free},t}\left[W\ind_{A}\right]}{\Ex_{\operatorname{free},t}\left[W\right]} & \le \frac{\Pr_{\operatorname{free},t}(A)}{(1-\delta)\Pr_{\operatorname{free},t}(W\ge 1-\delta)} \le \frac1{(1-\delta)^2}\Pr_{\operatorname{free},t}(A),
	\end{align*}
	where we again use \eqref{eq:Wlbd} for the last inequality. Inserting the above bound in \eqref{eq:ergup} we get
	\begin{align*}
		\Pr_t(A) & \le \frac1{(1-\delta)^2}\Ex\left[\ind\{\m{Fav}_t(\delta)\}\Pr_{\operatorname{free},t}(A)\right]+\Pr\left(\neg \m{Fav}_t(\delta)\right)  \\ & \le \frac{\delta}{(1-\delta)^2}+\frac{1}{(1-\delta)^2}\Ex\left[\ind\{\m{Fav}_t(\delta)\}\gamma(A)\right]+\Pr\left(\neg \m{Fav}_t(\delta)\right) \\ & \le \frac{\delta}{(1-\delta)^2}+\frac{1}{(1-\delta)^2}\gamma(A)+\Pr\left(\neg \m{Fav}_t(\delta)\right).
	\end{align*}
	The inequality in the penultimate line above follows from \eqref{eq:bmconv}. 
	Taking $\limsup$ both sides as $t\to \infty$, in view of \eqref{eq:favbd} we see that
	\begin{align*}
		\limsup_{t\to\infty}\Pr_t(A) \le \frac{\delta}{(1-\delta)^2}+\frac{1}{(1-\delta)^2}\gamma(A)+\e.
	\end{align*}
	As before taking $\limsup_{\delta\downarrow 0}$ and using the fact that $\e$ is arbitrary, we get that $	\limsup_{t\to\infty}\Pr_t(A)\le \gamma(A)$. With the matching upper bound for $\liminf$ derived above, we thus arrive at \eqref{eq:ergtoshow}, completing the proof of Theorem \ref{t:ergodic}. 
	
	\medskip
	
	\noindent\textbf{Step 4.}  In this step we prove \eqref{eq:favbd}. Fix any $\delta>0$. Recall the distributional convergence of KPZ line ensemble to Airy line ensemble from Proposition \ref{p:leconv}. By the Skorokhod representation theorem, we may assume that our probability space are equipped with $\mathcal{A}_1(x)$ $\mathcal{A}_2(x)$, such that as $t\to \infty$, almost surely we have
	\begin{align}\label{eq:ergsk}
		\max_{i=1,2} \sup_{x\in [-1,1]} |2^{1/3}\h_t^{(i)}(2^{1/3}x)-\mathcal{A}_i(x)| \to 0.
	\end{align}

	\begin{figure}[t]
		\centering
		\includegraphics[width=12cm]{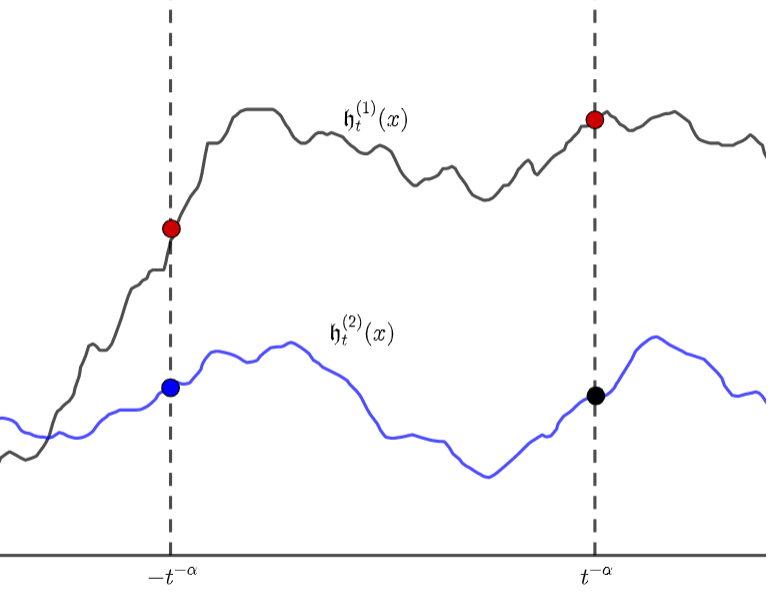}
		\caption{In the above figure $\m{Gap}_t(\delta)$ defined in \eqref{def:gap} denotes the event that the value of the blue point is smaller than the value of each of the red points at least by $\delta$, The $\m{Rise}_t(\delta)$ event defined in \eqref{def:rise} requires \textit{no} point on the whole blue curve (restricted to ${I}_t=(-t^{-\alpha},t^{-\alpha})$) exceed the value of the blue point by a factor $\frac14\delta$ (i.e., there is no significant rise). The $\m{Tight}_t(\delta)$ defined in \eqref{def:tight} event ensures the value of the red points are within $[-\delta^{-1},\delta^{-1}]$. The $\m{Fluc}_t^{(i)}(\delta)$ event defined in \eqref{eq:risei} signifies every value of every point on the $i$-th curve (restricted to ${I}_t$) is within $\frac14\delta$ distance away from its value on the left boundary: $\h_t^{(1)}(-t^{-\alpha})$. Finally, $\m{Sink}_t(\delta)$ event defined in \eqref{def:sinkerg} denotes the event that no point on the black curve (restricted to ${I}_t$) drops below the value of the red points by a factor larger than $\frac14\delta$, (i.e., there is no significant sink).} 
		\label{fig:sink0}
	\end{figure}

	For $i=1,2$, consider the event
	\begin{align}\label{eq:risei}
		\m{Fluc}_{t}^{(i)}(\delta):=\left\{\sup_{x\in I_t}|\h_t^{(i)}(x)-\h_t^{(i)}(-t^{-\alpha})|\le \tfrac14\delta\right\}.
	\end{align}
	See Figure \ref{fig:sink0} and its caption for an interpretation of this event. We claim that for every $\delta>0$,
	\begin{align}\label{eq:riselim}
		\liminf_{t\to \infty} \Pr\left(\m{Fluc}_t^{(i)}(\delta)\right)=1.
	\end{align}
	Let us complete the proof of \eqref{eq:favbd} assuming \eqref{eq:riselim}. Fix any $\e>0$.  Note that $\{|\h_t^{(1)}(-t^{-\alpha})-\h_t^{(1)}(t^{-\alpha})| \le \tfrac14\delta\} \supset \m{Fluc}_t^{(1)}(\delta)$. Recall $\m{Gap}_t(\delta)$ from \eqref{def:gap}. Observe that
	\begin{align*}
		\neg\m{Gap}_t(\delta) \cap \left\{|\h_t^{(1)}(-t^{-\alpha})-\h_t^{(1)}(t^{-\alpha})| \le \tfrac14\delta\right\} & \subset \left\{\h_t^{(2)}(-t^{\alpha})-\h_t^{(1)}(-t^{-\alpha}) \ge -\tfrac54\delta \right\} \\ & \subset \left\{  \inf_{x\in [-1,0]} [\h_t^{(2)}(x)-\h_t^{(1)}(x)] \ge -\tfrac54\delta \right\}.
	\end{align*}
	Using these two preceding set relations,  by union bound we have
	\begin{align*}
		\Pr\left(\neg\m{Gap}_t(\delta)\right) & \le \Pr\left(|\h_t^{(1)}(-t^{-\alpha})-\h_t^{(1)}(t^{-\alpha})| \ge \tfrac14\delta\right)+\Pr\left(\neg\m{Gap}_t(\delta) \cap |\h_t^{(1)}(-t^{-\alpha})-\h_t^{(1)}(t^{-\alpha})| \le \tfrac14\delta\right) \\ & \le \Pr\left(\neg \m{Rise}_t^{(1)}(\delta)\right)+\Pr\left(\inf_{x\in [-1,0]} [\h_t^{(2)}(x)-\h_t^{(1)}(x)] \ge -\tfrac54\delta\right).
	\end{align*}
	As $t\to \infty$, the first term goes to zero due \eqref{eq:riselim} and by Proposition \ref{p:leconv}, the second term goes to 
	$$\Pr\left(\inf_{x\in [-1,0]} [\mathcal{A}_2(2^{-1/3}x)-\mathcal{A}_1(2^{-1/3}x)] \ge -\tfrac5{4\cdot 2^{1/3}}\delta\right).$$
	But by \eqref{eq:order} we know Airy line ensembles are strictly ordered. Thus the above probability can be made arbitrarily small be choose $\delta$ small enough. In particular, there exists a $\delta_1\in (0,1)$ such that for all $\delta\in (0,\delta_1)$ the above probability is always less than $\frac\e2$. This forces \begin{align}\label{eq:gaplim}
		\liminf_{t\to \infty}\Pr\left(\m{Gap}_t(\delta)\right) \ge 1-\tfrac\e2.
	\end{align}
	Recall $\m{Rise}_t(\delta)$ from \eqref{def:rise}. Clearly $\m{Rise}_t(\delta)\subset \m{Fluc}_t^{(2)}(\delta)$. Thus for every $\delta>0$, \begin{align}\label{eq:rise0lim}
		\liminf_{t\to\infty} \Pr(\m{Rise}_t(\delta))=1.
	\end{align} Finally using Proposition \ref{p:kpzeq} \ref{p:stat} and \ref{p:tail} we see that $\h_t^{(1)}(t^{-\alpha}),\h_t^{(1)}(t^{-\alpha})$ are tight. Thus there exists $\delta_2\in (0,1)$ such that for all $\delta\in (0,\delta_2)$, we have 
	\begin{align}\label{eq:tightlim}
		\liminf_{t\to\infty} \Pr\left(\m{Tight}_t(\delta)\right) \ge 1-\tfrac{\e}{2}.
	\end{align}
	Combining \eqref{eq:gaplim}, \eqref{eq:rise0lim}, \eqref{eq:tightlim}, and recalling the definition of $\m{Fav}_t(\delta)$ from \eqref{def:fav}, by union bound we get \eqref{eq:favbd} for all $\delta\in (0,\min\{\delta_1,\delta_2\})$. \\
	
	Let us now prove \eqref{eq:riselim}. Recall $\m{Fluc}_t^{(i)}(\delta)$ from \eqref{eq:risei}. Define the event:
	\begin{align*}
		\m{Conv}_t(\delta):=\left\{\sup_{x\in [-1,1]}|\h_t^{(i)}(x)-2^{-1/3}\mathcal{A}_i(2^{-1/3}x)|\le \tfrac{1}{16}\delta\right\}.
	\end{align*}
	Observe that 
	\begin{align*}
		& \left\{\neg\m{Fluc}_t^{(i)}(\delta),\m{Conv}_t(\delta)\right\} \subset \left\{\sup_{|x|\le 2^{-1/3}t^{-\alpha}} \left[\mathcal{A}_i(x)-\mathcal{A}_i(-2^{-1/3}t^{-\alpha})\right] \ge \tfrac{2^{1/3}}{8}\delta\right\}.
	\end{align*}
	Thus by union bound
	\begin{align*}
		\Pr\left(\neg\m{Fluc}_t^{(i)}(\delta)\right) & \le \Pr\left(\neg \m{Conv}_t(\delta)\right)+\Pr\left(\neg\m{Fluc}_t^{(i)}(\delta),\m{Conv}_t(\delta)\right) \\ & \le \Pr\left(\neg \m{Conv}_t(\delta)\right) 
		+\Pr\left(\sup_{|x|\le 2^{-1/3}t^{-\alpha}} \left[\mathcal{A}_i(x)-\mathcal{A}_i(-2^{-1/3}t^{-\alpha})\right] \ge \tfrac{2^{1/3}}{8}\delta\right).
	\end{align*}
	By \eqref{eq:ergsk}, the first term above goes to zero as $t\to \infty$, whereas the second term goes to zero as $t\to \infty$, via modulus of continuity of Airy line ensembles from Proposition \ref{p:amodcon}.  Note that in Proposition \ref{p:amodcon} the modulus of continuity is stated for $\calA_i(x)+x^2$. However, in the above scenario since we deal with a vanishing interval $[-2^{-1/3}t^{-\alpha},2^{-1/3}t^{-\alpha}]$, the parabolic term does not play any role. This establishes \eqref{eq:riselim}.
	
	\medskip

	\noindent\textbf{Step 5.} In this step we prove \eqref{eq:Wlbd}. Let us consider the event
	\begin{align}
		\label{def:sinkerg}
		\m{Sink}_{t}(\delta):=\left\{\inf_{x\in I_t} \h_t^{(1)}(x) \ge -\tfrac14\delta+\min\{\h_t(-t^{-\alpha}),\h_t(t^{-\alpha})\}\right\}.
	\end{align}
	See Figure \ref{fig:sink0} and its caption for an interpretation of this event. Recall $\m{Gap}_t(\delta)$ and $\m{Rise}_t(\delta)$ from \eqref{def:gap} and \eqref{def:rise}. Observe that on the event $\m{Gap}_t(\delta)\cap \m{Rise}_t(\delta)$, we have $\sup_{x\in I_t} \h_t^{(2)}(x) \le \min\{\h_t(-t^{-\alpha}),\h_t(t^{-\alpha})\}-\frac34\delta$. Thus on $\m{Gap}_t(\delta)\cap \m{Rise}_t(\delta)\cap \m{Sink}_{t}(\delta)$, we have $$\inf_{x\in {I}_t} \left[\h_t^{(1)}(x)-\h_t^{(2)}(x)\right] \ge \tfrac12\delta.$$ Recall $W$ from \eqref{eq:ergW}. On the event $\{\inf_{x\in I_t} \left[\h_t^{(1)}(x)-\h_t^{(2)}(x)\right] \ge \tfrac12\delta\}$ we have the pointwise inequality
	$$W>\exp(-2t^{2/3-\alpha}e^{-\frac12t^{1/3}\delta})\ge 1-\delta,$$
	where we choose a $t_1(\delta)>0$ so that the last inequality is true for all $t\ge t_1$.  Thus for all $t\ge t_1$,
	\begin{align}\label{eq:sinkbd}
		\ind\{\m{Fav}_t(\delta)\}\Pr_{\operatorname{free},t}(W\ge 1-\delta) \ge \ind\{\m{Fav}_t(\delta)\}\Pr_{\operatorname{free},t}(\m{Sink}_t(\delta)).
	\end{align} 
	Recall that $\Pr_{\operatorname{free},t}$ denotes the law of a Brownian bridge $B_1(\cdot)$ on $I_t$ starting at $B_1(-t^{-\alpha})=\h_t(-t^{-\alpha})$ and ending at $B_1(t^{-\alpha})=\h_t(t^{-\alpha})$. Let us consider another Brownian bridge $\til{B}_1(\cdot)$ on $I_t$ starting and ending at $\min\{\h_t(-t^{-\alpha}),\h_t(t^{-\alpha})\}$. By standard estimates for Brownian bridge (see Lemma 2.11 in \cite{CH16} for example)
	\begin{align*}
		\Pr\left(\inf_{x\in I_t} \til{B}_1(x) \ge -\tfrac14\delta+\min\{\h_t(-t^{-\alpha}),\h_t(t^{-\alpha})\}\right)=1-\exp\left(-\tfrac{\delta^2}{8|I_t|}\right)=1-\exp\left(-\tfrac{\delta^2}{16}t^{\alpha}\right).
	\end{align*} 
	Note that $B_1(\cdot)$ is stochastically larger than $\til{B}_1(\cdot)$. Since the above event is increasing, we thus have $\Pr_{\operatorname{free},t}\left(\m{Sink}_t(\delta)\right)$ is at least $1-\exp\left(-\tfrac{\delta^2}{16}t^{\alpha}\right)$. We now choose $t_2(\delta)>0$, such that $1-\exp\left(-\tfrac{\delta^2}{16}t^{\alpha}\right)\ge 1-\delta$. Taking $t_0=\max\{t_1,t_2\}$, we thus get \eqref{eq:Wlbd} from \eqref{eq:sinkbd}. 
	\medskip

	\noindent\textbf{Step 6.} In this step we prove \eqref{eq:bmconv}. As before consider the Brownian bridge $B_1(\cdot)$ on $I_t$ starting at $B_1(-t^{-\alpha})=\h_t(-t^{-\alpha})$ and ending at $B_1(t^{-\alpha})=\h_t(t^{-\alpha})$. We may write $B_1$ as
	$$B_1(x)=\h_t^{(1)}(-t^{-\alpha})+\frac{x+t^{-\alpha}}{2t^{-\alpha}}(\h_t^{(1)}(t^{-\alpha})-\h_t^{(1)}(-t^{-\alpha}))+\bar{B}(x).$$
	where $\bar{B}$ is a Brownian bridge on $I_t$ starting and ending at zero. Thus,
	\begin{align}\label{eq:bmdecomp}
		t^{1/3}(B_1(t^{-2/3}x)-B_1(0))=t^{1/3}\left[\bar{B}(t^{-2/3}x)-\bar{B}(0)\right]+\tfrac12{t^{\alpha-1/3}x}(\h_t^{(1)}(t^{-\alpha})-\h_t^{(1)}(-t^{-\alpha})).
	\end{align}
	Recall that $\alpha=\frac16$. By Brownian scaling, $B_*(x):=t^{1/3}\bar{B}(t^{-2/3}x)$ is a Brownian bridge on the large interval $[-\sqrt{t},\sqrt{t}]$ starting and ending at zero. By computing the covariances, it is easy to check that as $t\to \infty$, $B_*(x)-B_*(0)$ converges weakly to a two-sided Brownian motion $B(\cdot)$ on $[-a,a]$. This gives us the weak limit for the first term on the r.h.s.~of \eqref{eq:bmdecomp}. For the second term, recall the event $\m{Tight}_t(\delta)$ from \eqref{def:tight}. As $|x|\le a$, on $\m{Tight}_t(\delta)$, we have
	$$ \tfrac12{t^{\alpha-1/3}x}(\h_t^{(1)}(t^{-\alpha})-\h_t^{(1)}(-t^{-\alpha})) \le {t^{-1/6}a}\delta^{-1}.$$
	This gives an uniform bound (uniform over the event $\m{Fav}_t(\delta))$ on the second term in \eqref{eq:bmdecomp}. Thus as long as the boundary data is in $\m{Tight}_t(\delta)$, $\Pr_{\operatorname{free},t}(A) \to \gamma(A)$ where $\gamma(A)=\Pr(B(\cdot)\in A)$. This proves \eqref{eq:bmconv}.

	\subsection{Dyson Behavior around joint maximum}\label{sec:dysonapp} In this subsection we state and prove Proposition \ref{p:dyson}.
	
	\begin{proposition}[Dyson behavior around joint maximum] \label{p:dyson}	  Fix $p\in (0,1)$. Set $q=1-p$. Consider $2$ independent copies of the KPZ equation $\calH_\uparrow(x,t)$, and $\calH_\downarrow(x,t)$, both started from the narrow wedge initial data. Let $\md_{p,t}$ be the almost sure unique maximizer of the process $x\mapsto (\calH_\uparrow(x,pt)+\calH_\downarrow(x,qt))$ which exists via Lemma \ref{lem1}. Set
		\begin{equation}
			\label{def:d1d2}
			\begin{aligned}
				D_1(x,t) & :=\calH_\uparrow(\md_{p,t},pt)-\calH_\uparrow(x+\md_{p,t},pt), \\
				D_2(x,t) & :=\calH_\downarrow(x+\md_{p,t},qt)-\calH_\downarrow(\md_{p,t},qt).
			\end{aligned}
		\end{equation}
		As $t\to\infty$, we have the following convergence in law
		\begin{align}\label{eq:dyson}
			(D_1(x,t),D_2(x,t)) \stackrel{d}{\to} (\calD_1(x),\calD_2(x))
		\end{align}
		in the uniform-on-compact topology. Here $\calD=(\calD_1,\calD_2): \R\to \R^2$ is a two-sided $\dbm$, that is, $\calD_{+}(\cdot):=\calD(\cdot)\mid_{[0,\infty)}$ and $\calD_{-}(\cdot):=\calD(-\cdot)\mid_{(-\infty,0]}$ are independent copies of $\dbm$ defined in Definition \ref{def:dbm}.
	\end{proposition}

	For clarity, the proof is completed over several subsections (Sections \ref{sb:frame}-\ref{sb:2lem}) below and we refer to Figure \ref{fig:struc} for the structure of the proof.
	
		\begin{center}
	\footnotesize
  \begin{figure}[h]
  \begin{tikzpicture}[auto,
  	block_main/.style = {rectangle, draw=black, thick, fill=white, text width=10em, text centered, minimum height=4em, font=\small},
	block_density/.style = {rectangle, draw=black, fill=white, thick, text width=11em, text centered, minimum height=4em},
	block_rewrite/.style = {rectangle, draw=black, fill=white, thick, text width=17em, text centered, minimum height=4em},
	block_kernels/.style = {rectangle, draw=black, fill=white, thick, text width=20em, text centered, minimum height=4em},
        line/.style ={draw, thick, -latex', shorten >=0pt}]
      		\node [block_main] (sb1) at (-5,4) {Recasting Proposition \ref{p:dyson} in the KPZ line ensemble framework (Section \ref{sb:frame})};
      		\node [block_main] (sb2) at (0,4) {Heuristics and outline of proof of Proposition \ref{p:dyson} (Section \ref{sb:ideas})};
      		\node [block_main] (sb3) at (5,4) {Reducing the global maximizer to the local maximizer (Section \ref{sb:gtol})};
      		\node [block_main] (sb4) at (-5,0) {Defining ``Nice" events that happen with high probability (Lemma \ref{l:nice}, Section \ref{sb:nice})};
      		\node [block_main] (sb5) at (0,0) {Conditioning w.r.t. large boundaries to obtain Brownian bridge law (Section \ref{sb:large})};
      		\node [block_main] (sb6) at (5,0) {Conditioning w.r.t. the max data and small boundaries to obtain $\nonintbb$ law (Section \ref{sb:small})};
      		\node [block_main] (sb7) at (-5,-4) {Obtaining matching upper and lower bounds for \eqref{wts1} and the desired convergence (Section \ref{sb:ublb})};
      		\node [block_main] (sb8) at (0,-4) {Proof of Lemma \ref{l:nice} (Section \ref{sb:pfnice})};
      		\node [block_main] (sb9) at (5,-4) {Proofs of Lemmas \ref{l:rn1} and \ref{l:closetod} (Section \ref{sb:2lem})};
      	
    \begin{scope}[every path/.style=line]
		\path (sb1) -- (sb2);
		\path (sb2) -- (sb3);
		\path (sb3) -- (sb4);
		\path (sb4) -- (sb5);
		\path (sb5) -- (sb6);
	    \path (sb6) -- (sb7);
		\path (sb7) -- (sb8);
		\path (sb8) -- (sb9);
    \end{scope}
  \end{tikzpicture}
  \caption{Structure of Section \ref{sec:dysonapp}.}
  \label{fig:struc}
  \end{figure}
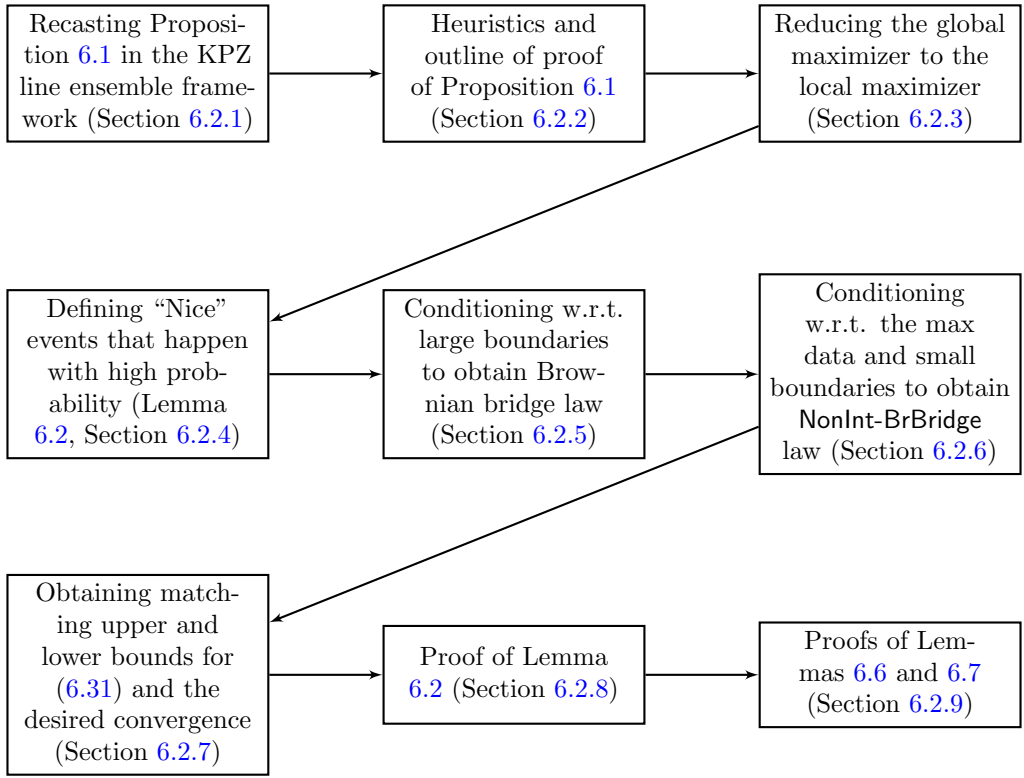
\end{center}
	
	\subsubsection{KPZ line ensemble framework} \label{sb:frame}

	In this subsection, we convert Proposition \ref{p:dyson} into the language of scaled KPZ line ensemble defined in Proposition \ref{line-ensemble}. We view $\calH_{\uparrow}(x,t)=\calH_{t,\uparrow}^{(1)}(x), \calH_{\downarrow}(x,t)=\calH_{t,\downarrow}^{(1)}(x)$ as the top curves of two (unscaled) KPZ line ensembles: $\{\calH_{t,\uparrow}^{(n)}(x),\calH_{t,\downarrow}^{(n)}(x)\}_{n\in \N,x\in\R}$. Following \eqref{eq:scaleKPZ} we define their scaled versions:
	\begin{align*}
		\h_{t,\uparrow}^{(n)}(x) & := t^{-1/3}\left(\calH_{t,\uparrow}^{(n)}(t^{2/3}x)+\tfrac{t}{24}\right), \qquad \h_{t,\downarrow}^{(n)}(x)  := t^{-1/3}\left(\calH_{t,\downarrow}^{(n)}(t^{2/3}x)+\tfrac{t}{24}\right).
	\end{align*}
	Along with the full maximizer $\md_{p,t}$, we will also consider local maximizer defined by 
	\begin{align}
		\label{eq:localmax}
		\md_{p,t}^M:=\argmax_{x\in [-Mt^{2/3},Mt^{2/3}]} (\calH_{pt,\uparrow}^{(1)}(x)+\calH_{qt,\downarrow}^{(1)}(x)), \qquad M\in[0,\infty].
	\end{align}
	For each $M>0$, $\md_{p,t}^M$ is unique almost surely by $\mathbf{H}_t$-Brownian Gibbs property. We now set
	\begin{equation}
		\label{eq:yupdown}
		\begin{aligned}
			Y_{M,t,\uparrow}^{(n)}(x) & := p^{1/3}\h_{pt,\uparrow}^{(n)}\big((pt)^{-2/3}\md_{p,t}^M\big)-p^{1/3}\h_{pt,\uparrow}^{(n)}\big(p^{-2/3}x\big), \\
			Y_{M,t,\downarrow}^{(n)}(x) & := q^{1/3}\h_{qt,\downarrow}^{(n)}\big(q^{-2/3}x\big)-q^{1/3}\h_{qt,\downarrow}^{(n)}\big((qt)^{-2/3}\md_{p,t}^M\big).
		\end{aligned}
	\end{equation}
	Taking into account of \eqref{def:d1d2} and all the above new notations, it can now be checked that for each $t>0$,
	\begin{align}\label{eq:dtoy}
		D_1(x,t)\stackrel{d}{=} t^{1/3}Y_{\infty,t,\uparrow}^{(1)}\big(t^{-2/3}(\md_{p,t}^{\infty}+x)\big), \quad D_2(x,t)\stackrel{d}{=} t^{1/3}Y_{\infty,t,\downarrow}^{(1)}\big(t^{-2/3}(\md_{p,t}^{\infty}+x)\big),
	\end{align}
	both as functions in $x$. Thus it is equivalent to verify Proposition \ref{p:dyson} for the above $Y_{\infty,t,\uparrow}^{(1)},Y_{\infty,t,\downarrow}^{(1)}$ expressions. In our proof we will mostly deal with local maximizer version, and so for convenience we define:
	\begin{align}\label{eq:dtoyM}
		D_{M,t,\uparrow}(x):{=} t^{1/3}Y_{M,t,\uparrow}^{(1)}\big(t^{-2/3}(\md_{p,t}^{M}+x)\big), \quad D_{M,t,\downarrow}(x)= t^{1/3}Y_{M,t,\downarrow}^{(1)}\big(t^{-2/3}(\md_{p,t}^{M}+x)\big).
	\end{align}
	where $Y_{M,t,\uparrow}^{(1)},Y_{M,t,\downarrow}^{(1)}$ are defined in \eqref{eq:yupdown}. We will introduce several other notations and parameters later in the proof. For the moment, the minimal set of notations introduced here facilitate our discussion of ideas and outline of the proof of Proposition \ref{p:dyson} in the next subsection. 
	
	\subsubsection{Ideas and Outline of Proof of Proposition \ref{p:dyson}} \label{sb:ideas}  Before embarking on a rather lengthy proof, in this subsection we explain the core ideas behind the proof and provide an outline for the remaining subsections.

	First we contrast the proof idea with that of Theorem \ref{t:ergodic}. Indeed, similar to the proof of Theorem \ref{t:ergodic}, from \eqref{eq:dtoy} we see that at the level $Y_{\infty,t,\uparrow}^{(1)},Y_{\infty,t,\downarrow}^{(1)}$ we are interested in understanding their laws restricted to a very small symmetric interval of order $O(t^{-2/3})$ around the point $t^{-2/3}\md_{p,t}^{\infty}$. However, the key difference from the conceptual argument presented at the beginning of the proof if Theorem \ref{t:ergodic} is that the centered point $t^{-2/3}\md_{p,t}^{\infty}$ is random and it does not go to zero. Rather by Theorem \ref{t:favpt} it converges in distribution to a nontrivial random quantity (namely $\Gamma(p\sqrt{2})$). Hence one must take additional care of this random point. This makes the argument significantly more challenging compared to that of Theorem \ref{t:ergodic}. \\

	\begin{figure}[h!]
		\centering
		\includegraphics[width=15cm]{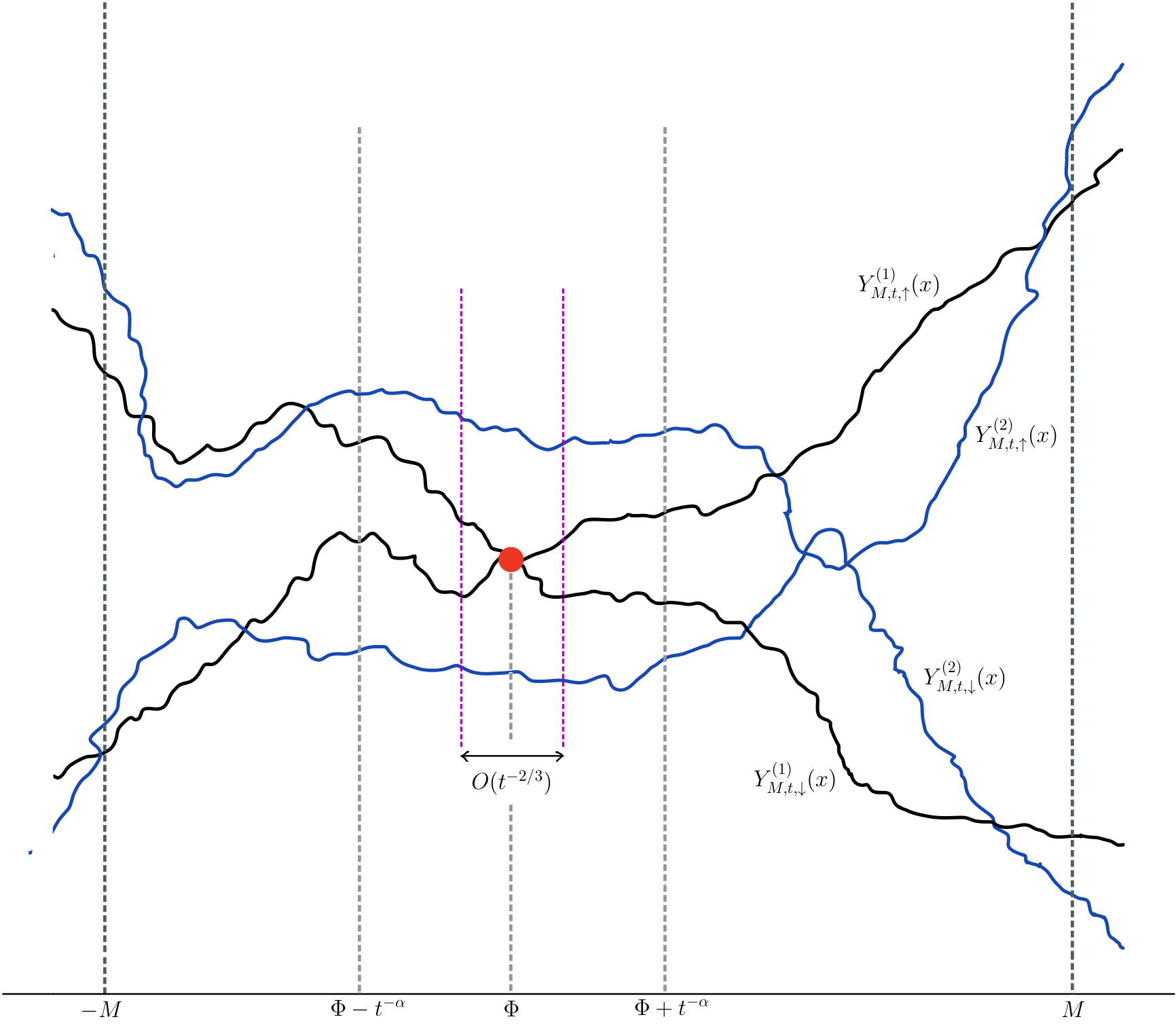}
		\caption{An overview of the proof for Proposition \ref{p:dyson}. The top and bottom black curves are $Y_{M,t,\uparrow}^{(1)}$ and $Y_{M,t,\downarrow}^{(1)}$ respectively. Note that the way they are defined in \eqref{eq:yupdown}, $Y_{M,t,\uparrow}^{(1)}(x)\ge Y_{M,t,\downarrow}^{(1)}(x)$ with equality at $x=\Phi=t^{-2/3}\md_{p,t}^{M}$ labelled as the red dot in the above figure. The blue curves are $Y_{M,t,\uparrow}^{(1)},Y_{M,t,\downarrow}^{(2)}$. There is no such ordering within blue curves. They may intersect among themselves as well as with the black curves. With $\alpha=\frac16$, we consider the interval $K_t=(\Phi-t^{-\alpha},\Phi+t^{-\alpha})$. In this vanishing interval around $\Phi$, the curves will be ordered with high probability. In fact, with high probability, there will be a uniform separation. For instance, for small enough $\delta$, we will have $Y_{M,t,\uparrow}^{(2)}(x)-Y_{M,t,\uparrow}^{(1)}(x) \ge \frac14\delta$, and $Y_{M,t,\downarrow}^{(1)}(x)-Y_{M,t,\downarrow}^{(2)}(x) \ge \frac14\delta$, for all $x\in K_t$ wth high probability. This will allow us to conclude black curves are behave approximately like two-sided $\nonintbb$s on that narrow window. Then upon going into a even smaller window of $O(t^{-2/3})$, the two-sided $\nonintbb$s turn into a two-sided $\dbm$.}
		\label{fig:dyson1}
	\end{figure}
	
	We now give a road-map of our proof. At this point, readers are also invited to look into Figure \ref{fig:dyson1} alongside the explanation offered in its caption.
	
	\begin{itemize}
		\item As noted in Lemma \ref{lem1}, the random centering $t^{-2/3}\md_{p,t}^{\infty}$ has decaying properties and can be approximated by $t^{-2/3}\md_{p,t}^{M}$ by taking large enough $M$. Hence on a heuristic level it suffices to work with the local maximizers instead. In Subsection \ref{sb:gtol}, this heuristics will be justified rigorously. We will show there how to pass from $Y_{\infty,t,\uparrow}^{(1)}, Y_{\infty,t,\downarrow}^{(1)}$ defined in \eqref{eq:dtoy} to their finite centering analogs: $Y_{M,t,\uparrow}^{(1)}, Y_{M,t,\downarrow}^{(1)}$. The rest of the proof then boils down to analyzing the laws of the latter.

		\item We now fix a $M>0$ for the rest of the proof. Our analysis will now operate with $\md_{p,t}^M$. For simplicity, let us also use the notation
		\begin{align}
			\label{def:phi}
			\Phi:=t^{-2/3}\md_{p,t}^M
		\end{align}
		for the rest of the proof. We now perform several conditioning on the laws of the curves.	Recall that by Proposition \ref{line-ensemble}, $\{\h_{pt,\uparrow}^{(n)}(\cdot)\}_{n\in \N}$ $\{\h_{qt,\downarrow}^{(n)}(\cdot)\}_{n\in \N}$ satisfy the $\mathbf{H}_{pt}$-Brownian Gibbs property and the $\mathbf{H}_{qt}$-Brownian Gibbs property respectively with $\mathbf{H}_t$ given by \eqref{eq:Ht}.  Conditioned on the end points of $\h_{pt, \uparrow}^{(1)}(\pm Mp^{-2/3})$ and $\h_{qt, \downarrow}^{(1)}(\pm Mq^{-2/3})$ and the second curves $\h_{pt, \uparrow}^{(2)}(\cdot)$ and $\h_{qt, \downarrow}^{(2)}(\cdot)$, the laws of $\h_{pt,\uparrow}^{(1)}(\cdot)$, and $\h_{pt,\uparrow}^{(1)}(\cdot)$ are absolutely continuous w.r.t.~Brownian bridges with appropriate end points. This conditioning is done in Subsection \ref{sb:large}.
		
		\item We then condition further on $\textit{{Max data}}$ : $\md_{p,t}^M, \h_{pt,\uparrow}^{(1)}((pt)^{-2/3}\md_{p,t}^M),\h_{qt,\downarrow}^{(1)}((qt)^{-2/3}\md_{p,t}^M)$. Under this conditioning, via the decomposition result in Proposition \ref{propA}, the underlying Brownian bridges mentioned in the previous point, when viewed from the joint maximizer, becomes two-sided $\nonintbb$s defined in Definition \ref{def:nibb}. This viewpoint from the joint maximizer is given by  $Y_{M,t,\uparrow}^{(1)}, Y_{M,t,\downarrow}^{(1)}$. See Figure \ref{fig:dyson1} for more details.

		\item We emphasize the fact that the deduction of $\nonintbb$s done above is only for the underlying Brownian law. One still needs to analyze the Radon-Nikodym (RN) derivative. As we are interested in the laws of $Y_{M,t,\uparrow}^{(1)}, Y_{M,t,\downarrow}^{(1)}$ on an interval of order $t^{-2/3}$ around $\Phi$, we analyze the RN derivative only on a small interval around $\Phi$.  To be precise, we consider a slightly larger yet vanishing interval of length $2t^{-\alpha}$ for $\alpha=\frac16$  around the random point $\Phi$. We show that the RN derivative on this small random patch is close to $1$. Thus upon further conditioning on the boundary data of this random small interval, the trajectories of $Y_{M,t,\uparrow}^{(1)}$ and $Y_{M,t,\downarrow}^{(1)}$ defined in \eqref{eq:yupdown} around $\Phi$ turns out to be close to two-sided $\nonintbb$ with appropriate (random) endpoints.
		\item Finally, we zoom further into a tiny interval of order $O(t^{-2/3})$ symmetric around the random point $\Phi$. Utilizing Lemma \ref{l:dtob}, we convert the two-sided $\nonintbb$s to two-sided $\dbm$s. 
	\end{itemize}

	We now provide an outline of the rest of the subsections. In Subsection \ref{sb:gtol} we reduce our proof from understanding laws around global maximizers to that of local maximizers. As explained in the above road-map, the proof follows by performing several successive conditioning. On a technical level, this requires defining several high probability events on which we can carry out our conditional analysis. These events are all defined in Subsection \ref{sb:nice} and are claimed to happen with high probability in Lemma \ref{l:nice}. We then execute the first layer of conditioning in Subsection \ref{sb:large}. The two other layers of conditioning are done in Subsection \ref{sb:small}. Lemma \ref{l:rn1} and Lemma \ref{l:closetod} are the precise technical expressions for the heuristic claims in the last two bullet points of the road-map. Assuming them, we complete the final steps of the proof in Subsection \ref{sb:ublb}. Proof of Lemma \ref{l:nice} is then presented in Subsection \ref{sb:pfnice}. Finally, in Subsection \ref{sb:2lem}, we prove the remaining lemmas: Lemma \ref{l:rn1} and \ref{l:closetod}.

	\subsubsection{Global to Local maximizer} \label{sb:gtol} We now fill out the technical details of the road-map presented in the previous subsection. Fix any $a > 0$. Consider any Borel set $A$ of $C([-a, a]\to \R^2)$ which is a continuity set of a two-sided $\dbm$ $\calD(\cdot)$ restricted to $[-a, a].$  By Portmanteau theorem, it is enough to show that 
	\begin{align}\label{wts}
		\Pr((D_1(\cdot,t),D_2(\cdot,t))\in A) \rightarrow \Pr(\calD(\cdot)\in A),
	\end{align}
	where $D_1,D_2$ are defined in \eqref{def:d1d2}. In this subsection, we describe how it suffices to check \eqref{wts} with $\md_{p,t}^M$. Recall $D_{M,t,\uparrow}(\cdot),D_{M,t,\downarrow}(\cdot)$ from \eqref{eq:dtoyM}.  We claim that for all $M>0$:
	\begin{align}
		\label{wts1}
		\lim_{t\to \infty}\Pr((D_{M,t,\uparrow}(\cdot),D_{M,t,\downarrow}(\cdot))\in A) \rightarrow \Pr(\calD(\cdot)\in A).
	\end{align}
	Note that when $\md_{p,t}^{\infty}=\md_{p,t}^{M}$,  $(D_{M,t,\uparrow}(\cdot),D_{M,t,\downarrow}(\cdot))$ is exactly equal to $$t^{1/3}Y_{\infty,t,\uparrow}^{(1)}\big(t^{-2/3}(\md_{p,t}^{\infty}+\cdot)\big) \ , \ t^{1/3}Y_{\infty,t,\downarrow}^{(1)}\big(t^{-2/3}(\md_{p,t}^{\infty}+\cdot)\big)$$ which via \eqref{eq:dtoy} is same in distribution as $D_1(\cdot,t),D_2(\cdot,t)$. Thus, 
	\begin{align*}
		{ \left|\Pr((D_{1}(\cdot,t),D_{2}(\cdot,t))\in A)-\Pr((D_{M,t,\uparrow}(\cdot),D_{M,t,\downarrow}(\cdot))\in A)\right| \le 2\Pr(\md_{p,t}\neq \md_{p,t}^M).}
	\end{align*}
	Now given any $\e>0$, by Lemma \ref{lem1}, we can take $M=M(\e)>0$ large enough so that $2\Pr(\md_{p,t}\neq \md_{p,t}^M) \le \e$. Then upon taking $t\to \infty$ in the above equation, in view of \eqref{wts1}, we see that 
	$$\limsup_{t\to\infty} \left||\Pr((D_{1}(\cdot,t),D_{2}(\cdot,t))\in A)-\Pr((\calD(\cdot)\in A)\right| \le \e.$$
	As $\e$ is arbitrary, this proves \eqref{wts}. The rest of the proof is now devoted in proving \eqref{wts1}.

	\subsubsection{Nice events} \label{sb:nice} In this subsection,
	we focus on defining several events that are collectively `nice' in the sense that they happen with high probability. We fix an $M>0$ for the rest of the proof and work with the local maximizer $\md_{p,t}^M$ defined in \eqref{eq:localmax}. We will also make use of the notation $\Phi$ defined in \eqref{def:phi} heavily in this and subsequent subsections. We now proceed to define a few events based on the location and value of the maximizer and values at the endpoints of an appropriate interval. Fix any arbitrary $\delta>0$. Let us consider the event:
	\begin{align}
		\label{eq:armx}
		\m{ArMx}(\delta):=\left\{\Phi \in [-M+\delta,M-\delta]\right\}.
	\end{align}
	The $\m{ArMx}(\delta)$ controls the location of the local maximizer $\Phi$. Set $\alpha=\frac16$. We define tightness event that corresponds to the boundary of the interval of length $2t^{-\alpha}$ around $\Phi:$
	\begin{align}
		\label{bdup}
		\m{Bd}_{\uparrow}(\delta) & :=\m{Bd}_{+,\uparrow}(\delta)\cap\m{Bd}_{-,\uparrow}(\delta), \quad
		\m{Bd}_{\downarrow}(\delta) := \m{Bd}_{+,\downarrow}(\delta)\cap\m{Bd}_{-,\downarrow}(\delta),
	\end{align}
	where
	\begin{align}
		\label{bdpmup}
		\m{Bd}_{\pm,\uparrow}(\delta)& :=\left\{\left|\h_{pt,\uparrow}^{(1)}\big(p^{-2/3}(\Phi\pm t^{-\alpha})\big)-\h_{pt,\uparrow}^{(1)}\big(\Phi p^{-2/3})\right|\le \tfrac1{\delta}t^{-\alpha/2}\right\} \\ \nonumber
		\m{Bd}_{\pm,\downarrow}(\delta)& :=\left\{\left|\h_{qt,\downarrow}^{(1)}\big(q^{-2/3}(\Phi\pm t^{-\alpha})\big)-\h_{qt,\downarrow}^{(1)}\big(\Phi q^{-2/3})\right|\le \tfrac1{\delta}t^{-\alpha/2}\right\}, 
	\end{align}
	Finally we consider the gap events that provide a gap between the first curve and the second curve for each of the line ensemble:
	\begin{align}
		\label{gapup}
		\m{Gap}_{M,\uparrow}(\delta) & := \left\{p^{1/3}\h_{pt,\uparrow}^{(1)}\big(\Phi p^{-2/3}\big) \ge  p^{1/3}\h_{pt,\uparrow}^{(2)}\big(\Phi p^{-2/3}\big)+\delta\right\}, \\ \label{gapdn}
		\m{Gap}_{M,\downarrow}(\delta) & := \left\{q^{1/3}\h_{qt,\downarrow}^{(1)}\big(\Phi q^{-2/3}\big) \ge  q^{1/3}\h_{qt,\downarrow}^{(2)}\big(\Phi q^{-2/3}\big)+\delta\right\}.
	\end{align}
	We next define the `rise' events which roughly says the second curves $\h_{pt,\uparrow}^{(1)}$ and $\h_{qt,\downarrow}^{(2)}$ of the line ensembles does not rise too much on a small interval of length $2t^{-\alpha}$ around $\Phi p^{-2/3}$ and $\Phi q^{-2/3}$ respectively. 
	\begin{align}
		\label{riseup}
		\m{Rise}_{M,\uparrow}(\delta) & :=\left\{\sup_{x\in [-t^{-\alpha},t^{-\alpha}]} p^{1/3}\h_{pt,\uparrow}^{(2)}\big(\Phi p^{-2/3}+x\big) \le  p^{1/3}\h_{pt,\uparrow}^{(2)}\big(\Phi p^{-2/3}\big)+\tfrac{\delta}{4}\right\}, \\ \label{risedn}
		\m{Rise}_{M,\downarrow}(\delta) & :=\left\{\sup_{x\in [-t^{-\alpha},t^{-\alpha}]} q^{1/3}\h_{qt,\downarrow}^{(2)}\big(\Phi q^{-2/3}+x\big) \le  q^{1/3}\h_{pt,\downarrow}^{(2)}\big(\Phi q^{-2/3}\big)+\tfrac{\delta}{4}\right\}.
	\end{align}
	$\m{Bd}$, $\m{Gap}$, $\m{Rise}$ type events and their significance are discussed later in Subsection \ref{sb:pfnice} in greater details. See also Figure \ref{fig:gap} and its caption for explanation of some of these events. We put all the above events into one final event:
	\begin{align}
		\label{def:nice}
		\Nice_{M}(\delta):=\left\{\m{ArMx}(\delta)\cap \bigcap_{x\in \{\uparrow,\downarrow\}} \m{Bd}_{x}(\delta) \cap  \m{Gap}_{M,x}(\delta) \cap  \m{Rise}_{M,x}(\delta)\right\}.
	\end{align}
	All the above events are dependent on $t$. But we have suppressed this dependence from the notations. The $\Nice_{M}(\delta)$ turns out to be a favorable event. We isolate this fact as a lemma below.
	\begin{lemma}\label{l:nice} For any $M>0$, under the above setup we have
		\begin{align}
			\label{eq:nice}
			\liminf_{\delta\downarrow 0}\liminf_{t\to\infty}\Pr_t\left(\Nice_{M}(\delta)\right)=1.
		\end{align}
	\end{lemma}
	We postpone the proof of this technical lemma to Section \ref{sb:pfnice} and for the moment we continue with the current proof of Proposition \ref{p:dyson} assuming its validity.
	\subsubsection{Conditioning with respect to large boundaries} \label{sb:large} As alluded in Subsection \ref{sb:ideas}, the proof involves conditioning on different $\sigma$-fields successively. We now specify all the different $\sigma$-fields that we will use throughout the proof. Set $\alpha=\frac16$. We consider the random interval 
	\begin{align}
		\label{eq:Kt}
		K_t:=(\Phi-t^{-\alpha},\Phi+t^{-\alpha}).
	\end{align}
	Let us define:
	\begin{align}
		\calF_1 & := \sigma\left(\left\{\h_{pt, \uparrow}^{(1)}(p^{-2/3}x), \h_{qt, \downarrow}^{(1)}(q^{-2/3}x)\right\}_{x\in (-M,M)^c}, \left\{\h_{pt, \uparrow}^{(2)}(x),\h_{qt, \downarrow}^{(2)}(x)\right\}_{x\in \R}\right) \label{sf1} \\ \label{sigfields} \calF_2 & := \sigma\left(\Phi, \h_{pt,\uparrow}^{(1)}(\Phi p^{-2/3}),\h_{qt,\downarrow}^{(1)}(\Phi q^{-2/3})\right), \\ \calF_3 & :=\sigma\left(\left\{\h_{pt, \uparrow}^{(1)}(p^{-2/3}x),\h_{qt, \downarrow}^{(1)}(q^{-2/3}x)\right\}_{x\in K_t^c}\right) \label{sf3}.
	\end{align}
	In this step we perform conditioning w.r.t.~ $\calF_1$ for the expression on the l.h.s.~of \eqref{wts1}.  We denote $\Pr_t(A):=\Pr\big((D_{M,t,\uparrow}(\cdot),D_{M,t,\downarrow}(\cdot))\in A\big)$. Taking the $\Nice_{M}(\delta)$ event defined in \eqref{def:nice} under consideration, upon conditioning with $\calF_1$ we have the following upper and lower bounds:
	\begin{align}
		\label{lwbd}
		\Pr_t(A) & \ge \Pr_t(\Nice_{M}(\delta),A) = \Ex_t\left[\Pr_t(\Nice_M(\delta), A\mid \calF_1)\right], \\ \label{upbd}
		\Pr_t(A) & \le \Pr_t(\Nice_{M}(\delta),A)+\Pr_t(\neg\Nice_{M}(\delta))= \Ex_t\left[\Pr_t(\Nice_M(\delta), A\mid \calF_1)\right]+\Pr_t(\neg\Nice_{M}(\delta)).
	\end{align}
	Note that the underlying measure consists of the mutually independent $\h_{pt, \uparrow}^{(1)}(\cdot)$ and $\h_{qt, \downarrow}^{(1)}(\cdot)$ which by Proposition \ref{line-ensemble} satisfy $\textbf{H}_{pt}$ and $\textbf{H}_{qt}$ Brownian Gibbs property respectively. Applying the respectively Brownian Gibbs properties and  following \eqref{eq:bgibbs} we have 
	\begin{align}\label{eq:cond1}
		\Pr_{t}(\Nice_M(\delta), A \mid \calF_1)  = \frac{\Ex_{\operatorname{free},t}[\ind_{\Nice_M(\delta), A }W_{\uparrow}W_{\downarrow}]}{\Ex_{\operatorname{free},t}[W_{\uparrow}W_{\downarrow}]}.
	\end{align}
	Here 
	\begin{align}\label{eq:cond2}
		W_{\uparrow}:= \exp\left(-t^{2/3}\int_{-M}^{M}\exp\left(t^{1/3}\big[p^{1/3}\h_{pt, \uparrow}^{(2)}(p^{-2/3}x)- p^{1/3}\h_{pt, \uparrow}^{(1)}(p^{-2/3}x)\big]\right)\d x\right)
	\end{align}
	and 
	\begin{align}\label{eq:cond3}
		W_{\downarrow}:= \exp\left(-t^{2/3}\int_{-M}^{M}\exp\left(t^{1/3}\big[q^{1/3}\h_{qt, \downarrow}^{(2)}(q^{-2/3}x)- q^{1/3}\h_{qt, \downarrow}^{(1)}(q^{-2/3}x)\big]\right)\d x\right).
	\end{align}
	In \eqref{eq:cond1}, $\Pr_{\operatorname{free}, t}$ and $\Ex_{\operatorname{free}, t}$ are the  probability and the expectation operator respectively corresponding to the joint `free' law for $(p^{1/3}\h_{pt,\uparrow}(p^{-2/3}x),q^{1/3}\h_{qt,\downarrow}(q^{-2/3}x))_{x\in[-M,M]}$ which by Brownian scaling is given by a pair of independent Brownian bridges $(\B_1(\cdot), \B_2(\cdot))$  on $[-M,M]$ with starting points $(p^{1/3}\h_{pt, \uparrow}(-Mp^{-2/3}), q^{1/3}\h_{qt, \downarrow}(-Mq^{-2/3}))$ and endpoints $(q^{1/3}\h_{pt, \uparrow}(Mp^{-2/3}), q^{1/3}\h_{qt, \downarrow}(Mq^{-2/3})).$

	\subsubsection{Conditioning with respect to maximum data and small boundaries} \label{sb:small}
	In this subsection we perform conditioning on the numerator of r.h.s.~of \eqref{eq:cond1} w.r.t.~$\calF_2$ and $\calF_3$ defined in \eqref{sigfields} and \eqref{sf3}.  Recall that by Proposition \ref{propA}, upon conditioning Brownian bridges on $\calF_2$,  the conditional laws around the joint local maximizer $\Phi$ over $[-M,M]$ is now given by two $\nonintbb$s (defined in Definition \ref{def:nibb}) with appropriate lengths and endpoints. Indeed, based on Proposition \ref{propA}, given $\calF_1,\calF_2$, we may construct the conditional laws for the two functions on $[-M,M]$:
	
	\begin{definition}[$\m{Nlarge}$ Law]\label{def:nlarge}
		Consider two independent $\nonintbb$ $\V_{\ell}^{\m{large}}$ and $\V_{r}^{\m{large}}$ with following description:
		\begin{enumerate}
			\item $\V_{\ell}^{\m{large}}$ is a $\nonintbb$ on $[0,\Phi+M]$ ending at 
			$$\left(p^{1/3}\left[\h_{pt,\uparrow}^{(1)}(\Phi p^{-2/3})-\h_{pt,\uparrow}^{(1)}(-Mp^{-2/3})\right],q^{1/3}\left[\h_{qt,\downarrow}^{(1)}(-Mq^{-2/3})-\h_{qt,\downarrow}^{(1)}(\Phi q^{-2/3})\right]\right),$$
			\item $\V_{r}^{\m{large}}$ is a $\nonintbb$ on $[0,M-\Phi]$ ending at 
			$$\left(p^{1/3}\left[\h_{pt,\uparrow}^{(1)}(\Phi p^{-2/3})-\h_{pt,\uparrow}^{(1)}(Mp^{-2/3})\right],q^{1/3}\left[\h_{qt,\downarrow}^{(1)}(Mq^{-2/3})-\h_{qt,\downarrow}^{(1)}(\Phi q^{-2/3})\right]\right).$$
		\end{enumerate}
		We then define ${B}^{\m{large}}:[-M,M]\to \R^2$ as follows:
		\begin{align*}
			{B}^{\m{large}}(x)=\begin{cases}
				\V_{\ell}(\Phi-x) & x\in [-M,\Phi] \\
				\V_{r}(x-\Phi) & x\in [\Phi,M] 
			\end{cases}.
		\end{align*}
		We denote the expectation and probability operator under above law for ${B}^{\m{large}}$ (which depends on $\calF_1,\calF_2$) as $\Ex_{\m{Nlarge|2,1}}$ and $\Pr_{\m{Nlarge|2,1}}$.   
	\end{definition}
	Thus we may write 
	\begin{align}
		\label{part1}
		\Ex_{\operatorname{free},t}[\ind_{\Nice_M(\delta), A }W_{\uparrow}W_{\downarrow}] & = \Ex_{\operatorname{free},t}[\Ex_{\m{Nlarge|2,1}}[\ind_{\Nice_M(\delta), A }W_{\uparrow}W_{\downarrow}]].
	\end{align}
	Since $\nonintbb$s are Markovian, we may condition further upon $\calF_3$ to get $\nonintbb$s again but on a smaller interval. To precisely define the law, we now give the following definitions: 
	
	\begin{definition}[$\m{Nsmall}$ law]\label{def:nsmall}
		Consider two independent $\nonintbb$ $\V_{\ell}^{\m{small}}$ and $\V_{r}^{\m{small}}$ with the following descriptions:
		\begin{enumerate}
			\item  $\V_{\ell}^{\m{small}}$ is a $\nonintbb$ on $[0,t^{-\alpha}]$ ending at 
			$$\left(p^{1/3}\left[\h_{pt,\uparrow}^{(1)}(\Phi p^{-2/3})-\h_{pt,\uparrow}^{(1)}(p^{-2/3}(\Phi-t^{-\alpha}))\right],q^{1/3}\left[\h_{qt,\downarrow}^{(1)}(q^{-2/3}(\Phi-t^{-\alpha}))-\h_{qt,\downarrow}^{(1)}(\Phi q^{-2/3})\right]\right),$$
			\item  $\V_{r}^{\m{small}}$ is a $\nonintbb$ on $[0,t^{-\alpha}]$ ending at 
			$$\left(p^{1/3}\left[\h_{pt,\uparrow}^{(1)}(\Phi p^{-2/3})-\h_{pt,\uparrow}^{(1)}(p^{-2/3}(\Phi+t^{-\alpha}))\right],q^{1/3}\left[\h_{qt,\downarrow}^{(1)}(q^{-2/3}(\Phi+t^{-\alpha}))-\h_{qt,\downarrow}^{(1)}(\Phi q^{-2/3})\right]\right).$$
		\end{enumerate}
		We then define ${B}^{\m{small}}:[\Phi+t^{-\alpha},\Phi-t^{-\alpha}]\to \R^2$ as follows:
		\begin{align*}
			{B}^{\m{small}}(x)=\begin{cases}
				\V_{\ell}(\Phi-x) & x\in [\Phi-t^{-\alpha},\Phi] \\
				\V_{r}(x-\Phi) & x\in [\Phi,\Phi+t^{-\alpha}] 
			\end{cases}.
		\end{align*}
		We denote the the expectation and probability operators under the above law for ${B}^{\m{small}}$ (which depends on $\calF_1,\calF_2,\calF_3$) as $\Ex_{\m{Nsmall|3,2,1}}$ and $\Pr_{\m{Nsmall|3,2,1}}$ respectively. 
	\end{definition}

	We thus have
	\begin{align}
		\label{part2}
		\mbox{r.h.s.~of \eqref{part1}} & = \Ex_{\operatorname{free},t}[\ind_{\Nice_M(\delta)}\Ex_{\m{Nsmall|3,2,1}}[\ind_{A} W_{\uparrow}W_{\downarrow}]].
	\end{align}
	The $\ind_{\Nice_M(\delta)}$ comes of the interior expectation above as $\Nice_{M}(\delta)$ is measurable w.r.t.~$\calF_1\cup\calF_2\cup \calF_3$ (see its definition in \eqref{def:nice}).

	Next note that due to the definition of $W_{\uparrow}, W_{\downarrow}$ from \eqref{eq:cond2} and \eqref{eq:cond3}, we may extract certain parts of it which are measurable w.r.t.~$\calF_1\cup\calF_2\cup \calF_3$. Indeed, we can write $W_{\uparrow}=W_{\uparrow,1}W_{\uparrow,2}$ and $W_{\downarrow}=W_{\downarrow,1}W_{\downarrow,2}$ where
	\begin{align}
		\label{wu1}
		& W_{\uparrow,1}:= \exp\left(-t^{2/3}\int_{K_t}\exp\left(t^{1/3}\big[p^{1/3}\h_{pt, \uparrow}^{(2)}(p^{-2/3}x)- p^{1/3}\h_{pt, \uparrow}^{(1)}(p^{-2/3}x)\big]\right)\d x\right)\\&W_{\uparrow, 2}:= \exp\left(-t^{2/3}\int_{[-M,M]\cap K_t^c}\exp\left(t^{1/3}\big[p^{1/3}\h_{pt, \uparrow}^{(2)}(p^{-2/3}x)- p^{1/3}\h_{pt, \uparrow}^{(1)}(p^{-2/3}x)\big]\right)\d x\right), \nonumber
	\end{align}
	and
	\begin{align}
		&W_{\downarrow,1}:= \exp\left(-t^{2/3}\int_{K_t}\exp\left(t^{1/3}\big[q^{1/3}\h_{qt, \downarrow}^{(2)}(q^{-2/3}x)- q^{1/3}\h_{qt, \downarrow}^{(1)}(q^{-2/3}x)\big]\right)\d x\right). \label{wd1} \\ \nonumber
		&W_{\downarrow,2}:= \exp\left(-t^{2/3}\int_{[-M,M]\cap K_t^c}\exp\left(t^{1/3}\big[q^{1/3}\h_{qt, \downarrow}^{(2)}(q^{-2/3}x)- q^{1/3}\h_{qt, \downarrow}^{(1)}(q^{-2/3}x)\big]\right)\d x\right),
	\end{align}
	where recall $K_t$ from \eqref{eq:Kt}. The key observation is that $W_{\uparrow,2},W_{\downarrow, 2}$ are measurable w.r.t.~$\calF_1\cup\calF_2\cup \calF_3$. Thus we have
	\begin{align}
		\label{part3}
		\mbox{r.h.s.~of \eqref{part2}} & = \Ex_{\operatorname{free},t}[\ind_{\Nice_M(\delta)}W_{\uparrow,2}W_{\downarrow,2} \cdot \Ex_{\m{Nsmall|3,2,1}}[\ind_{A}W_{\uparrow,1}W_{\downarrow,1}]]. 
	\end{align}
	\begin{remark}
		It is crucial to note that in \eqref{part2} the event $\Nice_{M}(\delta)$ includes the event $\m{ArMx}(\delta)$ defined in \eqref{eq:armx}. Indeed, the $\m{ArMx}(\delta)$ event is measurable w.r.t.~$\calF_1\cup\calF_2$ and ensures that $[\Phi-t^{-\alpha},\Phi+t^{-\alpha}] \subset [-M,M]$ for all large enough $t$, which is essential for going from $\m{Nlarge}$ law to $\m{Nsmall}$ law. Thus such a decomposition is not possible for $\Ex_{\operatorname{free},t}[W_{\uparrow}W_{\downarrow}]$ which appears in the denominator of r.h.s.~of \eqref{eq:cond1}. Nonetheless, we may still provide a lower bound for $\Ex_{\operatorname{free},t}[W_{\uparrow}W_{\downarrow}]$ as follows:
		\begin{align}
			\label{deno}
			\Ex_{\operatorname{free},t}[W_{\uparrow}W_{\downarrow}] \ge \Ex_{\operatorname{free},t}[\ind_{\Nice_M(\delta)}W_{\uparrow}W_{\downarrow}] =\Ex_{\operatorname{free},t}[W_{\uparrow,2}W_{\downarrow,2}\ind_{\Nice_M(\delta)} \cdot \Ex_{\m{Nsmall|3,2,1}}[W_{\uparrow,1}W_{\downarrow,1}]]. 
		\end{align}
	\end{remark}
	
	\bigskip

	With the deductions in \eqref{part3} and \eqref{deno}, we now come to the task of analyzing $W_{\uparrow,1}W_{\downarrow,1}$ under $\m{Nsmall}$ law. The following lemma ensures that on $\Nice_M(\delta)$, $W_{\uparrow,1}W_{\downarrow,1}$ is close to $1$ under $\m{Nsmall}$ law.
	\begin{lemma}
		\label{l:rn1} There exist $t_0(\delta)>0$ such that for all $t\ge t_0$ we have
		\begin{align}\label{eq:l2}
			\ind_{\Nice_M(\delta)}\Pr_{\m{Nsmall|3,2,1}}(W_{\uparrow, 1}W_{\downarrow, 1}>1-\delta)\ge \ind_{\Nice_M(\delta)}(1- \delta).
		\end{align}
	\end{lemma}
	This allow us to ignore $W_{\uparrow,1}W_{\downarrow,1}$, in $\Ex_{\m{Nsmall|3,2,1}}[\ind_{A}W_{\uparrow,1}W_{\downarrow,1}]$. Hence it suffices to study $\Pr_{\m{Nsmall}|3,2,1}(A)$. The following lemma then compares this conditional probability with that of $\dbm$.
	\begin{lemma}
		\label{l:closetod}
		There exist $t_0(\delta)>0$ such that for all $t\ge t_0$ we have
		\begin{align}\label{eq:l3}
			\ind_{\Nice_M(\delta)}|\Pr_{\m{Nsmall|3,2,1}}(A)-\tau(A)|\le \ind_{\Nice_M(\delta)}\cdot \delta,
		\end{align}
		where $\tau(A):=\Pr(\calD(\cdot)\in A)$, $\calD$ being a two-sided $\dbm$ defined in the statement of Proposition \ref{p:dyson}.
	\end{lemma}
	We prove these two lemmas in Section \ref{sb:2lem}. For now, we proceed with the current proof of \eqref{wts1} in the next section.
	
	\subsubsection{Matching Lower and Upper Bounds} \label{sb:ublb} In this subsection, we complete the proof of \eqref{wts1} by providing matching lower and upper bounds in the two steps below. We assume throughout this subsection that $t$ is large enough, so that \eqref{eq:l2} and \eqref{eq:l3} holds.
	
	\medskip
	
	\noindent\textbf{Step 1: Lower Bound.} We start with \eqref{lwbd}. Following the expression in \eqref{eq:cond1}, and our deductions in \eqref{part1}, \eqref{part2}, \eqref{part3} we see that
	\begin{align}
		\nonumber \Pr_t(A) & \ge \Ex_t\left[\Pr_t(\Nice_M(\delta), A\mid \calF_1)\right] \\ \label{627.0} & =\Ex \left[\frac{\Ex_{\operatorname{free},t}[\ind_{\Nice_M(\delta)}W_{\uparrow,2}W_{\downarrow,2} \cdot \Ex_{\m{Nsmall|3,2,1}}[\ind_{A}W_{\uparrow,1}W_{\downarrow,1}]] }{\Ex_{\operatorname{free},t}[W_{\uparrow}W_{\downarrow}]}\right] \\ & \ge (1-\delta)\Ex_t \left[\frac{\Ex_{\operatorname{free},t}[\ind_{\Nice_M(\delta)}W_{\uparrow,2}W_{\downarrow,2} \cdot \Pr_{\m{Nsmall|3,2,1}}(A,W_{\uparrow,1}W_{\downarrow,1}>1-\delta)] }{\Ex_{\operatorname{free},t}[W_{\uparrow}W_{\downarrow}]}\right] \label{627.1}
	\end{align}
	where in the last inequality we used the fact $W_{\uparrow,1}W_{\downarrow,1}\le 1$. Now applying Lemma \ref{l:rn1} and Lemma \ref{l:closetod} successively we get 
	\begin{align*}
		& \ind_{\Nice_M(\delta)}\Pr_{\m{Nsmall|3,2,1}}(A,W_{\uparrow,1}W_{\downarrow,1}>1-\delta) \\ & \ge \ind_{\Nice_M(\delta)}[\Pr_{\m{Nsmall|3,2,1}}(A)-\Pr_{\m{Nsmall|3,2,1}}(W_{\uparrow,1}W_{\downarrow,1}\le 1-\delta)] \\ & \ge \ind_{\Nice_M(\delta)}[\Pr_{\m{Nsmall|3,2,1}}(A)-\delta] \\ & \ge \ind_{\Nice_M(\delta)}[\tau(A)-2\delta] 
	\end{align*}
	where recall $\tau(A)=\Pr(\calD(\cdot)\in A)$. As $W_{\uparrow,1}W_{\downarrow,1}\le 1$ and probabilities are nonnegative, following the above inequalities we have
	\begin{align*}
		\ind_{\Nice_M(\delta)}\Pr_{\m{Nsmall|3,2,1}}(A,W_{\uparrow,1}W_{\downarrow,1}>1-\delta)\ge \max\{0,\tau(A)-2\delta\} \ind_{\Nice_M(\delta)}W_{\uparrow,1}W_{\downarrow,1}.
	\end{align*}
	Substituting the above bound back to \eqref{627.1} and using the fact that $W_{\uparrow,2}W_{\downarrow,2}W_{\uparrow,1}W_{\downarrow,1}=W_{\uparrow}W_{\downarrow}$, we get
	\begin{align*}
		\Pr_t(A)  & \ge (1-\delta)\max\{0,\tau(A)-2\delta\}\Ex_t\left[\frac{\Ex_{\operatorname{free},t}[\ind_{\Nice_M(\delta)}W_{\uparrow}W_{\downarrow} ]}{\Ex_{\operatorname{free},t}[W_{\uparrow}W_{\downarrow}]}\right] \\ & =  (1-\delta)\max\{0,\tau(A)-2\delta\}\Pr_t(\Nice_M(\delta)). 
	\end{align*}
	In view of Lemma \ref{l:nice}, taking $\liminf_{t\to \infty}$ followed by $\liminf_{\delta\downarrow 0}$ we get that $\liminf_{t\to\infty}\Pr_t(A) \ge \tau(A)$. This proves the lower bound. 

	\medskip
	
	\noindent\textbf{Step 2: Upper Bound.} We start with \eqref{upbd}. Using the equality in \eqref{627.0} we get
	\begin{align} \nonumber
		\Pr_t(A) & \le \Ex \left[\frac{\Ex_{\operatorname{free},t}[\ind_{\Nice_M(\delta)}W_{\uparrow,2}W_{\downarrow,2} \cdot \Ex_{\m{Nsmall|3,2,1}}[\ind_{A}W_{\uparrow,1}W_{\downarrow,1}]] }{\Ex_{\operatorname{free},t}[W_{\uparrow}W_{\downarrow}] }\right] + \Pr_t(\neg\Nice_M(\delta)) \\ & \nonumber \le \Ex \left[\frac{\Ex_{\operatorname{free},t}[\ind_{\Nice_M(\delta)}W_{\uparrow,2}W_{\downarrow,2} \cdot \Pr_{\m{Nsmall|3,2,1}}(A)] }{\Ex_{\operatorname{free},t}[W_{\uparrow}W_{\downarrow}] }\right] + \Pr_t(\neg\Nice_M(\delta)) \\ & \le (\tau(A)+\delta)\Ex \left[\frac{\Ex_{\operatorname{free},t}[\ind_{\Nice_M(\delta)}W_{\uparrow,2}W_{\downarrow,2}] }{\Ex_{\operatorname{free},t}[W_{\uparrow}W_{\downarrow}] }\right] + \Pr_t(\neg\Nice_M(\delta)). \label{627.5}
	\end{align}
	Let us briefly justify the inequalities presented above. Going from first line to second line we used the fact $W_{\uparrow,1}W_{\downarrow,1}\le 1$. The last inequality follows from Lemma \ref{l:closetod} where recall that $\tau(A)=\Pr(\calD(\cdot)\in A).$ Now note that by Lemma \ref{l:rn1}, on $\Nice_{M}(\delta)$,
	\begin{align*}
		\Ex_{\m{Nsmall|3,2,1}}[W_{\uparrow,1}W_{\downarrow,1}] & \ge \Ex_{\m{Nsmall|3,2,1}}[\ind_{W_{\uparrow,1}W_{\downarrow,1}\ge 1-\delta} \cdot W_{\uparrow,1}W_{\downarrow,1}] \\ & \ge (1-\delta)\Pr_{\m{Nsmall|3,2,1}}({W_{\uparrow,1}W_{\downarrow,1}\ge 1-\delta}) \ge (1-\delta)^2.
	\end{align*}
	Using the expression from \eqref{deno} we thus have
	\begin{align*}
		\Ex_{\operatorname{free},t}[W_{\uparrow}W_{\downarrow}] & \ge \Ex_{\operatorname{free},t}[\ind_{\Nice_M(\delta)}W_{\uparrow,2}W_{\downarrow,2} \cdot \Ex_{\m{Nsmall|3,2,1}}[W_{\uparrow,1}W_{\downarrow,1}]] \\ & \ge (1-\delta)^2\Ex_{\operatorname{free},t}[\ind_{\Nice_M(\delta)}W_{\uparrow,2}W_{\downarrow,2}].
	\end{align*}
	Going back to \eqref{627.5}, this forces
	\begin{align*}
		\mbox{r.h.s.~of \eqref{627.5}} \le \frac{\tau(A)+\delta}{(1-\delta)^2}+\Pr_t(\neg\Nice_M(\delta)).
	\end{align*}
	In view of Lemma \ref{l:nice}, taking $\limsup_{t\to\infty}$, followed by $\limsup_{\delta\downarrow 0}$ in above inequality we get that $\limsup_{t\to\infty} \Pr_t(A) \le \tau(A).$ Along with the matching lower bound obtained in \textbf{Step 1} above, this establishes \eqref{wts1}.

	\subsubsection{Proof of Lemma \ref{l:nice}} \label{sb:pfnice} Recall  from \eqref{def:nice} that $\Nice_{M}(\delta)$ event is an intersection of several kinds of events. To show \eqref{eq:nice}, it suffices to prove the same for each of the events. That is, given an event $\m{E}$ which is part of $\Nice_M(\delta)$ we will show
	\begin{align}
		\label{eq:e}
		\limsup_{\delta\to \infty}\limsup_{t\to\infty}\Pr(\m{E})=1.
	\end{align}
	Below we analyze each such possible choices for $\m{E}$ separately.
	
	\medskip
	
	\noindent\underline{$\m{ArMx}(\delta)$ \textbf{event.}} Recall $\m{ArMx}(\delta)$ event from \eqref{eq:armx}. As noted in \eqref{mxfinite}, $$\md_{p,t}^M \stackrel{d}{\to} \argmax_{x\in [-M,M]} \calA(x),$$ where $\calA$ is defined in \eqref{ax}. Since $\calA$ restricted to $[-M,M]$ is absolutely continuous with Brownian motion with appropriate diffusion coefficients, $\argmax_{x\in [-M,M]} \calA(x)\in (-M,M)$ almost surely. In other words, maximum is not attained on the boundaries almost surely. But then
	\begin{equation}
		\label{eq:armxcal}
		\begin{aligned}
			\liminf_{\delta\downarrow0}\liminf_{t\to \infty} \Pr(\m{ArMx}(\delta)) & =\liminf_{\delta\downarrow0} \Pr(\argmax_{x\in [-M,M]} \calA(x)\in [-M+\delta,M-\delta]) \notag\\ & =\Pr(\argmax_{x\in [-M,M]} \calA(x)\in (-M,M))=1.
		\end{aligned}
	\end{equation}
	This proves \eqref{eq:e} with $\m{E}\mapsto \m{ArMx}(\delta).$
	
	\bigskip
	
	\noindent\underline{$\m{Bd}_{\uparrow}(\delta),\m{Bd}_{\downarrow}(\delta)$ \textbf{events.}} We first define
	\begin{equation*}
		\begin{aligned}
			\m{Tight}_{\pm,\uparrow}(\lambda):=\left\{p^{1/3}\left|\h_{pt,\uparrow}^{(1)}(\Phi p^{-2/3})-\h_{pt,\uparrow}^{(1)}(\pm Mp^{-2/3})\right|\le \tfrac1\lambda \right\}, \\ 
			\m{Tight}_{\pm,\downarrow}(\lambda):=\left\{q^{1/3}\left|\h_{qt,\downarrow}^{(1)}(\Phi q^{-2/3})-\h_{qt,\downarrow}^{(1)}(\pm Mq^{-2/3})\right|\le \tfrac1\lambda \right\},
		\end{aligned}
	\end{equation*}
	and set
	\begin{align}
		\label{def:sp}
		\m{Sp}(\lambda):=\m{ArMx}(\lambda)\cap  \m{Tight}_{+,\uparrow}(\lambda)\cap\m{Tight}_{-,\uparrow}(\lambda)\cap\m{Tight}_{+,\downarrow}(\lambda)\cap\m{Tight}_{-,\downarrow}(\lambda)
	\end{align}
	where $\m{ArMx}(\lambda)$ is defined in \eqref{eq:armx}. We claim that 
	\begin{align}
		\label{eq:sp}
		\limsup_{\lambda\downarrow 0}\limsup_{t\to\infty} \Pr(\neg \m{Sp}(\lambda)))=0.
	\end{align}
	Let us assume \eqref{eq:sp} for the time being and consider the main task of analyzing the probability of the events  $\m{Bd}_{\uparrow}(\delta),\m{Bd}_{\downarrow}(\delta)$ defined in \eqref{bdup}. We have $\m{Bd}_{\uparrow}(\delta)=\m{Bd}_{+\uparrow}(\delta)\cap \m{Bd}_{-,\uparrow}(\delta)$ where
	$\m{Bd}_{\pm,\uparrow}(\delta)$ is defined in \eqref{bdpmup}. Let us focus on $\m{Bd}_{+,\uparrow}(\delta)$. Recall the  $\sigma$-fields $\calF_1,\calF_2$ from \eqref{sf1} and \eqref{sigfields}. As described in Subsection \ref{sb:small}, upon conditioning on $\calF_1\cup\calF_2$, the conditional law on $[-M,M]$ are given by $\m{Nlarge}$ defined in Definition \ref{def:nlarge}, which are made up of $\nonintbb$s $V_{\ell}^{\m{large}},V_{r}^{\m{large}}$ defined in Definition \ref{def:nlarge}. 
	
	Note that applying Markov inequality conditionally we have
	\begin{align*}
		&  \ind_{\m{Sp}(\lambda)}\Pr\left(\m{Bd}_{+,\uparrow}(\delta)\mid \calF_1,\calF_2\right) \\ & = \ind_{\m{Sp}(\lambda)} \cdot \Pr\left(|\h_{pt,\uparrow}^{(1)}(p^{-2/3}(\Phi+t^{-\alpha}))-\h_{pt,\uparrow}^{(1)}(\Phi p^{-2/3})| >\tfrac1{\delta}t^{-\alpha/2} \mid \calF_1,\calF_2\right) \\ & \le  \ind_{\m{Sp}(\lambda)} \cdot \delta^2 t^{2\alpha} \cdot \Ex_{\m{Nlarge}|2,1}\left[[\V_{r,1}^{\m{large}}(p^{-2/3}t^{-\alpha})]^4\right] 
	\end{align*}
	However, on $\ind_{\m{Sp}(\lambda)}$, the $\nonintbb$ has length bounded away from zero and the endpoints are tight. Applying \eqref{db5} with $K\mapsto 2, t\mapsto 1,s\mapsto 0, n\mapsto p^{2/3}t^{\alpha}, M\mapsto 1/\lambda$,  for all large enough $t$ we get $\Ex_{\m{Nlarge}|2,1}\left[[\V_{r,1}^{\m{large}}(p^{-2/3}t^{-\alpha})]^4\right] \le \Con_{p,\lambda} t^{-2\alpha}$. Thus,
	\begin{align*}
		\limsup_{t\to\infty}\Pr\left(\neg\m{Bd}_{+,\uparrow}(\delta)\right) \le \limsup_{t\to\infty}\Pr(\neg \m{Sp}(\lambda))+\delta^2 {\Con_{p,\lambda}.}
	\end{align*}
	Taking $\delta\downarrow0$, followed by $\lambda\downarrow 0$, in view of \eqref{eq:sp} we get $\limsup_{\delta\downarrow 0}\limsup_{t\to \infty} \Pr(\neg\m{Bd}_{+,\uparrow}(\delta))=0$.
	Similarly one can conclude $\limsup_{\delta\downarrow 0}\limsup_{t\to \infty} \Pr(\neg\m{Bd}_{-,\uparrow}(\delta))=0$ Thus, this two together yields $\liminf_{\delta\downarrow 0}\liminf_{t\to\infty} \Pr(\m{Bd}_{\uparrow}(\delta))=1$. By exactly the same approach one can derive that $ \Pr(\m{Bd}_{\downarrow}(\delta))$ goes to $1$ under the same iterated limit. Thus it remains to show \eqref{eq:sp}. \\

	Let us recall from \eqref{def:sp} that $\m{Sp}(\lambda)$ event is composed of four tightness events and one event about the $\argmax$. We first claim that $\limsup_{\lambda\downarrow 0}\limsup_{t\to\infty}\Pr(\m{Tight}_{x,y}(\lambda))=1$ for  each $x\in \{+,-\}$ and $y\in \{\uparrow,\downarrow\}$.  The earlier analysis of $\m{ArMx}(\lambda)$ event in \eqref{eq:armxcal} then enforces \eqref{eq:sp}. Since all the tightness events are similar, it suffices to prove any one of them say $\m{Tight}_{+,\uparrow}$. By Proposition \ref{line-ensemble} we have the distributional convergence of $2^{1/3}\h_{pt,\uparrow}^{(1)}(2^{1/3}x)$ to $\calA_1(x)$ in the uniform-on-compact topology, where $\calA_1(\cdot)$ is the parabolic $\operatorname{Airy}_2$ process. As $\Phi\in [-M,M]$, we thus have
	\begin{align*}
		\limsup_{t\to\infty}\Pr(\m{Tight}_{+,\uparrow}(\lambda)) & \le \limsup_{t\to\infty}\Pr\left(p^{1/3}\sup_{x\in [-M,M]}\left|\h_{pt,\uparrow}^{(1)}(x p^{-2/3})-\h_{pt,\uparrow}^{(1)}(Mp^{-2/3})\right|\le \tfrac1\lambda\right)   \\ & =  \Pr\left(p^{1/3}\sup_{|x|\le 2^{-1/3}M}\left|\calA_1(x p^{-2/3})-\calA_1(2^{-1/3}Mp^{-2/3})\right|\le \tfrac{2^{1/3}}\lambda\right).
	\end{align*}
	For fixed $p,M$, by tightness of parabolic $\operatorname{Airy}_2$ process on a compact interval, the last expression goes to one as $\lambda \downarrow 0$, which is precisely what we wanted to show. 
	
	\bigskip
	
	\noindent\underline{$\m{Gap}_{M,\uparrow}(\delta),\m{Gap}_{M,\downarrow}(\delta)$  \textbf{events.}}  
	Recall the definitions of $\m{Gap}_{M, \uparrow}(\delta)$ and $\m{Gap}_{M, \downarrow}(\delta)$ from\eqref{gapup} and \eqref{gapdn}.We begin with the proof of    $\m{Gap}_{M, \uparrow}(\delta)$. Let $$\m{Diff}_{M,\uparrow}(\delta):= \left\{\inf_{|x|\le M}p^{1/3}\left(\h_{pt, \uparrow}^{(1)}(p^{-2/3}x)- \h_{pt, \uparrow}^{(2)}(p^{-2/3}x)\right) \ge \delta\right\}.$$ Note that $\Phi \in [-M,M]$. 
	Thus $\m{Gap}_{M, \uparrow}(\delta) \supset \m{Diff}_{M,\uparrow}(\delta).$ Thus to show \eqref{eq:e} with $\m{E}\mapsto \m{Gap}_{M, \uparrow}(\delta)$ it suffices to prove
	\begin{align}\label{lptm}
		\liminf_{\delta \downarrow 0}\liminf_{t \rightarrow \infty}\Pr(\m{Diff}_{M,\uparrow}(\delta)) = 1,
	\end{align}
	We recall from Proposition \ref{p:leconv} the distributional convergence of the KPZ line ensemble to the Airy line ensemble in the uniform-on-compact topology. By Skorokhod representation theorem, we may assume that our probability space is equipped with $\calA_1(\cdot)$ and $\calA_2(\cdot)$ such that almost surely as $t\to \infty$
	\begin{align}\label{sep0}
		\max_{i = 1, 2}\sup_{|x|\le Mp^{-2/3}}|2^{1/3}\h_{t,\uparrow}^{(i)}(2^{1/3}x)- \calA_i(x)|\rightarrow 0.
	\end{align} 
	We thus have
	\begin{align}
		\label{sep00}
		\liminf_{t \rightarrow \infty}\Pr(\m{Diff}_{M,\uparrow}(\delta))= \Pr\left(\inf_{|x|\le M 2^{-1/3}p^{-2/3}} p^{1/3}\left(\calA_1(x)- \calA_2(x)\right) \ge 2^{1/3}\delta\right).
	\end{align}

	As the Airy line ensemble is absolutely continuous w.r.t.~non-intersecting Brownian motions, it is strictly ordered with touching probability zero (see \eqref{eq:order}). Hence r.h.s.~of \eqref{sep00} goes to zero as $\delta\downarrow 0$. This proves \eqref{lptm}. The proof is similar for $\m{Gap}_{M, \downarrow}(\delta).$ 

	\bigskip
	
	\noindent\underline{$\m{Rise}_{M,\uparrow}(\delta),\m{Rise}_{M,\uparrow}(\delta)$ \textbf{events.}} Recall $\m{Rise}_{M,\uparrow}(\delta),\m{Rise}_{M,\uparrow}(\delta)$ events from \eqref{riseup} and \eqref{risedn}. Due to their similarities, we only analyze the $\m{Rise}_{M,\uparrow}(\delta)$ event. As with the previous case, we assume that our probability space is equipped with $\calA_1(\cdot)$ and $\calA_2(\cdot)$ (first two lines of the Airy line ensemble) such that almost surely as $t\to \infty$  \eqref{sep0} holds. 
	Applying union bound we have
	\begin{align*}
		\Pr\left(\neg \m{Rise}_M(\delta)\right) & \le \Pr\left(\sup_{{|x|\le Mp^{-2/3}}} p^{1/3}|2^{1/3}\h_{pt,\uparrow}^{(2)}(2^{1/3}x)-\calA_2(x)|\ge \tfrac{\delta}{16}\right) \\ & \hspace{2cm}+\Pr\left(\neg\m{Rise}_M(\delta),\sup_{|x|\le Mp^{-2/3}} p^{1/3}|2^{1/3}\h_{pt,\uparrow}^{(2)}(2^{1/3}x)-\calA_2(x)|\le \tfrac{\delta}{16}\right) \\ & \le \Pr\left(\sup_{|x|\le Mp^{-2/3}} p^{1/3}|2^{1/3}\h_{pt,\uparrow}^{(2)}(2^{1/3}x)-\calA_2(x)|\ge \tfrac{\delta}{16}\right) \\ & \hspace{2cm}+\Pr\bigg(\sup_{\substack{x,y\in [-M,M] \\ |x-y|\le t^{-\alpha}}} p^{1/3}|\calA_2(x)-\calA_2(y)|\ge \tfrac{\delta}{8}\bigg).
	\end{align*}
	In the r.h.s.~of above equation, the first term goes to zero as $t\to\infty$ by \eqref{sep0}. The second term on the other hand goes to zero as $t\to\infty$  by modulus of continuity estimates for Airy line ensemble from Proposition \ref{p:amodcon}. This shows, $\lim_{t\to\infty} \Pr(\m{Rise}_{M,\uparrow}(\delta))=1$. Similarly one has  $\lim_{t\to\infty} \Pr(\m{Rise}_{M,\downarrow}(\delta))=1$ as well. This proves \eqref{eq:e} for $\m{E} \mapsto \m{Rise}_{M,\uparrow}(\delta),\m{Rise}_{M,\downarrow}(\delta)$. \\

	We have thus shown \eqref{eq:e} for all the events listed in \eqref{def:nice}. This establishes \eqref{eq:nice} concluding the proof of Lemma \ref{l:nice}.

	\subsubsection{Proof of Lemma \ref{l:rn1} and \ref{l:closetod}} \label{sb:2lem}

	In this subsection we prove Lemma \ref{l:rn1} and \ref{l:closetod}.
	
	\medskip
	
	\noindent\textbf{Proof of Lemma \ref{l:rn1}.} Recall $W_{\uparrow,1}$ and $W_{\downarrow,1}$ from \eqref{wu1} and \eqref{wd1} respectively. We claim that for all large enough $t$, on $\Nice_M(\delta)$ we have
	\begin{align}
		\label{l2wts}
		\Pr_{\m{Nsmall}|3,2,1}(W_{\uparrow,1}>\sqrt{1-\delta}) \ge 1-\tfrac12\delta, \quad \Pr_{\m{Nsmall}|3,2,1}(W_{\downarrow,1}>\sqrt{1-\delta}) \ge 1-\tfrac12\delta
	\end{align}
	simultaneously. \eqref{eq:l2} then follows via union bound. Hence we focus on proving \eqref{l2wts}. In the proof below we only focus on first part of \eqref{l2wts} and the second one follows analogously. We now define the `sink' event:
	\begin{align}
		\label{sinkdef}
		\m{Sink}_{\uparrow}(\delta) & := \left\{ \inf_{x\in [-t^{-\alpha},t^{-\alpha}]} p^{1/3}\h_{pt,\uparrow}^{(1)}(\Phi p^{-2/3}+x) \ge p^{1/3}\h_{pt,\uparrow}^{(1)}(\Phi p^{-2/3})-\tfrac{\delta}{4} \right\}.
	\end{align}

	\begin{figure}
		\centering
		\includegraphics[width=9cm]{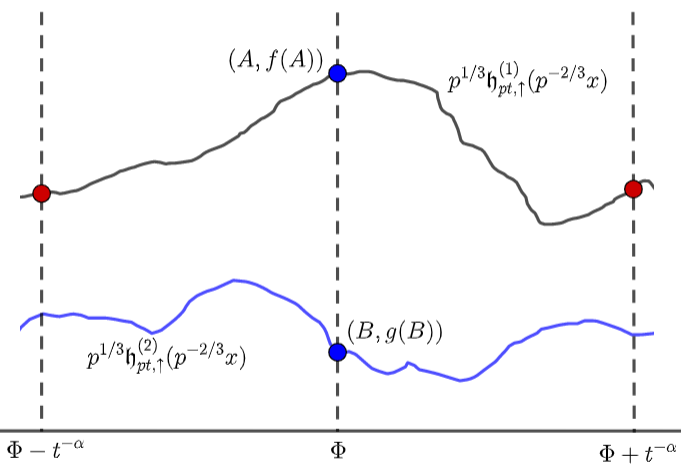}
		\caption{In the above figure we have plotted the curves $f(x):=p^{1/3}\h_{pt,\uparrow}^{(1)}(p^{-2/3}x)$ (black) and $g(x):=p^{1/3}\h_{pt,\uparrow}^{(2)}(p^{-2/3}x)$ (blue) restricted to the interval $K_t:=(\Phi-t^{-\alpha},\Phi+t^{-\alpha})$. For convenience, we have marked two blue points along with their values as $(A,f(A))$, $(B,g(B))$. $\m{Gap}_{M,\uparrow}(\delta)$ defined in \eqref{gapup} denote the event that the blue points are separated by $\delta$, i.e, $f(A)-g(B)\ge \delta$. The $\m{Rise}_{M,\uparrow}(\delta)$ defined in \eqref{riseup} ensures \textit{no} point on the blue curve (restricted to $K_t$) has value larger than  $g(B)+\frac14\delta$ (that is no significant rise). The $\m{Bd}_{\uparrow}(\delta)$ event defined in \eqref{bdup} indicates the red points on the black curve are within $[f(A)-\frac1\delta t^{-\alpha/2},f(A)+\frac1\delta t^{-\alpha/2}]$. The $\m{Sink}_{\uparrow}(\delta)$ event defined in \eqref{sinkdef} ensures that \textit{all} points on the black curve (restricted to $K_t$) have values larger than $f(A)-\frac14\delta$ (that is no significant sink). Clearly then on $\m{Sink}_{\uparrow}(\delta)\cap \m{Rise}_{M,\uparrow}(\delta)\cap  \m{Gap}_{M,\uparrow}(\delta)$ for all $x\in K_t$, we have $f(x)-g(x)\ge f(A)-\frac14\delta-g(B)-\frac14\delta \ge \frac12\delta$.}
		\label{fig:gap}
	\end{figure}
	Recall $\m{Rise}_{M,\uparrow}(\delta)$ and $\m{Gap}_{M,\uparrow}(\delta)$ from \eqref{riseup} and \eqref{gapup}. Note that on $\m{Sink}_{\uparrow}(\delta)\cap \m{Rise}_{M,\uparrow}(\delta) \cap \m{Gap}_{M,\uparrow}(\delta)$ we have uniform separation between $\h_{pt,\uparrow}^{(1)}$ and $\h_{pt,\downarrow}^{(2)}$ on the interval $p^{-2/3}{K}_t$, that is
	\begin{align}
		\label{usep}
		\inf_{x\in [\Phi-t^{-\alpha},\Phi+t^{-\alpha}]} \left[p^{1/3}\h_{pt,\uparrow}^{(1)}(p^{-2/3}x)-p^{1/3}\h_{pt,\uparrow}^{(2)}(p^{-2/3} x)\right] \ge \tfrac{\delta}{2}.
	\end{align}
	See Figure \ref{fig:gap} alongside its caption for further explanation of the above fact.
	But then \eqref{usep} forces $W_{\uparrow,1} \ge \exp(-t^{2/3}2t^{-\alpha}e^{-\frac14t^{1/3}\delta})$ which can be made strictly larger than $\sqrt{1-\delta}$ for all large enough $t$. Thus,
	\begin{align}
		\label{intm}
		\ind_{\Nice_{M}(\delta)}\Pr_{\m{Nsmall|3,2,1}}(W_{\uparrow,1}>\sqrt{1-\delta}) \ge \ind_{\Nice_{M}(\delta)}\Pr_{\m{Nsmall|3,2,1}}(\m{Sink}_{\uparrow}(\delta)).
	\end{align}
	Now we divide the sink event into two parts: $\m{Sink}_{\uparrow}(\delta)=\m{Sink}_{+\uparrow}(\delta)\cap\m{Sink}_{-,\uparrow}(\delta)$ where
	\begin{align*}
		\m{Sink}_{\pm,\uparrow}(\delta) & := \left\{ \inf_{x\in [0,t^{-\alpha}]} p^{1/3}\h_{pt,\uparrow}^{(1)}(\Phi p^{-2/3}\pm x) \ge p^{1/3}\h_{pt,\uparrow}^{(1)}(\Phi p^{-2/3})-\tfrac{\delta}{4} \right\}, 
	\end{align*}
	In view of \eqref{intm}, to prove first part of \eqref{l2wts}, it suffices to show for all large enough $t$, on $\m{Nice}_M(\delta)$ we have
	\begin{align}
		\label{l3wts}
		\Pr_{\m{Nsmall|3,2,1}}(\m{Sink}_{+,\uparrow}(\delta)) \ge 1-\tfrac{\delta}{4}, \quad \Pr_{\m{Nsmall|3,2,1}}(\m{Sink}_{-,\uparrow}(\delta)) \ge 1-\tfrac{\delta}{4}.
	\end{align}
	We only prove first part of \eqref{l3wts} below. Towards this end, recall $Y_{M,t,\uparrow}^{(1)}, Y_{M,t,\downarrow}^{(1)}$ from \eqref{eq:yupdown}. Observe that 
	\begin{align*}
		Y_{M,t,\uparrow}^{(1)}(\Phi+x)=p^{1/3}\h_{pt,\uparrow}^{(1)}(\Phi p^{-2/3})-p^{1/3}\h_{pt,\uparrow}^{(1)}(\Phi p^{-2/3}+x).
	\end{align*}
	Recall $\m{Nsmall}$ law from Definition \ref{def:nsmall}. Our discussion in Subsection \ref{sb:small} implies that under $\Pr_{\m{Nsmall}|3.2,1}$, $$(Y_{M,t,\uparrow}^{(1)}, Y_{M,t,\downarrow}^{(1)})(\Phi+\cdot)|_{[0,t^{-\alpha}]}\stackrel{d}{=} V_{r}^{\m{small}}(\cdot), \quad (Y_{M,t,\uparrow}^{(1)}, Y_{M,t,\downarrow}^{(1)})(\Phi+\cdot)|_{[-t^{-\alpha},0]}\stackrel{d}{=}V_{\ell}^{\m{small}}(-\cdot),$$
	where recall that $\V_\ell^{\m{small}}$ and $\V_{r}^{\m{small}}$ are conditionally independent $\nonintbb$ on $[0,t^{-\alpha}]$ with appropriate end points, defined in Definition \ref{def:nsmall}. In particular we have,
	\begin{align}
		\label{sinktov}
		\Pr_{\m{Nsmall|3,2,1}}(\m{Sink}_{+,\uparrow}(\delta)) =\Pr_{\m{Nsmall|3,2,1}}\left(\sup_{x\in [0,t^{-\alpha}]} V_{r,1}^{\m{small}}(x)\le \tfrac14\delta\right)
	\end{align}
	where $V_{r}^{\m{small}}=(V_{r,1}^{\m{small}},V_{r,2}^{\m{small}})$. Recall $\m{Nice}_M(\delta)$ event from \eqref{def:nice}. It contains $\m{Bd}_{\uparrow}(\delta)$ event defined in \eqref{bdup}. On this event, $-\frac1\delta \le V_{r,1}^{\m{Small}}(t^{-\alpha}),V_{r,2}^{\m{Small}}(t^{-\alpha}) \le \frac1\delta t^{-\alpha/2}$. We consider another $\nonintbb$ {$U = (U_1, U_2)$} on $[0,t^{-\alpha}]$ with non-random endpoints $U_1(t^{-\alpha})=U_2(t^{-\alpha})=\frac1\delta t^{-\alpha/2}$. On $\m{Bd}_{\uparrow}(\delta)$ event, by monotonicity of non-intersecting Brownian bridges (Lemma 2.6 in \cite{CH14}), one may couple $U=(U_1,U_2)$ and $V_{r}^{\m{small}}$ so that $U_i$ always lies above $V_{r,i}^{\m{small}}$ for $i=1,2$. Thus on $\m{Bd}_{\uparrow}(\delta)$ event,
	\begin{align*}
		\Pr_{\m{Nsmall|3,2,1}}\left(\sup_{x\in [0,t^{-\alpha}]} V_{r,1}^{\m{small}}(x)\le \lambda t^{-\alpha/2} \right) \ge 
		\Pr\left(\sup_{x\in [0,1]} t^{\alpha/2}U_1(xt^{-\alpha}) \le \lambda\right) \ge 1-\tfrac{\delta}4,
	\end{align*}
	where the last inequality is true by taking $\lambda$ large enough. This choice of $\lambda$ is possible as by Brownian scaling, $t^{\alpha/2}U_1(xt^{-\alpha}),t^{\alpha/2}U_2(xt^{-\alpha})$ is $\nonintbb$ on $[0,1]$ ending at $(\frac1\delta,\frac1\delta)$. Taking $t$ large enough one can ensure $\lambda t^{-\alpha/2} \le \frac{\delta}{4}$. Using the equality in \eqref{sinktov} we thus establish the first part of \eqref{l3wts}. The second part is analogous. This proves the first part of \eqref{l2wts}. The second part of \eqref{l2wts} follows similarly. This completes the proof of Lemma \ref{l:rn1}.
	
	\medskip
	
	\noindent\textbf{Proof of Lemma \ref{l:closetod}.} {The idea behind this proof is Proposition \ref{lemmaC}, which states that a $\nonintbb$ after Brownian rescaling converges in distribution to a $\dbm$. The following fills out the details.} 
	Recall that $$\Pr_{\m{Nsmall}|3.2,1}(A)=\Pr_{\m{Nsmall}|3.2,1}(D_{M,t,\uparrow},D_{M,t,\downarrow}(\cdot)\in A).$$
	Recall from \eqref{eq:dtoyM} that $D_{M,t,\uparrow},D_{M,t,\downarrow}$ is a diffusive scaling of $Y_{M,t,\uparrow}^{(1)}, Y_{M,t,\downarrow}^{(1)}$ when centering at $\Phi$, where $Y_{M,t,\uparrow}^{(1)}, Y_{M,t,\downarrow}^{(1)}$ are defined in \eqref{eq:yupdown}. Recall $\m{Nsmall}$ law from Definition \ref{def:nsmall}. Our discussion in Subsection \ref{sb:small} implies that under $\Pr_{\m{Nsmall}|3.2,1}$, $$(Y_{M,t,\uparrow}^{(1)}, Y_{M,t,\downarrow}^{(1)})(\Phi+\cdot)|_{[0,t^{-\alpha}]}\stackrel{d}{=} V_{r}^{\m{small}}(\cdot), \quad (Y_{M,t,\uparrow}^{(1)}, Y_{M,t,\downarrow}^{(1)})(\Phi+\cdot)|_{[-t^{-\alpha},0]}\stackrel{d}{=}V_{\ell}^{\m{small}}(-\cdot),$$
	where $\V_\ell^{\m{small}}$ and $\V_{r}^{\m{small}}$ are conditionally independent $\nonintbb$ on $[0,t^{-\alpha}]$ with appropriate end points defined in Definition \ref{def:nsmall}. Using Brownian scaling, we consider $$\V_{\ell}^0(x):=t^{\alpha/2}\V_{\ell}^{\m{small}}(x t^{-\alpha}), \quad \V_{r}^0(x):=t^{\alpha/2}\V_{r}^{\m{small}}(x t^{-\alpha}),$$
	which are now $\nonintbb$ on $[0,1]$. Note that on $\m{Bd}_{\uparrow}(\delta),\m{Bd}_{\downarrow}(\delta)$ (defined in \eqref{bdup}), we see that endpoints of $V_{\ell}^0, V_{r}^0$ are in $[-\frac1\delta,\frac1\delta]$. Thus as $\alpha=\frac16$, performing another diffusive scaling by Proposition \ref{lemmaC} we see that as $t\to \infty$
	$$t^{1/4}\V_{\ell}^0(x t^{-1/2})\ , \ t^{1/4}\V_{r}(xt^{-1/2})$$ converges to two independent copies of $\dbm$s (defined in Definition \ref{def:dbm}) in the uniform-on-compact topology. Hence we get two-sided $\dbm$ convergence for the pair $(D_{M,t,\uparrow},D_{M,t,\downarrow})$ under $\Pr_{\m{Nsmall}|3.2,1}$ as long as $\ind\{\Nice_M(\delta)\}$ holds. This proves \eqref{eq:l3}.
	
	\subsection{Proof of Theorem \ref{t:bessel}} \label{sec:last}
	
	We take $p\mapsto \frac12$ and $t\mapsto 2t$ in Proposition \ref{p:dyson}. Then by Lemma \ref{l:conec}, $\mx_{2,t}$ defined in the statement of Theorem \ref{t:bessel} is same as $\md_{\frac12,2t}$ considered in Proposition \ref{p:dyson}. Its uniqueness is already justified in Lemma \ref{lem1}. Furthermore,
	$$R_2(x,t)\stackrel{d}{=} D_1(x,t)-D_2(x,t),$$
	as functions in $x$, where $R_2(x,t)$ is defined in \eqref{eq:bessel} and $D_1,D_2$ are defined in \eqref{eq:dyson}. 
	By Proposition \ref{p:dyson} and Lemma \ref{l:dtob} we get that $D_1(x,t)-D_2(x,t)\stackrel{d}{\to} \mathcal{R}_2(x)$ in the uniform-on-compact topology. This proves Theorem \ref{t:bessel} for $k=2$ case. \\
	
	For $k=1$ case, by Lemma \ref{l:conec}, $\mx_{1,t}$ is same as $\md_{*,t}$ which is unique almost surely by Lemma \ref{lem1}. This guarantees $\mx_{1,t}$ is unique almost surely as well.  Thus we are left to show
	\begin{align}
		\label{eq:bconv}
		\calH(\mx_{1,t},t)-\calH(x+\mx_{1,t},t) \stackrel{d}{\to} \mathcal{R}_1(x).
	\end{align}
	where $\mathcal{R}_1(x)$ is a two-sided Bessel process with diffusion coefficient $1$ defined in Definition \ref{def:bessel}. The proof of \eqref{eq:bconv} is exactly similar to that of Proposition \ref{p:dyson} with few minor alterations listed below.
	
	
	\begin{enumerate}
		\item Just as in Subsection \ref{sb:frame}, one may put the problem in \eqref{eq:bconv} under the framework of KPZ line ensemble. Compared to Subsection \ref{sb:frame}, in this case, clearly there will be just one set of line ensemble.
		
		\item Given the decay estimates for $\md_{*,t}$ from Lemma \ref{lem1}, it boils down to show Bessel behavior around local maximizers. The rigorous justification follows from a soft argument analogous to what is done in Subsection \ref{sb:gtol}.
		
		\item In the spirit of Subsection \ref{sb:nice}, one can define a similar $\Nice'_M(\delta)$ event but now for a single line ensemble. $\Nice'_M(\delta)$ will contain similar events, such as: 
		\begin{itemize}
			\item control on the location of local maximizer (analog of $\m{ArMx}(\delta)$ event \eqref{eq:armx}),
			\item control on the gap between first curve and second curve at the maximizer (analog of $\m{Gap}_{M,\uparrow}(\delta)$ event \eqref{gapup}),
			\item fluctuations of the first curve on a small interval say $I$ around maximizer (analog of $\m{Rise}_{M,\uparrow}(\delta)$ event \eqref{riseup},
			\item and control on the value of the endpoints of $I$ (analog of $\m{Bd}_{\uparrow}(\delta)$ event \eqref{bdup}).
		\end{itemize}   On $\Nice'_M(\delta)$ event, the conditional analysis can be performed in the same manner. 
		
		\item Next, as in proof of Proposition \ref{p:dyson}, we proceed by three layers of conditioning. {For first layer, we use the $\mathbf{H}_t$ Brownian Gibbs property of the single line ensemble under consideration. Next, conditioning on the location and values of the maximizer, we similarly apply the same Bessel bridge decomposition result from Proposition \ref{p:bbdecomp} to convert the conditional law to that of the Bessel bridges over a large interval (see Subsection \ref{sb:large}). Finally, analogous to Subsection \ref{sb:small}, the third layer of conditioning reduces large Bessel bridges to smaller ones following the Markovian property of Bessel bridges, see Lemma \ref{l:tdbes}. }
		
		\item Since a Bessel bridge say on $[0,1]$ is a Brownian bridge conditioned to stay positive on $[0,1]$, it has the Brownian scaling property and it admits monotonicity w.r.t.~endpoints. These are two crucial tools that went into the Proof of Lemma \ref{l:rn1} in Subsection \ref{sb:2lem}. Thus the Bessel analogue of Lemma \ref{l:rn1} can be derived using {the scaling property and monotonicity stated above} in the exact same way.  Finally, the Bessel analogue of Lemma \ref{l:closetod} can be obtained from Corollary \ref{lemmaC1}. Indeed Corollary \ref{lemmaC1} ensures that small Bessel bridges converges to Bessel process under appropriate diffusive limits on the $\Nice'_{M}(\delta)$ event. 
	\end{enumerate}
	
	Executing all the above steps in an exact same manner as proof of Proposition \ref{p:dyson}, \eqref{eq:bconv} is established. This completes the proof of Theorem \ref{t:bessel}. 
	
	\section{Proof of localization theorems} \label{sec:pfmain}

	In this section we prove our main results: Theorem \ref{t:main} and Theorem \ref{t:main2}. In Section \ref{sec:7.1} we study certain tail properties (Lemma \ref{p:besselwd} and Proposition \ref{p:ctail}) of the quantities that we are interested in and prove  Theorem \ref{t:main}. Proof of Proposition \ref{p:ctail} is then completed in Section \ref{sec:7.2} along with proof of Theorem \ref{t:main2}.
	
	\subsection{Tail Properties and proof of Theorem \ref{t:main}} \label{sec:7.1}
	
	We first settle the question of finiteness of the Bessel integral appearing in the statements of Theorems \ref{t:main} and \ref{t:main2} in the following Lemma.
	
	\begin{lemma}\label{p:besselwd} Let $R_{\sigma}(\cdot)$ be a Bessel process with diffusion coefficient $\sigma>0$, defined in Definition \ref{def:bessel}. Then
		\begin{align*}
			\Pr\left(\int_{\R} e^{-R_{\sigma}(x)}\d x \ { \in (0,\infty)} \right)=1.
		\end{align*}
	\end{lemma}
	
	\begin{proof} Since $R_{\sigma}(\cdot)$ has continuous paths, $\sup_{x\in [0,1]} R_{\sigma}(x)$ is finite almost surely. Thus almost surely we have
		$$\int_{\R} e^{-R_{\sigma}(x)}\d x \ge \int_{0}^1 e^{-R_{\sigma}(x)}\d x > 0.$$
		On the other hand, by the classical result from \cite{motoo1959proof} it is known that
		$$\Pr(R_{\sigma}(x) < x^{1/4} \mbox{ infinitely often})=0.$$ Thus, there exists $\Omega$ such that $\Pr(\Omega)=1$ and for all $\omega\in \Omega$, there exists $x_0(\omega)\in (0,\infty)$ such that
		\begin{align*}
			R_{\sigma}(x)(\omega) \ge x^{1/4} \mbox{ for all } x\ge x_0(\omega).
		\end{align*}
		Hence for this $\omega$,
		\begin{align*}
			{\int_0^{\infty} e^{-R_{\sigma}(x)(\omega)}\d x = \int_0^{x_0(\omega)} e^{-R_{\sigma}(x)(\omega)}\d x +\int_{x_0(\omega)}^{\infty} e^{-R_{\sigma}(x)(\omega)}\d x < x_0(\omega)+\int_{0}^{\infty} e^{-x^{1/4}}\d x <\infty.}
		\end{align*}
		This establishes that $\int_{\R} e^{-R_{\sigma}(x)}\d x$ is finite almost surely.
	\end{proof}
	
	Our next result studies the tail of the integral of the pre-limiting process.
	
	\begin{proposition}\label{p:ctail} Fix $p\in (0,1)$. Set $q=1-p$. Consider $2$ independent copies of the KPZ equation $\calH_\uparrow(x,t)$, and $\calH_\downarrow(x,t)$, both started from the narrow wedge initial data. Let $\md_{p,t}$ be the almost sure unique maximizer of the process $x\mapsto (\calH_\uparrow(x,pt)+\calH_\downarrow(x,qt))$ which exists via Lemma \ref{lem1}. Set
		\begin{equation}
			\label{def:0d1d2}
			\begin{aligned}
				D_1(x,t) & :=\calH_\uparrow(\md_{p,t},pt)-\calH_\uparrow(x+\md_{p,t},pt), \\
				D_2(x,t) & :=\calH_\downarrow(x+\md_{p,t},qt)-\calH_\downarrow(\md_{p,t},qt).
			\end{aligned}
		\end{equation} For all $\rho>0$ we have
		\begin{align}\label{eq:dyson2}
			\limsup_{K\to\infty}\limsup_{t\to\infty}\Pr\left(\int_{[-K,K]^c}e^{D_2(x,t)-D_1(x,t)}\,\d x \ge \rho\right)=0.
		\end{align}
	\end{proposition}

	As a corollary, we derive that for any $p\in (0,1)$ the $pt$-point density of point-to-point $\cdrp$ of length $t$ indeed concentrates in a microscopic region of size $O(1)$ around the favorite point.

	\begin{corollary}\label{cor:tight} Recall the definition of $\cdrp$ and the notation $\Pr^{\xi}$ from Definition \ref{def:cdrp}. Fix $p\in (0,1)$. Suppose $X \sim \cdrp(0,0;0,t)$. Consider $\md_{p,t}$ the almost sure unique mode of $f_{p,t}$, the quenched density of $X(pt)$. We have
		\begin{align*}
			\limsup_{K\to\infty}\limsup_{t\to \infty}\Pr^{\xi}\left(|X(pt)-\md_{p,t}|\ge K\right)=0, \mbox{ in probability}.
		\end{align*}
	\end{corollary}
	
	One also has the analogous version of Proposition \ref{p:ctail} involving one single copy of the KPZ equation viewed around its maximum. This leads to a similar corollary about tightness of the quenched endpoint distribution for point-to-line $\cdrp$ (see Definition \ref{def:cdrp2}) when re-centered around its mode. The details are skipped for brevity.\\
	
	The proof of Proposition \ref{p:ctail} is heavily technical and relies on the tools as well as notations from Proposition \ref{p:dyson}. For clarity, we first prove Corollary \ref{cor:tight} and Theorem \ref{t:main} assuming the validity of Proposition \ref{p:ctail}. The proof of Proposition \ref{p:ctail} is then presented in Section \ref{sec:7.2}.
	
	\begin{proof}[Proof of Corollary \ref{cor:tight}] We have $\calZ(0,0;x,pt)\stackrel{d}{=}e^{\calH_{\uparrow}(x,pt)}$ and by time reversal property $\calZ(x,pt;0,t)\stackrel{d}{=} e^{\calH_{\downarrow}(x,qt)}$ as functions in $x$, where $\calH_{\uparrow}, \calH_{\downarrow}$ are independent copies of KPZ equation started from narrow wedge initial data. The uniqueness of the mode $\md_{p,t}$ for $f_{p,t}$ is already settled in Lemma \ref{lem1}. Thus, the quenched density of $X(pt)-\md_{p,t}$ is given by
		\begin{align}\label{dfnf}
			f_{p,t} (x+\md_{p,t})= \frac{\exp(D_2(x,t)-D_1(x,t))}{\int\limits_{\R} \exp(D_2(y,t)-D_1(y,t))\d y},
		\end{align}
		where $D_i(x, t), i =1, 2$ are defined in \eqref{def:d1d2}. Thus,
		\begin{align}
			\label{final}
			\Pr^{\xi}\left(|X(pt)-\md_{p,t}|\ge K\right)=\frac{\int\limits_{[-K,K]^c} e^{D_2(x,t)-D_1(x,t)}\,\d x}{\int\limits_{\R} e^{D_2(x,t)-D_1(x,t)}\,\d x} \le \frac{\int\limits_{[-K,K]^c} e^{D_2(x,t)-D_1(x,t)}\,\d x}{\int\limits_{[-K,K]} e^{D_2(x,t)-D_1(x,t)}\,\d x}.
		\end{align}
		Notice that by by \eqref{eq:dyson2} the numerator of r.h.s.~of \eqref{final} goes to zero in probability under the iterated limit $\limsup_{t\to\infty}$ followed by $\limsup_{K\to\infty}$. Whereas due to Proposition \ref{p:dyson}, under the iterated limit, the denominator converges in distribution to $\int_{\R} e^{-R_2(x)}\d x$ which is strictly positive by Lemma \ref{p:besselwd}. Thus overall the r.h.s.~of \eqref{final} goes to zero in probability under the iterated limit. This completes the proof.
	\end{proof}


	\begin{proof}[Proof of Theorem \ref{t:main}] Fix any $p \in (0,1).$ Set $q=1-p$.  Recall from \eqref{dfnf} that 
		\begin{align}\label{dfnf2}
			f_{p,t} (x+\md_{p,t})= \frac{\exp(D_2(x,t)-D_1(x,t))}{\int\limits_{\R} \exp(D_2(y,t)-D_1(y,t))\d y}
		\end{align}
		where $D_i(x, t), i =1, 2$ are defined in \eqref{def:d1d2}. Note that by Proposition \ref{p:dyson}, 
		a continuous mapping theorem immediately implies that for any $K< \infty$
		\begin{align}\label{cc1}
			\frac{\exp(D_2(x,t)- D_1(x,t) )}{\int_{-K}^K\exp(D_2(y,t)- D_1(y,t))\d y }\stackrel{d}{\rightarrow} \frac{e^{-\mathcal{R}_2(x)}}{\int_{-K}^K e^{-\mathcal{R}_2(y)}\d y}
		\end{align} in the uniform-on-compact topology. Here $\mathcal{R}_2$ is a 3D Bessel process with diffusion coefficient $2$. For simplicity, we denote 
		\begin{align*}
			\g_{t}(x): = \exp(D_2(x,t) - D_1(x,t)) \text{ and } \g(x) = \exp(-\mathcal{R}_2(x)).
		\end{align*}
		We can then rewrite \eqref{dfnf2} as product of four factors:
		\begin{align*}
			f_{p,t}(x + \calM_{p,t})= \frac{\g_{t}(x)}{\int_{\R}\g_{t}(y)\d y}=\frac{\int_{-K}^K \g_{t}(y) \d y}{\int_{\R} \g_{t}(y)\d y}\cdot \frac{\int_{\R}\g(y) \d y}{\int_{-K}^K \g(y)\d y}\cdot \frac{\int_{-K}^K \g(y) \d y}{\int_{\R} \g(y)\d y}\cdot\frac{ \g_{t}(x)}{\int_{-K}^K \g_{t}(y)\d y}.
		\end{align*}
		Corollary \ref{cor:tight} ensures 
		$$\frac{\int_{-K}^K \g_t(y) \d y}{\int_{\R} \g_t(y) \d y}= \Pr^{\xi}(|X(pt)-\md_{p,t}|\le K)\stackrel{p}{\to} 1$$ as $t\to \infty$ followed by $K\to \infty$. Lemma \ref{p:besselwd} with $\sigma = 2$ yields that $\int_{[-K,K]^c}\g(y)\d y= \int_{[-K,K]^c}e^{-\calR_2(y)}\d y \stackrel{p}{\rightarrow} 0$ as $K \rightarrow \infty.$ Thus as $K \rightarrow \infty$
		\begin{align*}
			\frac{\int_{\R}\g(y) \d y}{\int_{-K}^K \g(y)\d y} \stackrel{p}{\rightarrow}1.
		\end{align*} Meanwhile, \eqref{cc1} yields that as $t \rightarrow \infty,$ 
		\begin{align*}
			\frac{\int_{-K}^K \g(y)\d y}{\int_{\R}\g(y)\d y}\cdot\frac{ \g_t(x)}{\int_{-K}^K \g_t(y)\d y} \stackrel{d}{\rightarrow} \frac{\int_{-K}^K \g(y)\d y}{\int_{\R}\g(y)\d y}\cdot\frac{ \g(x)}{\int_{-K}^K \g(y)\d y} = \frac{\g(x)}{\int_{\R} \g(y)\d y}.
		\end{align*}
		in the uniform-on-compact topology. Thus, overall we get that $f_{p,t}(x+\md_{p,t}) \stackrel{d}{\to} \frac{\g(x)}{\int_{\R} \g(y)\d y},$ in the uniform-on-compact topology. This establishes \eqref{e:main}, completing the proof of Theorem \ref{t:main}.
	\end{proof}

	\subsection{Proof of Proposition \ref{p:ctail} \label{sec:7.2} and Theorem \ref{t:main2}} Coming to the proof of Proposition \ref{p:ctail}, we note that the setup of Proposition \ref{p:ctail} is same as that of Proposition \ref{p:dyson}. Hence all the discussions pertaining to Proposition \ref{p:dyson} are applicable here. In particular, to prove Proposition \ref{p:ctail}, we will be using few notations and certain results from the proof of Proposition \ref{p:dyson}.

	\begin{proof}[Proof of Proposition \ref{p:ctail}] Fix any $M>0$. The proof of \eqref{eq:dyson} proceeds by dividing the integral into two parts depending on the range: 
		\begin{align}\tag{Deep Tail} \label{dtail}
			U_1 & :=[-t^{2/3}M-\md_{p,t},t^{2/3}M-\md_{p,t}]^c, \\  \tag{Shallow Tail} \label{stail} U_2 & :=[K,K]^c \cap [-t^{2/3}M-\md_{p,t},t^{2/3}M-\md_{p,t}],
		\end{align}
		and controlling each of them individually. See Figure \ref{fig:tail} for details. In the following two steps, we control these two kind of tails respectively.

		\begin{figure}[h!]
			\centering
			\includegraphics[width=14cm]{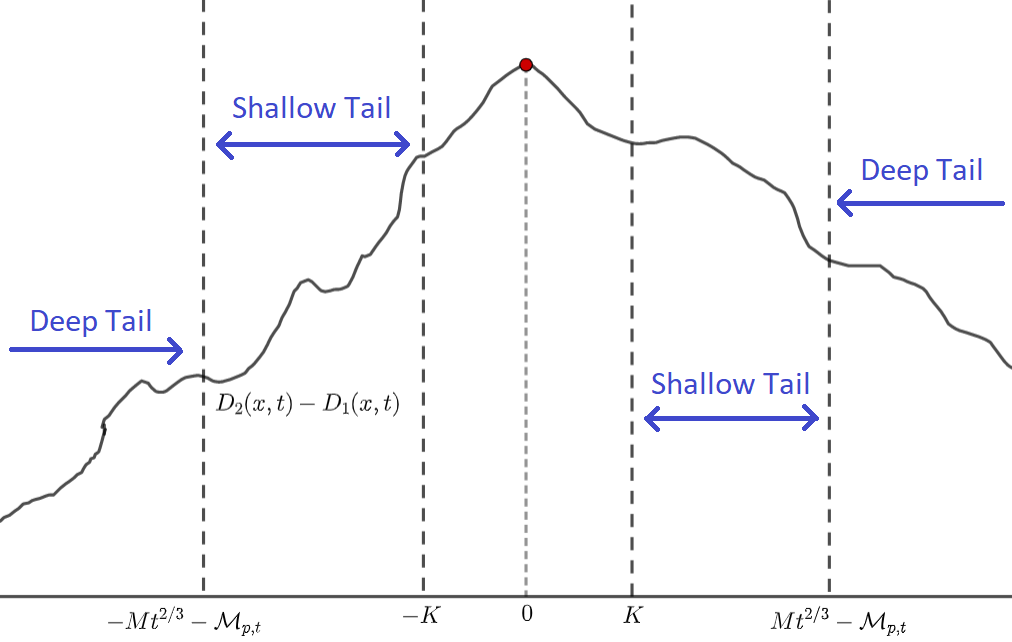}
			\caption{Illustration for the proof of Proposition \ref{p:ctail}. In \ref{dtail} region we use parabolic decay of KPZ line ensemble, and in \ref{stail} we use non-intersecting Brownian bridge separation estimates from Proposition \ref{pgamma}.}
			\label{fig:tail}
		\end{figure}
		
		\medskip

		\noindent\textbf{Step 1.} In this step, we control the \ref{dtail} region: $U_1$. The goal of this step is to show
		\begin{align}
			\label{7wts}
			\limsup_{t\to \infty} \Pr\left(\int_{U_1}e^{D_2(x,t)-D_1(x,t)}\,\d x \ge \tfrac{\rho}{2}\right) \le \Con \exp(-\tfrac1\Con M^3),
		\end{align}
		for some constant $\Con=\Con(p)>0$. We now recall the framework of KPZ line ensemble discussed in Subsection \ref{sb:frame}. We define
		\begin{align}\label{cc4}
			\mathcal{S}_{p, t}(x) : = p^{1/3}\h^{(1)}_{pt, \uparrow}(p^{-2/3}x)+ q^{1/3}\h^{(1)}_{qt, \downarrow}(q^{-2/3}x)
		\end{align} 
		where $\h_{t,\uparrow},\h_{t,\downarrow}$ are scaled KPZ line ensembles corresponding to $\calH_{\uparrow}, \calH_{\downarrow}$, see \eqref{eq:htx}.	Observe that 
		\begin{align*}
			D_2(x,t)-D_1(x,t) \stackrel{d}{=} t^{1/3}\left[ \mathcal{S}_{p,t}(t^{-2/3}(x+\md_{p,t}))-\sup_{z\in \R} \mathcal{S}_{p,t}(z)\right],
		\end{align*}	
		where $D_1,D_2$ are defined in \eqref{def:0d1d2}.
		Thus we have
		\begin{align*}
			\int_{U_1}\exp(D_2(x,t)-D_1(x,t))\d x \stackrel{d}{=}  \int_{|x|\ge M}\exp\left(t^{1/3}\left[ \mathcal{S}_{p,t}(x)-\sup_{z\in \R} \mathcal{S}_{p,t}(z)\right]\right)\d x
		\end{align*}
		where $U_1$ is defined in \eqref{dtail}. Towards this end, we define two events
		\begin{align*}
			\m{A}:=\left\{\sup_{z\in \R}\mathcal{S}_{p,t}(z) \le  -\tfrac{M^2}{4}\right\}, \quad
			\m{B}:=\left\{\sup_{x\in \R}\left(\mathcal{S}_{p,t}(x)+x^2 \right)>\tfrac{M^2}{4}\right\},
		\end{align*}
		Note that on $\neg A \cap \neg B$, for all $|x|\ge M$, we have
		\begin{align*}
			\mathcal{S}_{p,t}(x)-\sup_{z\in \R}\mathcal{S}_{p,t}(z) \le \tfrac{M^2}{4}+\tfrac{M^2}{4}-x^2 \le \tfrac{M^2}{2}-\tfrac{3M^2}{4}-\tfrac{x^2}{4} \le -\tfrac{M^2}{4}-\tfrac{x^2}{4}.
		\end{align*}
		This forces
		\begin{align*}
			\int_{|x|\ge M}\exp\left(t^{1/3}\left[ \mathcal{S}_{p,t}(x)-\sup_{z\in \R} \mathcal{S}_{p,t}(z)\right]\right)\d x \le \int_{[-M, M]^c} \exp\left(-t^{1/3}(\tfrac{M^2}{2}+\tfrac{y^2}{4})\right)\d y,
		\end{align*}
		which goes to zero as $t\to \infty.$ Hence $\mbox{l.h.s.~of \eqref{7wts}} \le \Pr(\neg\m{A})+\Pr(\neg\m{B})$. Hence it suffices to show
		\begin{align}
			\label{aevent}
			\Pr(\neg\m{A}) \le \Con\exp\left(-\tfrac1\Con M^3\right), \quad \Pr(\neg\m{B}) \le \Con\exp\left(-\tfrac1\Con M^3\right).
		\end{align}
		
		To prove the first part of \eqref{aevent}, note that
		\begin{align*}
			\Pr\left(\neg A\right) & \le \Pr\left(\mathcal{S}_{p,t}(0) \le  -\tfrac{M^2}4\right) \\ & \le \Pr\left(p^{1/3}\h_{pt,\uparrow}^{(1)}(0) \le -\tfrac{M^2}{8}\right)+\Pr\left(q^{1/3}\h_{qt,\downarrow}^{(1)}(0) \le -\tfrac{M^2}{8}\right) \le  \Con\exp(-\tfrac1\Con {M^3}).
		\end{align*}
		where the last inequality follows by Proposition \ref{p:kpzeq} \ref{p:tail}, for some constant $\Con=\Con(p)>0$. This proves the first part of \eqref{aevent}. For the second part of \eqref{aevent},	following the definition of $\mathcal{S}_{p,t}(x)$ from \eqref{cc4}, and using the elementary inequality $\frac1{4p}+\frac1{4q}\ge 1$ by a union bound we have
		\begin{equation}
			\label{cc7}
			\begin{aligned}
				\Pr\left(\sup_{x\in \R}\left(\mathcal{S}_{p,t}(x)+x^2 \right)>\tfrac{M^2}{4}\right) & \le \Pr\left(\sup_{x\in \R} \left(p^{1/3}\h_{pt,\uparrow}^{(1)}(p^{-2/3}x)+\tfrac{x^2}{4p}\right)>\tfrac{M^2}{8}\right) \\ & \hspace{1cm}+\Pr\left(\sup_{x\in \R} \left(q^{1/3}\h_{qt,\uparrow}^{(1)}(q^{-2/3}x)+\tfrac{x^2}{4q}\right)>\tfrac{M^2}{8}\right). 
			\end{aligned}
		\end{equation}
		Applying Proposition \eqref{p:kpzeq} \ref{p:supproc} with $\beta=\frac12$, we get that  each of the terms on r.h.s.~of \eqref{cc7} are at most $\Con\exp(-\frac1\Con M^3)$ where $\Con=\Con(p)>0$. This establishes the second part of \eqref{aevent} completing the proof of \eqref{7wts}.
		
		\medskip
		
		\noindent\textbf{Step 2.} In this step, we control the \ref{stail} region: $U_2$. We first lay out the heuristic idea behind the \ref{stail} region controls. We recall the nice event $\m{Sp}(\lambda)$ from \eqref{def:sp} which occurs with high probability. Assuming $\m{Sp}(\lambda)$ holds, we apply the  the $\mathbf{H}_t$ Brownian Gibbs property of the KPZ line ensembles, and analyze the desired integral $$\int_{U_2} e^{D_2(x,t)-D_1(x,t)}\d x$$ under the `free' Brownian bridge law. Further conditioning on the information of the maximizer converts the free law into the law of the $\nonintbb$ (defined in Definition \ref{def:nibb}). On $\m{Sp}(\lambda)$, we may apply Proposition \ref{pgamma} to obtain the desired estimates for the `free' law. One then obtain the desired estimates for KPZ law using the lower bound for the normalizing constant from Proposition \ref{line-ensemble} (b). \\
		
		We now expand upon the technical details. In what follows we will only work with the right tail:
		\begin{align*}
			U_{+,2}:=[-t^{2/3}M-\md_{p,t},t^{2/3}M-\md_{p,t}] \cap [K,\infty)=[K,t^{2/3}M-\md_{p,t}]
		\end{align*}
		and the argument for the left part of the shallow tail is analogous. Note that we also implicitly assumed $t^{2/3}M-\md_{p,t}\ge K$ above. Otherwise there is nothing to prove.  As before we utilize the the notations defined in Subsection \ref{sb:frame}. Recall the local maximizer $\md_{p,t}^M$ defined in \eqref{eq:localmax}. 
		Recall $Y_{M,t,\uparrow}^{(1)}, Y_{M,t,\downarrow}^{(1)}$ from \eqref{eq:yupdown}. Set
		\begin{align}
			\label{eqgamma}
			\Gamma_{t,M,K} & :=\int_K^{Mt^{2/3}-\calM_{p, t}} e^{-t^{1/3}\left[Y_{M, t, \uparrow}^{(1)}(t^{-2/3}(\calM_{p,t}^M + x)) - Y_{M, t, \downarrow}^{(1)}(t^{-2/3}(\calM_{p,t}^M + x))\right]}\d x \\ & \nonumber =\int_K^{Mt^{2/3}-\calM_{p, t}} \exp(-D_{M,t,\uparrow}(x)+D_{M,t,\downarrow}(x))\d x,
		\end{align}
		where the last equality follows from the definition of $D_{M,t,\uparrow},D_{M,t,\downarrow}$ from \eqref{eq:dtoyM}. Recall that the only difference between $D_1,D_2$ (defined in \eqref{eq:dtoy}) and $D_{M,t,\uparrow},D_{M,t,\downarrow}$ is that former is defined using the global maximizer $\md_{p,t}$ and the latter by local maximizer $\md_{p,t}^M$. However, Lemma \ref{lem1} implies that with probability at least $1 - \Con\exp(-\frac{1}{\Con}M^3),$ we have
		$\calM_{p, t}= \calM_{p,t}^M$. Next, fix $\lambda>0$. Consider $\m{Sp}(\lambda)$ event defined in \eqref{def:sp}. We thus have
		\begin{align}\label{c16}
			\Pr\left(\int_{U_{+,2}} e^{D_2(x,t)-D_1(x,t)}\d x \ge \frac{\rho}{4}\right) \le \Con\exp(-\tfrac{1}{\Con}M^3) + \Pr(\neg \m{Sp}(\lambda))+ \Pr\left(\Gamma_{t, M, K} \ge \tfrac{\rho}{4}, \m{Sp}(\lambda)\right).
		\end{align}
		
		We recall the $\sigma$-fields $\calF_1, \calF_2$ defined in \eqref{sf1} and \eqref{sigfields}. We first condition on $\calF_1$. As noted in Subsection \ref{sb:large}, since $\h_{pt,\uparrow}^{(1)}$ and $\h_{qt,\downarrow}^{(1)}$ are independent,  applying $\mathbf{H}_{pt}$ and $\mathbf{H}_{qt}$ Brownian Gibbs property from Proposition \ref{line-ensemble} for $\h_{pt,\uparrow}^{(1)}$, $\h_{qt,\downarrow}^{(1)}$ respectively we have
		\begin{align}\label{c18}
			\Pr\left(\Gamma_{t, M, K} \ge \tfrac{\rho}{2}, \m{Sp}(\lambda)\right) = \Ex\left[\frac{\Ex_{\operatorname{free},t}[ \ind_{\Gamma_{t, M, K} \ge \tfrac{\rho}{4}, \m{Sp}(\lambda)} W_{\uparrow}W_{\downarrow}]}{\Ex_{\operatorname{free},t}[W_{\uparrow}W_{\downarrow}]}\right],
		\end{align}
		where $W_{\uparrow}$, $W_{\downarrow}$ are defined in \eqref{eq:cond2} and \eqref{eq:cond3}. Here $\Pr_{\operatorname{free}, t}$ and $\Ex_{\operatorname{free}, t}$ are the  probability and the expectation operator respectively corresponding to the joint `free' law for $(p^{1/3}\h_{pt,\uparrow}(p^{-2/3}x)$, and  $q^{1/3}\h_{qt,\downarrow}(q^{-2/3}x))_{x\in[-M,M]}$ which by Brownian scaling is given by a pair of independent Brownian bridges $(\B_1(\cdot), \B_2(\cdot))$  on $[-M,M]$ with starting points $(p^{1/3}\h_{pt, \uparrow}(-Mp^{-2/3}), q^{1/3}\h_{qt, \downarrow}(-Mq^{-2/3}))$ and endpoints $(q^{1/3}\h_{pt, \uparrow}(Mp^{-2/3}), q^{1/3}\h_{qt, \downarrow}(Mq^{-2/3})).$

		In addition, from the last part of Proposition \ref{line-ensemble} we know that for any given $\lambda> 0$, there exists $\delta(M,p,\lambda)>0$ such that 
		\begin{align}\label{c19}
			\Pr(\Ex_{\operatorname{free},t}[W_{\uparrow}W_{\downarrow}] > \delta) \ge 1- \lambda.
		\end{align}
		Since the weight $W_{\uparrow}W_{\downarrow}\in [0,1],$ \eqref{c18} and \eqref{c19} give us 
		\begin{align}\label{c20}
			\mbox{r.h.s.~of \eqref{c16}} \le \Con\exp(-\tfrac{1}{\Con}M^3) + \Pr(\neg \m{Sp}(\lambda))+\lambda+{\frac{1}{\delta}}\Ex\left[\Pr_{\operatorname{free},t}\left(\Gamma_{t, M, K} \ge \tfrac{\rho}{4}, \m{Sp}(\lambda)\right)\right].
		\end{align}

		Next we condition on $\calF_2$ defined in \eqref{sigfields}. By Proposition \ref{propA}, upon conditioning the free measure of two Brownian bridges when viewed around the maximizer are given by two $\nonintbb$ (defined in Definition \ref{def:nibb}). The precise law is given by ${\m{Nlarge}}$ law defined in Definition \ref{def:nlarge}. Note that $\m{Sp}(\lambda)$ is measurable w.r.t.~$\calF_1\cup\calF_2$.  By Reverse Fatou's Lemma and the tower property of conditional expectations, we obtain that 
		\begin{align}
			\nonumber & \limsup_{K \rightarrow \infty}\limsup_{t\rightarrow \infty}\Ex\left[\Pr_{\operatorname{free},t}\left(\Gamma_{t, M, K} \ge \tfrac{\rho}{4}, \m{Sp}(\lambda)\right)\right] \\ \label{c22} & \le \Ex\left[\limsup_{K \rightarrow \infty}\limsup_{t\to\infty}\ind_{\m{Sp}(\lambda)}\Pr_{\m{Nlarge}|2,1}\left(\Gamma_{t, M, K} \ge \tfrac{\rho}{4}\right)\right].
		\end{align}
		Following the Definition \ref{def:nlarge} and \eqref{eqgamma} we see that under $\m{Nlarge}$ law,
		\begin{align}
			\label{largebb}
			\Gamma_{t,M,K} \stackrel{d}{=} \int_K^{Mt^{2/3}-\calM_{p, t}} e^{-t^{1/3}\left[V_{r,1}^{\m{large}}(t^{-2/3}x)-V_{r,2}^{\m{large}}(t^{-2/3}x)\right]}\d x.
		\end{align}
		where $V_{r}^{\m{large}}=(V_{r,1}^{\m{large}},V_{r,2}^{\m{large}})$ is a {$\nonintbb$} 
		defined in Definition \ref{def:nlarge}. 
		Now notice that by the definition in \eqref{def:sp}, on the $\m{Sp}(\lambda)$ event, the length of the Brownian bridges considered are bounded from below and above and the end points are tight. Following the equality in distribution in \eqref{largebb}, the technical result of Proposition \ref{pgamma} precisely tells us that the term inside the expectation of r.h.s.~of \eqref{c22} is zero. Thus, going back to \eqref{c20} we get that 
		\begin{align*}
			\limsup_{K\to\infty}\limsup_{t\to\infty}\Pr\left(\int_{U_{+,2}} e^{D_2(x,t)-D_1(x,t)}\d x \ge \frac{\rho}{4}\right) \le  \Con\exp(-\tfrac{1}{\Con}M^3) + \limsup_{t\to\infty}\Pr(\neg \m{Sp}(\lambda))+\lambda.
		\end{align*}
		Taking $\limsup_{\lambda \downarrow 0}$, in view of \eqref{eq:sp}, we get that last two terms in r.h.s.~of the above equation are zero. Similarly one can show the same bound for the integral under $U_{-,2}:=[-t^{2/3}M-\md_{p,t},t^{2/3}M-\md_{p,t}] \cap (-\infty,-K]$. Together with \eqref{7wts}, we thus have
		\begin{align*}
			\limsup_{K\to\infty}\limsup_{t\to\infty}\Pr\left(\int_{[-K,K]^c} e^{D_2(x,t)-D_1(x,t)}\d x \ge \rho \right) \le  \Con\exp(-\tfrac{1}{\Con}M^3).
		\end{align*}
		Taking $M\to \infty$ we get \eqref{eq:dyson2} completing the proof.
	\end{proof}

	\begin{proof}[Proof of Theorem \ref{t:main2}]  Recall from \eqref{eq:cdrp2} that 
		\begin{align*}
			f_{*,t} (x)= \frac{\calZ(0,0;x,t)}{\calZ(0,0;*,t)} = \frac{e^{\calH(x,t)}}{\int_{\R} e^{\calH(y,t)}\d y}.
		\end{align*}
		The uniqueness of the mode $\md_{*,t}$ for $f_{*,t}$ is already proved in Lemma \ref{lem1}. Thus, we have
		\begin{align*}
			f_{*,t} (x+\md_{*,t})= \frac{\exp\left(\calH(\md_{*,t}+x,t)-\calH(\md_{*,t},t)\right)}{\int\limits_{\R} \exp\left(\calH(\md_{*,t}+y,t)-\calH(\md_{*,t},t)\right)\d y}.
		\end{align*}
		Just like in Proposition \ref{p:ctail}, we claim that
		\begin{align}
			\label{eq:bessel2}
			\limsup_{K\to\infty}\limsup_{t\to\infty}\Pr\left(\int_{[-K,K]^c} e^{\calH(\md_{*,t}+y,t)-\calH(\md_{*,t},t)}\d y \ge \rho \right)=0.
		\end{align}
		The proof of \eqref{eq:bessel2} is exactly same as that of \eqref{eq:dyson2}, {where we divide the integral in \eqref{eq:bessel2} into a deep tail and a shallow tail and bound them individually. } To avoid repetition, we just add few pointers for the readers. Indeed the two key steps of proof of Proposition \ref{p:ctail} that bound the deep and shallow tails can be carried out for the \eqref{eq:bessel2} case. The deep tail regime follows an exact similar strategy as Step 1 of the proof of Proposition \ref{p:ctail} and utilizes the same parabolic decay of the KPZ equation from Proposition \ref{p:kpzeq}. The analogous shallow tail regime also follows in a similar manner by using the uniform separation estimate for Bessel bridges from Corollary \ref{pgamma1}. \\
		
		Now note that by Theorem \ref{t:bessel} with $k=1$, we have
		\begin{align}
			\label{convb}
			\calH(\md_{*,t}+x,t)-\calH(\md_{*,t},t) \stackrel{d}{\to} \mathcal{R}_1(x),
		\end{align}
		in the uniform-on-compact topology. Here $\mathcal{R}_1$ is a 3D-Bessel process with diffusion coefficient $1$. With the tail decay estimate in \eqref{eq:bessel2} and the same for the Bessel process from Proposition \ref{p:besselwd}, in view of \eqref{convb} one can show $f_{*,t}(x+\md_{*,t})\to \frac{e^{-\mathcal{R}_1(x)}}{\int_{\R} e^{-\mathcal{R}_1(y)}\d y}$ in the uniform-on-compact topology by following the analogous argument from the proof of Theorem \ref{t:main}.  This completes the proof.
	\end{proof}
	
	\appendix
	
	\section{Non-intersecting random walks}\label{sec:ap1}

	In this section we prove Lemma \ref{propF} that investigates the convergence of non-intersecting random walks to non-intersecting brownian motions. We remark that similar types of Theorems are already known in the literature such as \cite{eichelsbacher2008ordered}, where the authors considered random walks to start at different locations. Since our walks starts at the same point, additional care is required. \\
	
	We now recall Lemma \ref{propF} for readers' convenience.
	
	\begin{lemma}\label{1propF}
		{Let $X_j^{i}$ be i.i.d. $\operatorname{N}(0,1)$ random variables. Let $S_0^{(i)} = 0$ and $S_k^{(i)} = \sum_{j= 1}^{k}X_j^{i}.$ Consider $Y_n(t) = (Y_{n,1}(t),Y_{n,2}(t)) := (\frac{S_{nt}^{(1)}}{\sqrt{n}}, \frac{S_{nt}^{(2)}}{\sqrt{n}})$ an $\R^2$ valued process on $[0,1]$ where the in-between points are defined by linear interpolation. Then conditioned on the non-intersecting event $\Lambda_n:= \cap_{j=1}^n\{ S_j^{(1)} > S_j^{(2)}\},$
			$Y_n \stackrel{d}{\to} \W$, where $\W(t) = (\W_1(t), \W_2(t))$ is distributed as $\nonintbm$ defined in Definition \ref{def:nibm}.}
	\end{lemma}

	\begin{proof}[Proof of Lemma \ref{1propF}] To show weak convergence, it suffices to show finite dimensional convergence and tightness. Based on the availability of exact joint densities for non-intersecting random walks from Karlin-McGregor formula \cite{karlin1959coincidence}, the verification of weak convergence is straightforward. So, we only highlight major steps of the computations below. 
		
		\medskip

		\noindent\textbf{Step 1. One point convergence at $t=1$}. Note that
		\begin{align*}
			\Pr\left(|{\sqrt{n} Y_{n,i}(t)}-S_{\lfloor nt \rfloor }^{(i)}|>\sqrt{n}\e \mid \Lambda_n \right)\le \frac1{\Pr(\Lambda_n)}\Pr\left(|X_{\lfloor nt \rfloor+1}|>\sqrt{n}\e\right) \le \tfrac{\Con}{\e^2 \sqrt{n}}
		\end{align*}
		The last inequality above follows by Markov inequality and the  classical result that $\Pr(\Lambda_n) \ge \tfrac{\Con}{\sqrt{n}}$ in Spitzer \cite{spitzer1960tauberian}. Thus it suffices to show finite dimensional convergence for the cadlag process: 
		\begin{align}
			\label{t0}
			(Z_{nt}^{(1)},Z_{nt}^{(2)}):=\frac1{\sqrt{n}}(S_{\lfloor nt \rfloor }^{(1)},S_{\lfloor nt \rfloor }^{(2)}).
		\end{align}
		We assume that $n$ large enough so that $\frac{n-1}{M\sqrt{n}} \ge 1$ for some $M>0$ to be chosen later. When $t = 1$,  applying the Karlin-McGregor formula, we obtain that $$\Pr(Z_{n}(1)\in \d y_1, Z_{n}(1)\in \d y_2| \Lambda_n)=\tau_n \cdot f_{n,1}(y_1,y_2)\d y_1\d y_2$$ where
		\begin{align*}
			f_{n,1}(y_1,y_2):=\int\limits_{a_1> a_2}p_1(a_1)p_1(a_2)\det(p_{n-1}(a_i -y_j\sqrt{n}))_{i, j =1}^2\d a_1 \d a_2,
		\end{align*} 
		and
		\begin{align}
			\label{taun}
			\tau_n^{-1}:=\int_{r_1 > r_2}\int_{a_1> a_2}p_1(a_1)p_1(a_2)\det(p_{n-1}(a_i -r_j\sqrt{n}))_{i, j =1}^2\d a_1 \d a_2 \d r_1 \d r_2.
		\end{align}
		Note that here the Karlin-McGregor formula, after we have conditioned on the first step of the random walks with $X_1^1 = a_1 > X_1^2 = a_2.$
		
		We will now show that $\frac{(n-1)^2}{\sqrt{n}}\tau_n^{-1}$ and $\frac{(n-1)^2}{\sqrt{n}}f_{n,1}(y_1,y_2)$ converges to a nontrivial limit. Observe that
		\begin{equation}
			\label{heatker}
			\begin{aligned}
				\tfrac{(n-1)^2}{\sqrt{n}}\det(p_{n-1}(a_i -y_j\sqrt{n}))_{i, j =1}^2 &= (n-1)p_{n-1}(a_1 -y_2\sqrt{n})p_{n-1}(a_2 -y_1\sqrt{n})\\ & \hspace{2.5cm}\cdot \tfrac{n-1}{\sqrt{n}}[e^{\frac{\sqrt{n}(a_1-a_2)(y_1 - y_2)}{n-1}} - 1].
			\end{aligned} 
		\end{equation}
		Thus, as $n\to \infty$, we have
		\begin{align}\label{heatker2}
			\tfrac{(n-1)^2}{\sqrt{n}}\det(p_{n-1}(a_i -y_j\sqrt{n}))_{i, j =1}^2 &\to p_1(y_1)p_1(y_2)(a_1-a_2)(y_1-y_2). 
		\end{align}
		
		Next we proceed to find a uniform bound for the expression in \eqref{heatker}. Not that for $x,r\ge 1$, one has the elementary inequality $x^r\ge x^r-1\ge r(x-1)$. 
		Now taking $r=\frac{n-1}{M\sqrt{n}}$ and $x=\exp(\frac{\sqrt{n}}{n-1}(a_1-a_2)(y_1-y_2)$ we get
		\begin{align}\nonumber
			\mbox{r.h.s.~of \eqref{heatker}} & \le \frac1{2\pi}\exp\left(-\tfrac{(a_1-y_2\sqrt{n})^2}{2n-2}-\tfrac{(a_2-y_1\sqrt{n})^2}{2n-2}+\tfrac1{M}(a_1-a_2)(y_1-y_2)\right) \\ \nonumber & \le \frac1{2\pi}\exp\left(-\tfrac{y_2^2}{4}-\tfrac{y_1^2}{4}+\tfrac1{M}(a_1-a_2)(y_1-y_2)+\tfrac1M(|a_1y_2|+|a_2y_1|)\right) \\ \label{a01} & \le \frac1{2\pi}\exp\left(-\tfrac{y_2^2}{4}-\tfrac{y_1^2}{4}+ \tfrac{2(a_1^2+y_1^2+a_2^2+y_2^2)}{M})\right),
		\end{align}
		where the last inequality follows by several application of the elementary inequality $|xy| \le \frac12(x^2+y^2)$. One can choose $M$ large enough so that the uniform bound in \eqref{a01} is integrable w.r.t.~the measure $p_1(a_1)p_1(a_2)\d a_1\d a_2$. With the pointwise limit from \eqref{heatker2},  by dominated convergence theorem we have
		\begin{align*}
			\tfrac{(n-1)^2}{\sqrt{n}}f_{n,1}(y_1,y_2)	& =\tfrac{(n-1)^2}{\sqrt{n}}\int_{a_1> a_2}p_1(a_1)p_1(a_2)\det(p_{n-1}(a_i -y_j\sqrt{n}))_{i, j =1}^2\d a_1 \d a_2 \\ & \hspace{2cm}\to p_1(y_1)p_1(y_2)(y_1-y_2)\int_{a_1> a_2}(a_1-a_2)p_1(a_1)p_1(a_2)\d a_1 \d a_2.
		\end{align*}
		Similarly one can compute the pointwise limit for the integrand in $\tau_n^{-1}$ (defined in \eqref{taun}) and the uniform bound in \eqref{a01} works for the denominator as well. We thus have
		\begin{align}
			\label{taulim}
			& \tfrac{(n-1)^2}{\sqrt{n}}\tau_n^{-1} \to \int_{a_1> a_2}\int_{r_1>r_2}p_1(r_1)p_1(r_2)(r_1-r_2)(a_1-a_2)p_1(a_1)p_1(a_2)\d a_1 \d a_2\d r_1 \d r_2.
		\end{align}
		Plugging these limits back in \eqref{t0}, we arrive at \eqref{nibm1} (the one point density formula for $\nonintbm$) as the limit for \eqref{t0}. 
		
		\medskip
		
		\textbf{Step 2. One point convergence at $0<t<1$.} When $0 < t < 1$, with the Karlin-Mcgregor formula, we similarly obtain 	
		\begin{align}\label{t2}
			&\Pr(Z_{nt}^{(1)}\in \d y_1, Z_{nt}^{(2)}\in \d y_2 \mid \Lambda_n) = \tau_n \cdot f_{n,t}(y_1,y_2)\d y_1 \d y_2
		\end{align} 
		where $\tau_n$ is defined in \eqref{taun} and
		\begin{equation}
			\label{a02}
			\begin{aligned}
				f_{n,t}(y_1,y_2) & =\int_{r_1 > r_2}\int_{a_1> a_2}p_1(a_1)p_1(a_2)\left[\det(p_{\lfloor nt \rfloor -1}(a_i -y_j\sqrt{n}))_{i, j =1}^2\right. \\ & \hspace{3cm} \left. n\cdot \det(p_{n-\lfloor nt \rfloor} (\sqrt{n}y_i - \sqrt{n}r_j))_{i, j = 1}^2\right] \d a_1 \d a_2 \d r_1 \d r_2.
			\end{aligned}
		\end{equation}
		One  can check that as $n\to \infty$, we have 
		\begin{align*}
			n^{3/2}\det(p_{\lfloor nt \rfloor -1}(a_i -y_j\sqrt{n}))_{i,j=1}^2 & \to \tfrac1t p_t(y_1)p_t(y_2)(a_1-a_2)(y_1-y_2), \\  n\cdot \det(p_{n-\lfloor nt \rfloor} (\sqrt{n}y_i - \sqrt{n}r_j))_{i, j = 1}^2 & \to \det(p_{1-t} (y_i - r_j))_{i, j = 1}^2.
		\end{align*}
		One can provide uniformly integrable bound for the integrand in $f_{n,t}(y_1,y_2)$ in a similar fashion. Thus by dominated convergence theorem,
		\begin{align*}
			n^{3/2}f_{n,t}(y_1,y_2) & \to \tfrac1tp_t(y_1)p_t(y_2)(y_1-y_2)\int_{a_1> a_2}p_1(a_1)p_1(a_2)(a_1-a_2) \d a_1 \d a_2 \\ & \hspace{3cm} \int_{r_1 > r_2}  \det(p_{1-t} (y_i - r_j))_{i, j = 1}^2 \d r_1 \d r_2. 
		\end{align*}
		Using \eqref{taulim} we get that $\tau_n \cdot f_{n,t}(y_1,y_2)$ converges to \eqref{nibm2}, the one point density formula for $\nonintbm$. 
		
		\medskip
		
		\noindent\textbf{Step 3. Transition density convergence.}	For the transition densities, let $0 < t_1 < t_2 < 1,$ and fix $x_1>x_2$. Another application of Karlin-McGregor formula tells us
		
		\begin{equation}
			\label{trst1}
			\begin{aligned}
				&\Pr(Z_{nt_2}^{(1)}\in \d y_1, Z_{nt_2}^{(2)}\in \d y_2 \mid Z_{nt_1}^{(1)}=x_1, Z_{nt_1}^{(2)}=x_2) \\ & = n\det(p_{\lfloor nt_2 \rfloor-\lfloor nt_1 \rfloor}(\sqrt{n}y_i - \sqrt{n}x_j))_{i, j = 1}^2 \\ & \hspace{2cm}\cdot \frac{\int\limits_{r_1 > r_2}\det(p_{n-\lfloor nt_2 \rfloor}(\sqrt{n}y_i - \sqrt{n}r_j))_{i, j = 1}^2 \d r_1 \d r_2 \d y_1 \d y_2}{\int\limits_{r_1 > r_2}\det(p_{n-\lfloor nt_1 \rfloor }(\sqrt{n}x_i - \sqrt{n}r_j))_{i, j = 1}^2 \d r_1 \d r_2 }.
			\end{aligned}
		\end{equation}
		One can check as $n\to \infty$ 
		\begin{align*}
			\text{ r.h.s of  }\eqref{trst1} \to  \frac{\det(p_{t_2 - t_1}(y_i - x_j))_{i, j = 1}^2\int_{r_1 > r_2}\det(p_{1-t_2}(y_i - r_j))_{i, j = 1}^2 \d r_1 \d r_2 \d y_1 \d y_2}{\int_{r_1 > r_2}\det(p_{1-t_1}(x_i - r_j))_{i, j = 1}^2 \d r_1 \d r_2 }
		\end{align*}
		which is same as transition densities for $\nonintbm$ as shown in \eqref{nibm3}. This proves finite dimensional convergence.

		\medskip
		
		\textbf{Step 4. Tightness.} To show tightness, by Kolmogorov tightness criterion, it suffices to show there exist $K>0$ and $n_0\in \N$ such that for all $n\ge n_0$
		\begin{align}\label{tgt}
			\Ex\left[|Y_{n,i}(t)-Y_{n,i}(s)|^{K}\mid \Lambda_n\right] \le \Con_{K,n_0} \cdot (t-s)^2
		\end{align}
		holds for all $0\le s<t\le 1.$ \\
		
		Recall that $\Pr(\Lambda_n)\ge \frac{\Con}{\sqrt{n}}$. For $t-s \le \frac1n$ with $K\ge5$ we have 
		\begin{align*}
			\Ex\left[|Y_{n,i}(t)-Y_{n,i}(s)|^{K}\mid \Lambda_n\right] & \le \Con \cdot \sqrt{n}\Ex\left[|Y_{n,i}(t)-Y_{n,i}(s)|^{K}\right] \\ & \le \Con \cdot \sqrt{n}\frac{(nt-ns)^K}{n^{K/2}}\Ex[|X_1^1|^K] \le \Con n^{\frac{1-K}{2}}(nt-ns)^2 \le \Con_K (t-s)^2.
		\end{align*}
		Thus we may assume $t-s \ge 1/n$. Then it is enough to show \eqref{tgt} for $Z_{nt}^{(i)}$ (defined in \eqref{t0}) instead. Note that if $t-s \in [n^{-1},{n^{-1/4}}]$, we may take $K$ large enough so $\frac14(K-4) \ge 1$. Then we have 
		\begin{align*}
			\Ex\left[|Z_{nt}^{(i)}-Z_{ns}^{(i)}|^{K}\mid \Lambda_n\right] & \le \Con \cdot \sqrt{n}\Ex\left[|Z_{nt}^{(i)}-Z_{ns}^{(i)}|^{K}\right] \\ & \le \Con \cdot \sqrt{n}(t-s)^{K/2} \le \Con \cdot n^{1/2-(K-4)/8}(t-s)^2 
		\end{align*}
		where in the last line we used the fact $(t-s)^{(K-4)/2} \le n^{-(K-4)/8}$. As $\frac14(K-4) \ge 1$, we have $\Ex\left[|Z_{nt}^{(i)}-Z_{ns}^{(i)}|^{K}\mid \Lambda_n\right] \le \Con(t-s)^2$ in this case. So, we are left with the case $t-s \ge n^{-1/4}$. \\
		
		Let us assume $t=0$, $s\ge n^{-\frac14}$. As $ns \ge n^{3/4} \to \infty$, we will no longer make the distinction between $ns$ and $\lfloor ns \rfloor$ in our computations. We use the pdf formula from \eqref{t2} and \eqref{a02} to get
		\begin{equation}
			\label{a03}
			\begin{aligned}
				\Ex[|Z_{ns}^{(i)}|^5] & \le \tau_n\int_{y_1>y_2}|y_i|^5\int_{r_1 > r_2}\int_{a_1>a_2}p_1(a_1)p_1(a_2)\det(p_{ns  -1}(a_i -y_j\sqrt{n}))_{i,j=1}^2 \\ & \hspace{3cm} \left. n\cdot \det(p_{n- ns} (\sqrt{n}y_i - \sqrt{n}r_j))_{i, j = 1}^2\right] \d a_1 \d a_2 \d r_1 \d r_2 \d y_1 \d y_2. 
			\end{aligned}
		\end{equation}
		For the last determinant we may use
		\begin{align*}
			& n\cdot \det(p_{n- ns} (\sqrt{n}y_i - \sqrt{n}r_j))_{i, j = 1}^2 \d r_1 \d r_2 \\ & \le n\cdot p_{n-ns} (\sqrt{n}y_1 - \sqrt{n}r_1) p_{n- ns} (\sqrt{n}y_2 - \sqrt{n}r_2) \d r_1 \d r_2
		\end{align*}
		which integrates to $1$ irrespective of the value of $y_1,y_2$. Thus
		\begin{align}
			\label{a04}
			\mbox{r.h.s.~of \eqref{a03}} & \le \tau_n\int_{y_1>y_2}\hspace{-0.2cm}|y_i|^5 \int_{a_1,a_2}p_1(a_1)p_1(a_2)\det(p_{ ns -1}(a_i -y_j\sqrt{n}))_{i,j=1}^2 \d a_1 \d a_2 \d y_1 \d y_2.
		\end{align}
		Making the change of variable $y_i=\sqrt{s} z_i$ and setting $m=ns$, we have $$\mbox{r.h.s.~of \eqref{a04}}
		\le  \tau_n \cdot s^{\frac{5}2+1}\mathcal{I}_m,$$
		where
		\begin{align*}
			\mathcal{I}_m:=\int_{z_1>z_2}|z_i|^5\int_{a_1>a_2}p_1(a_1)p_1(a_2)\det(p_{m-1}(a_i -z_j\sqrt{m}))_{i,j=1}^2 \d a_1 \d a_2 \d z_1 \d z_2.
		\end{align*}
		We claim that $\frac{(m-1)^2}{\sqrt{m}}{\mathcal{I}_m} \le \Con$ for some universal constant $\Con>0$. Clearly this integral is finite for each $m$. And by exact same approach in \textbf{Step 1}, one can show as $m\to \infty$, $$
		\frac{(m-1)^2}{\sqrt{m}}\mathcal{I}_m:=\int_{z_1>z_2}|z_i|^5\int_{a_1>a_2}p_1(z_1)p_1(z_2)p_1(a_1)p_1(a_2)(a_1-a_2)(z_1-z_2) \d a_1 \d a_2 \d z_1 \d z_2.$$  Thus, $\frac{(m-1)^2}{\sqrt{m}}\mathcal{I} \le \Con$ for all $m\ge 1$. Thus following \eqref{a03}, \eqref{a04}, in view of the above estimate we get
		\begin{align*}
			\Ex[|Z_{ns}^{(i)}|^5] \le \Con \tau_n  \frac{\sqrt{m}}{(m-1)^2} s^{\frac{5}{2}+1}.
		\end{align*}
		However, by \textbf{Step 1}, $n^{3/2}\tau_n^{-1}$ converges to a finite positive constant. As $m=ns$, we thus get that the above term is at most $\Con \cdot s^2$. The case $t\neq 0$ can be checked similarly using the formulas from \eqref{t2} and \eqref{a02} as well as transition densities formula \eqref{trst1}. This completes the proof.
	\end{proof}

	\bibliographystyle{alphaabbr}		
	\bibliography{paths}
\end{document}